\documentclass[mnsc]{informs_nojournal}
\RequirePackage{tgtermes}
\RequirePackage{newtxtext}
\RequirePackage{newtxmath}
\RequirePackage{bm}
\RequirePackage{endnotes}

\OneAndAHalfSpacedXII 

\usepackage{algorithm}
\usepackage{algpseudocode}
\usepackage{float}
\usepackage{tikz}
\usepackage{bbm}
\usepackage{graphicx}
\usepackage{subcaption}
\usepackage{verbatim}
\usepackage{hyperref}
\usepackage{booktabs}

\usepackage{natbib}
 \bibpunct[, ]{(}{)}{,}{a}{}{,}%
 \def\bibfont{\small}%
\EquationsNumberedThrough    

\TheoremsNumberedThrough     
\ECRepeatTheorems  %

\MANUSCRIPTNO{MOOR-0001-2024.00}

\newcommand{\EE}{\mathbb{E}}
\newcommand{\bP}{\mathbb{P}}
\newcommand{\cP}{\mathcal{P}}
\newcommand{\one}{\mathbf{1}}
\newcommand{\bbR}{\mathbb{R}}
\newcommand{\KL}{\mathrm{KL}}

\newcommand{\Var}{\mathrm{Var}}

\newcommand{\wdfeo}{W^{\textnormal{ER-O}}_{\textnormal{DF}}}

\newcommand{\tv}{d_{\mathrm{TV}}}
\newcommand{\lspo}{\ell_{\textnormal{decision}}}

\newcommand{\ff}{\mathfrak{f}}

\newcommand{\wspo}{W_{\textnormal{decision}}}
\newcommand{\wdfo}{W^{\textnormal{O}}_{\textnormal{DF}}}

\newcommand{\rwc}{R^{\textnormal{WC}}}
\newcommand{\rwcer}{R^{\textnormal{WC}}_{\textnormal{ER}}}
\newcommand{\diam}{\mathrm{diam}}

\begin{document}


\RUNAUTHOR{ Mo Liu and Suhan Liu}

\RUNTITLE{Decision-Focused Optimal Transport}

\TITLE{Decision-Focused Optimal Transport}

\ARTICLEAUTHORS{%

\AUTHOR{Mo Liu}
\AFF{Department of Statistics and Operations Research,
University of North Carolina at Chapel Hill, \EMAIL{mo\_liu@unc.edu}}

\AUTHOR{Suhan Liu}
\AFF{Department of Statistics and Operations Research,
University of North Carolina at Chapel Hill, \EMAIL{suhanliu@unc.edu}}
} 

\ABSTRACT{%
Operational decisions are often optimized for one population and applied to another. We ask how much regret is unavoidable when decisions are retained after a population shift, since changes in parameters need not change optimal decisions. For linear programs with random objective coefficients, we propose the optimistic decision-focused (DF) divergence: the minimum expected regret over all couplings of the populations used for decision making and evaluation. We make three methodological contributions. First, the divergence depends on the population for decision making only through its distribution of optimal decisions over the extreme points of the feasible region. This reduces computation to a semi-discrete quadratic optimal transport problem between decision and coefficient vectors. Second, we establish a finite-sample error bound for the plug-in estimator at a parametric rate whose exponent is independent of the coefficient dimension. Third, we establish Lipschitz stability with respect to classical metrics, characterize DF divergence balls as the largest sets of shifts compatible with a regret tolerance, and derive intermediate regret values via entropy regularization. Using patient data and a stylized decision model, we show that entropy-regularized DF couplings calibrated to observed decision persistence yield regret estimates closer to patient-tracked regret than Wasserstein or independent couplings.

}%





\maketitle

\section{Introduction}\label{sec:Intro}

Quantifying the discrepancy between two probability distributions is a  fundamental problem in statistics and machine learning. A well-designed distance or divergence is a critical tool for clustering, interpolation,  kernel weighting, anomaly detection, and many other tasks. Classical  choices, such as the total variation (TV) distance, the Kullback--Leibler  (KL) divergence, and the Wasserstein distance, quantify discrepancies  in probability mass or in the geometry of the underlying parameter space. 
In many operations research applications, however, uncertain quantities enter a downstream optimization problem as objective coefficients or constraint parameters. From a decision-focused perspective, the goal of comparing two distributions is not merely to quantify how much the distributions differ, but rather to understand how their differences affect the optimal decisions and the resulting objective values.

One setting where this distinction is particularly important is decision making under population shifts. Many operational decisions are optimized for one population and then applied to another. A retailer sets assortments for its current customers and then serves the next season's customers, and an insurer designs a plan menu using last year's enrollees. In each case, a manager needs to know how much is lost by keeping the original decisions as the population shifts, and in particular whether re-measuring the population and re-optimizing the decisions is worth the associated cost. We therefore ask the following question.

\begin{quote}
\emph{Given a population for which decisions have been made and a population to which it shifts, what is the minimum unavoidable decision regret when the decisions are kept unchanged?}
\end{quote}

Classical discrepancies between probability distributions are \emph{decision-blind} in the sense that they do not account for the downstream optimization problem. In particular, they miss two features that determine the regret of keeping the original decisions. 
The first feature is \emph{decision relevance}. The parameters of a decision problem, such as costs, demands, or clinical indicators, can change substantially while the original decisions remain optimal, in which case keeping the decisions incurs no regret. Conversely, even a small shift in the parameter space can change the optimal decisions, and the resulting regret depends on the objective loss of retaining the original decisions. The second feature is \emph{coupling}. The regret of keeping an individual's decision depends on how that individual's parameters evolve, so the expected regret depends on the joint distribution of the individual's parameters before and after the shift, rather than solely on the two marginal distributions. The TV distance and the KL divergence compare the marginal distributions without modeling these individual-level transitions, while the Wasserstein distance selects a coupling based on a geometric transport cost rather than on decision regret.

\subsection{Decisions under population change}
\label{sec:decision_focused_setting}

We study the question above for stochastic linear optimization problems with random objective coefficients. Each individual in a population, such as a patient or a customer, is described by a vector of objective coefficients and receives the feasible decision that is optimal for their own coefficient vector. We call the population for which the decisions are made the \emph{population for decision making}, and the population on which the decisions are later evaluated the \emph{population for evaluation}. An individual whose coefficient vector has changed, but whose decision has not, incurs a regret equal to the Smart Predict-then-Optimize (SPO) loss of \citet{elmachtoub2022smart}, i.e., the excess cost of their old decision over the decision that is optimal for their new coefficient vector. The expected regret depends on the coupling of the two populations, i.e., the joint distribution of an individual's coefficient vectors before and after the shift. When decisions are made, this coupling is typically unknown, because the population for evaluation is observed through historical cohorts, a comparable region, or a forecast, rather than through follow-up measurements of the same individuals.

The standard way to evaluate decisions under a new distribution draws the old and the new coefficient vectors independently, which yields the \emph{independent-coupling regret}. Independence ignores \emph{persistence}. A patient whose condition is mild today is unlikely to be among the most severe patients next year, even if severe cases are common in next year's population. We show that the regret under the true coupling equals the independent-coupling regret plus a persistence covariance, so the independent-coupling regret overstates the regret whenever individuals' earlier decisions tend to remain relatively good for those same individuals after the shift (Proposition~\ref{prop:cov_decomp} in Appendix~\ref{sec:persistence}). The independent-coupling regret is positive even when the population does not change at all.

To answer the question above, we define the \emph{optimistic decision-focused (DF) divergence} as the minimum expected regret over all couplings of the two populations, i.e., an optimal transport problem whose transport cost is the SPO loss. It is the regret that \emph{cannot be avoided}: no matter how individuals moved, keeping the decisions costs at least this much, so a value above a tolerance warrants re-optimization. Because the optimistic DF divergence is the most favorable case, we also study an entropy-regularized version that yields intermediate regret values between the divergence and the independent-coupling regret.

Besides its decision relevance, the optimistic DF divergence admits a smaller transport formulation and an estimation exponent independent of dimension. Table~\ref{tab:intro_comparison} compares it with the quadratic Wasserstein cost for bounded, continuously distributed coefficients. 

\begin{table}[ht]
\centering
\begin{tabular}{lcc}
\hline
 & Quadratic Wasserstein cost $W_2^2$ & Optimistic DF divergence \\
\hline
Computational formulation & Infinite-dimensional & Semi-discrete \\
 & optimal transport & optimal transport \\
Empirical estimation error bound & $O(n^{-2/d})$ for $d>4$ & $O(n^{-1/2})$ \\
Entropy-regularization error bound
& $O(\varepsilon\log(1/\varepsilon))$
& $\varepsilon\log K_\mu$ \\
\hline
\end{tabular}
\caption{Comparison of $W_2^2(\mu,\nu)$ on the original coefficient distributions and the optimistic DF divergence for \textbf{continuous} distributions with bounded support. Here, $n$ is the sample size per marginal and $d$ the coefficient dimension. The DF rate follows from the finite-support reduction; its constants depend on the number of decisions used and a uniform regret bound, which may depend on $d$. The final row bounds the difference between the regularized
optimal objective, including the KL penalty, and the
unregularized value, where $\varepsilon>0$ is the KL regularization parameter, and
$K_\mu$ is the number of decisions used by $\mu$. Sections~\ref{sec:calculation}, \ref{sec:sample_complexity}, and~\ref{sec:property} establish the three rows for the optimistic DF divergence.}
\label{tab:intro_comparison}
\end{table}

\paragraph{Contributions.}
We make one conceptual and three methodological contributions.
\begin{itemize}
    \item \textbf{A decision-focused divergence for population change (Section~\ref{sec:Method}).} We define the optimistic DF divergence between a population for decision making and a population for evaluation and characterize it as the unavoidable regret of not updating decisions. Intuitively, our definition replaces the quadratic transport cost in the classical optimal transport formulation with the decision loss. The DF divergence is nonnegative, vanishes exactly when the two populations induce the same distribution of optimal decisions (under a no-ties condition), and is asymmetric. We also establish its Lipschitz stability under perturbations of either population. To the best of our knowledge, we are the first to exploit the structure of downstream linear optimization problems to improve both the computational and statistical efficiency of evaluating decision-focused divergences.

    \item \textbf{Computation through a push-forward reduction and a quadratic transport formulation (Section~\ref{sec:calculation}).} Computing an optimal transport plan between two continuously distributed populations is typically challenging. However, because the decision loss depends on an individual's original coefficient vector only through the corresponding optimal decision, the DF divergence depends on the population for decision making only through its push-forward onto the extreme points of the feasible region. This property reduces the computation to a semi-discrete transport problem. Moreover, we establish an equivalent quadratic transport formulation in which the best-case alignment \emph{maximizes} the expected squared Euclidean distance between decision vectors and coefficient vectors. Although subtracting a decision vector from a coefficient vector appears unnatural because the two vectors have different meanings and units, we show why this reformulation is exact and invariant under rescaling that preserves the downstream objective.

    \item \textbf{A finite-sample estimation error bound with a dimension-independent rate (Section~\ref{sec:sample_complexity}).} The convergence rates of empirical optimal transport costs between continuous distributions typically deteriorate as the dimension of the underlying space increases (Table~\ref{tab:intro_comparison}). In contrast, we establish that the plug-in estimator of the optimistic DF divergence converges at the parametric square-root rate in the sample size, regardless of the dimension of the coefficient space. The constants in the estimation error bound depend on the number of distinct optimal decisions used by the population for decision making and on a uniform regret bound. This result is particularly beneficial when the coefficient space is high-dimensional but the feasible region has relatively few extreme points, or, more generally, when the population for decision making uses only a small subset of the extreme points.

    \item \textbf{Comparison with classical metrics, DF divergence balls, and entropy regularization (Section~\ref{sec:property}).} We relate the optimistic DF divergence to the Wasserstein distance, the total variation distance, and the KL divergence. We also show that the ball of the optimistic DF divergence is the largest set of population shifts for which keeping the original decisions can incur a regret of at most a given tolerance. Furthermore, we study the entropy-regularized DF divergence, which admits the same push-forward reduction and is bounded below by the optimistic DF divergence and above by the independent-coupling regret. We establish that the regret of the entropy-regularized coupling increases monotonically from the optimistic DF divergence to the independent-coupling regret as the regularization parameter increases, while the dependence between an individual's coefficient vectors before and after the population shift decreases. Thus, by tuning the regularization parameter, we can construct transport plans that reflect different degrees of decision persistence and yield intermediate regret values between the optimistic and independent-coupling benchmarks.
\end{itemize}

As an extension, Section~\ref{sec:pessimistic} introduces the worst-case regret, the maximum regret over all couplings, which bounds the regret from above; Appendix~\ref{append:pess} collects its properties. Appendix~\ref{sec:interpolant} records preliminary results on decision-focused interpolation between populations.

\paragraph{Organization of the paper.}
Section~\ref{sec:Method} introduces the setting, the optimistic DF divergence, its interpretation, and its stability properties. Section~\ref{sec:calculation} develops the push-forward reduction, the quadratic transport formulation, and the resulting algorithm. Section~\ref{sec:sample_complexity} establishes the finite-sample estimation error bound, and Section~\ref{sec:property} studies comparison inequalities, DF divergence balls, and entropy regularization. Section~\ref{sec:pessimistic} presents the worst-case extension, Section~\ref{sec:numerical} reports numerical experiments, and Section~\ref{sec:Conclusion} concludes.

As this paper introduces the optimistic DF divergence, we primarily focus on developing computational approaches, establishing estimation error bounds, and characterizing its fundamental properties. We leave the systematic investigation of its applications for future research. In the numerical study, we use a newsvendor problem to illustrate how the DF divergence captures discrepancies between populations that classical distributional measures fail to capture. We also use a real-world dataset of patients with Parkinson’s disease to illustrate how different coupling assumptions affect regret evaluation. The optimistic DF divergence provides a lower bound on the patient-tracked regret, while its entropy-regularized version, calibrated to the observed proportion of unchanged optimal decisions, yields regret estimates closer to the patient-tracked regret than those obtained through Wasserstein and independent couplings.

\subsection{Literature review}

We review three streams of related work: decision-focused learning and decision-aware distances, problem-based distances in stochastic programming, and optimal transport.

\paragraph{Decision-focused learning and decision-aware distances.}
Decision-focused learning (also referred to as decision-aware, task-based, or end-to-end learning) is motivated by the observation that improved prediction accuracy does not necessarily lead to improved downstream decision quality in stochastic optimization; see \citet{mandi2024decision,sadana2025survey} for comprehensive surveys and \citet{liu2026decision} for a tutorial. We highlight two subsets of related papers. The first subset focuses on estimating the uncertainty distribution. \citet{bertsimas2020predictive,elmachtoub2023estimate,elmachtoub2025dissecting} study sample average approximation (SAA) in the decision-focused setting. Robust formulations that jointly account for estimation and optimization have been studied in, e.g., \citet{zhu2022joint}. \citet{iceo} consider estimating right-hand-side uncertainty in constraints by incorporating decision loss, and \citet{cristian2025efficient} address computational challenges by proposing training methods that approximate the optimization oracle. The second subset focuses on constructing pointwise predictions via decision-aligned losses. \citet{elmachtoub2022smart} and \citet{huang2024decision} propose surrogate loss functions and establish consistency under various conditions. \citet{homem2024forecasting,er2025decision} characterize the existence of decision-unbiased predictions. Statistical guarantees and comparisons with decision-blind approaches are provided in \citet{elmachtoub2023estimate,el2022generalization,hu2022fast,liu2023active}. Beyond contextual cost prediction, data-driven models of customer choice have also been incorporated into downstream optimization problems, such as assortment optimization under nonparametric decision forest models \citep{akchen2026assortment}. This literature primarily studies learning under a fixed distribution.

A related stream of research studies how to make and update decisions when the underlying environment changes. \citet{ahn2026online} study how potentially biased offline data can facilitate online decision making in inventory and pricing problems and establish regret bounds that account for the size and bias of the offline dataset. In dynamic resource allocation problems, \citet{brown2025fluid} establish performance guarantees for fluid-based policies, including policies based on reoptimization. While these studies focus on designing policies that account for new information or changes in the underlying system, we study the regret of retaining previously optimized decisions under population shifts.

Closer to our work, several papers compare distributions or datasets with the downstream decision problem in mind. \citet{alvarez2020geometric} define an optimal transport distance between labeled datasets, and \citet{rodriguez2025right} extend it to predict-then-optimize tasks by adding a decision-quality term to a weighted ground cost of features and labels, using the resulting metric to predict transferability across tasks. \citet{ren2024decision} evaluate worst-case distribution shift through its effect on decision quality. Our divergence instead uses the decision regret itself as the transport cost, separates the population used for decision making from the population used for evaluation, and studies the role of the coupling between the two, which leads to the reduction, estimation, and stability results above.

\paragraph{Problem-based distances in stochastic programming.}
Stability analysis of stochastic programs measures the distance between distributions through classes of functions generated by the optimization problem \citep{rachev2002quantitative,romisch2003stability}, and the nested distance of \citet{pflug2012distance} plays this role for multistage problems. Scenario reduction and clustering methods use problem-based dissimilarities: \citet{hewitt2022decision} and \citet{keutchayan2023problem} cluster scenarios by the opportunity cost of applying the decision optimal for one scenario when another scenario occurs, which coincides with the SPO loss for linear objectives, and \citet{henrion2022problem} and \citet{bertsimas2023optimization} construct reduced scenario sets that account for the objective and the problem structure. In particular, \citet{bertsimas2023optimization} introduce a problem-dependent divergence based on an optimal transport formulation with decision regret as the transport cost. For discrete distributions and linear programs with random objective coefficients, our optimistic DF divergence coincides with the square of their asymmetric problem-dependent divergence, under the same selection of optimal decisions. 
While their definition accommodates more general downstream optimization problems, we focus on linear optimization over a compact polyhedral feasible region and extend the analysis to continuous distributions. We exploit the structure of linear optimization to establish an exact semi-discrete transport reduction, an equivalent quadratic transport formulation, and a finite-sample estimation error bound with a dimension-independent convergence exponent, none of which are established in their work. Furthermore, these scenario reduction methods target a single here-and-now decision and approximate one distribution by another with smaller support. We instead consider personalized decisions applied to a changing population, where the coupling between the two populations determines the regret.

Decision theory offers a related notion: every loss function induces a divergence between distributions, namely the excess expected loss of acting optimally for the wrong distribution \citep{grunwald2004game,dawid2007geometry,gneiting2007strictly}. For a single linear decision, this divergence depends only on the means of the two distributions. Our optimistic DF divergence, by contrast, vanishes exactly when the two populations induce the same distribution of optimal decisions under a no-ties condition, much as the divergence of a proper scoring rule vanishes when the optimal acts coincide.

\paragraph{Optimal transport.}
In the OR community, KL-based and Wasserstein-based uncertainty sets have been used in distributionally robust optimization; see, e.g., \citet{bennouna2022holistic}. More generally, \citet{bental2010soft} develop a soft robust optimization framework that allows the decision maker to specify different levels of protection against distributional uncertainty. The Wasserstein distance has been widely used because, for many objective functions, the expected objective value is Lipschitz with respect to the Wasserstein distance between the underlying distributions \citep{blanchet2019quantifying}; see \citet{kuhn2019wasserstein} for an overview and \citet{blanchet2022optimal,ho2023strong,gao2024wasserstein,gao2023finite} for guarantees and finite-sample analysis. Beyond classical Wasserstein ambiguity sets, \citet{wang2026sinkhorn} study distributionally robust optimization with ambiguity sets constructed using entropy-regularized optimal transport and develop dual reformulations and computational algorithms with complexity guarantees.

Wasserstein ambiguity sets have also been used when the available data do not come from the distribution of interest. \citet{taskesen2021sequential} construct robust least squares estimators from divergence regions around a source distribution and a target distribution, \citet{selvi2025distributionally} intersect Wasserstein balls, \citet{rychener2024wasserstein} study heterogeneous data sources, and \citet{keehan2025dont} center a Wasserstein ball at a weighted empirical distribution to account for a distribution that drifts over time. While these approaches account for the downstream objective when making robust decisions, the discrepancies used to construct their uncertainty sets are generally defined in the parameter space without explicitly accounting for decision regret. Our DF divergence instead incorporates the downstream decision problem into the transport cost itself.

Our quadratic transport formulation builds on the classical equivalence between quadratic optimal transport and the maximization of the expected inner product \citep{brenier1991polar,villani2009optimal}, and our computational approach uses semi-discrete optimal transport \citep{genevay2016stochastic,kitagawa2019convergence,peyre2019computational} and entropic regularization \citep{cuturi2013sinkhorn}. On the statistical side, the plug-in Wasserstein distance suffers from the curse of dimensionality \citep{ChewiNilesWeedRigollet2025StatOT}, entropic optimal transport admits parametric rates with dimension-dependent constants \citep{genevay2019sample,mena2019statistical}, and \citet{hundrieser2024empirical} show that the rate of the empirical optimal transport cost is governed by the less complex of the two measures and attains the parametric rate in the semi-discrete case. Our estimation result is an instance of this phenomenon, made explicit for the DF divergence. Finally, extremal values of an expectation over all couplings with given marginals have been used to bound treatment effects \citep{fan2010sharp,galichon2016optimal} and aggregate risks under dependence uncertainty \citep{embrechts2013model}; the optimistic DF divergence and the worst-case regret are such bounds for the regret of not updating decisions.

\section{The Decision-Focused Divergence}\label{sec:Method}

This section introduces the setting, defines the optimistic DF divergence, and establishes its stability and basic properties.
Let $\mu$ and $\nu$ be probability measures on $\bbR^d$, and let $c:\bbR^d\times\bbR^d\to[0,\infty)$ be a measurable cost function. A \emph{coupling} of $\mu$ and $\nu$ is a probability measure $\gamma$ on $\bbR^d\times\bbR^d$ whose marginals are $\mu$ and $\nu$; we denote the set of all such couplings by $\Gamma(\mu,\nu)$. In particular, $\gamma\in\Gamma(\mu,\nu)$ if and only if, for every bounded measurable function $f:\bbR^d\to\bbR$,
\[
\int f(x)\,\gamma(dx,dy)=\int f\,d\mu
\quad\text{and}\quad
\int f(y)\,\gamma(dx,dy)=\int f\,d\nu .
\]
The (Kantorovich) optimal transport problem seeks a coupling that minimizes the expected transport cost,
\begin{equation}\label{equ:kantor}
\inf_{\gamma\in\Gamma(\mu,\nu)} \int_{\bbR^d\times\bbR^d} c(x,y)\,\gamma(dx,dy).
\end{equation}
Any $\gamma\in\Gamma(\mu,\nu)$ can be interpreted as the joint law of a pair $(X,Y)$ with marginals $X\sim\mu$ and $Y\sim\nu$. A classical choice is the power cost $c(x,y)=\|x-y\|^p$ with $p\ge1$, which leads to the $p$-Wasserstein distance
\begin{equation}\label{eq:wp_def}
    W_p(\mu,\nu)
    :=\left(\inf_{\gamma\in\Gamma(\mu,\nu)}\int_{\bbR^d\times\bbR^d}\|x-y\|^p\,\gamma(dx,dy)\right)^{1/p}.
\end{equation}
As a special case, when $\mu$ and $\nu$ have finite second moments, the corresponding optimal transport value defines the quadratic Wasserstein distance
\begin{equation}\label{WD2}
    W_2(\mu,\nu)
    =\inf_{\gamma\in\Gamma(\mu,\nu)}
    \left(\int_{\bbR^d\times\bbR^d}\|x-y\|_2^2\,\gamma(dx,dy)\right)^{1/2}.
\end{equation}

We also use two classical discrepancies that compare probability mass at the same location. For probability measures $P$ and $Q$ on a common measurable space $(\Omega,\mathcal{F})$, the total variation distance is $\tv(P,Q):=\sup_{A\in\mathcal{F}}|P(A)-Q(A)|$, and the Kullback--Leibler (KL) divergence is $\KL(P\|Q):=\int_\Omega\log(dP/dQ)\,dP$ if $P$ is absolutely continuous with respect to $Q$ and $\KL(P\|Q):=+\infty$ otherwise.

\subsection{Downstream problem and regret}\label{sec:setup}

Let $S\subset\bbR^d$ be a nonempty compact polyhedron (a polytope) with set of extreme points $\mathfrak{S}=\{w_1,\ldots,w_{|\mathfrak{S}|}\}$, and let $D_W:=\sup\{\|w-w'\|_2:\ w,w'\in S\}$ denote its diameter. For a coefficient (cost) vector $x\in\bbR^d$, the downstream problem is $\min_{w\in S}w^\top x$, and the optimization oracle $w^*(\cdot):\bbR^d\to\bbR^d$ returns
\begin{equation}\label{eq:downstream_oracle}
    w^*(x)\in\arg\min_{w\in S} w^\top x.
\end{equation}
Throughout, we fix a measurable selection $w^*(\cdot)$ that takes values in $\mathfrak{S}$ and breaks ties by a fixed priority order over $\mathfrak{S}$. Because the set of minimizers in \eqref{eq:downstream_oracle} is invariant to multiplying $x$ by a positive scalar, this selection satisfies $w^*(\kappa x)=w^*(x)$ for every $\kappa>0$. The \emph{decision regions}
\[
D_i:=\{x\in\bbR^d:\ w^*(x)=w_i\},\qquad i=1,\ldots,|\mathfrak{S}|,
\]
form a partition of $\bbR^d$; up to its boundary, $D_i$ is the normal cone of $S$ at $w_i$. We write $D_i^\circ\subseteq D_i$ for the set of $x$ at which $w_i$ is the \emph{unique} minimizer, and $\mathcal{T}:=\bbR^d\setminus\bigcup_i D_i^\circ$ for the set of coefficient vectors with multiple minimizers, which is contained in a finite union of hyperplanes and hence has Lebesgue measure zero. Let
\[
z(y):=\min_{w\in S}w^\top y=y^\top w^*(y)
\]
denote the optimal value, which is concave in $y$ as a minimum of linear functions.

To measure the decision discrepancy between two coefficient vectors, we adopt the \emph{Smart Predict-then-Optimize (SPO) loss} of \citet{elmachtoub2022smart},
\[
\lspo(x,y):=y^\top w^*(x)-y^\top w^*(y)=y^\top w^*(x)-z(y),
\]
the objective difference between the decisions $w^*(x)$ and $w^*(y)$ when the true coefficient vector is $y$. For each extreme point, we also define the \emph{reduced cost}
\begin{equation}\label{eq:reduced_cost}
c_i(y):=y^\top w_i-z(y)\ \ge 0,\qquad i=1,\ldots,|\mathfrak{S}|,
\end{equation}
the regret of decision $w_i$ at $y$. Because the decision regions partition $\bbR^d$, $\lspo(x,y)=c_i(y)$ whenever $x\in D_i$, i.e.,
\begin{equation}\label{eq:spo_ci}
\lspo(x,y)=\sum_{i=1}^{|\mathfrak{S}|}\one\{x\in D_i\}\,c_i(y).
\end{equation}
Thus, $\lspo$ depends on its first argument only through the optimal decision, a property that drives the computational and statistical results in Sections~\ref{sec:calculation} and~\ref{sec:sample_complexity}.

We use three elementary properties of the SPO loss throughout; see, e.g., \citet{liu2026decision}. For all $x,y,y'\in\bbR^d$,
\begin{align}
&0\le\lspo(x,y)\le(y-x)^\top\big(w^*(x)-w^*(y)\big)\le D_W\|x-y\|_2,\label{eq:spo_cross}\\
&\lspo(x,y)\le D_W\|y\|_2,\label{eq:spo_bound}\\
&\big|\lspo(x,y)-\lspo(x,y')\big|\le D_W\|y-y'\|_2.\label{eq:spo_lipschitz}
\end{align}
Inequality \eqref{eq:spo_cross} shows that a geometric shift of size $\|x-y\|_2$ can cause regret of at most $D_W\|x-y\|_2$, even when it moves mass across a decision boundary, whereas a large shift within a decision region causes no regret at all. The gap between decision-focused and geometric discrepancies is therefore one-directional. Inequality \eqref{eq:spo_lipschitz} shows that the SPO loss is Lipschitz in its second argument, although it is discontinuous in its first argument.

\paragraph{Populations for decision making and for evaluation.}
We call $\mu$ the \emph{population for decision making} and $\nu$ the \emph{population for evaluation}, and we assume throughout that $\int\|y\|_2\,\nu(dy)<\infty$. Both populations may have continuous support spanning multiple decision regions, and couplings may move mass across their boundaries. Each individual in the population for decision making, with coefficient vector $x$, receives the decision $w^*(x)$ that is optimal for their own coefficient vector. When the population changes, the same individual has coefficient vector $y$, keeps the decision $w^*(x)$, and incurs regret $\lspo(x,y)$. The joint law of $(X,Y)$ for a randomly chosen individual is a coupling $\gamma\in\Gamma(\mu,\nu)$. This personalized paradigm differs from choosing a single decision that minimizes $\EE_\mu[X^\top w]$. In the latter case only the mean of $\mu$ matters, and the natural decision-focused discrepancy is $\lspo(\EE_\mu X,\EE_\nu Y)$, which does not involve a coupling. When decisions are personalized, the regret depends on how individuals move, and hence on the coupling.

\subsection{Coupled regret and the optimistic DF divergence}

Given a coupling $\gamma$, the decision regret between the population for decision making $\mu$ and the population for evaluation $\nu$ is defined in Definition \ref{def:df_div}.

\begin{definition}[Coupled decision regret]\label{def:df_div}
Given $\mu,\nu$ and a coupling $\gamma\in\Gamma(\mu,\nu)$, the \emph{coupled decision regret} is
\begin{equation}\label{DWD}
    \wspo(\mu,\nu;\gamma)
    :=\int_{\bbR^d\times\bbR^d}\lspo(x,y)\,\gamma(dx,dy).
\end{equation}
\end{definition}

The coupled regret depends on the joint law $\gamma$, not only on the marginals. If $X\sim\mu$ and $Y\sim\nu$ are independent, then $\gamma=\mu\otimes\nu$, the product (independent) coupling of $\mu$ and $\nu$, and \eqref{DWD} reduces to the \emph{independent-coupling regret}
\begin{equation}\label{eq:regret_def}
    R(\mu,\nu)
    :=\EE\big[\lspo(X,Y)\big]
    =\wspo(\mu,\nu;\mu\otimes\nu).
\end{equation}
This is the standard way to evaluate decisions under a new distribution, but it assumes that an individual's later coefficient vector is unrelated to their earlier one. When individuals' coefficient vectors remain dependent over time, for example because a patient's condition changes gradually even when the population-level distribution changes considerably, the coupling should be modeled explicitly. Varying $\gamma$ over $\Gamma(\mu,\nu)$ corresponds to varying the individual-level transitions, and the most favorable transitions define our main object.

\begin{definition}[Optimistic DF divergence]\label{def:df_dist}
Given a population for decision making $\mu$ and a population for evaluation $\nu$, the \emph{optimistic decision-focused divergence} is
\begin{equation}
    \wdfo(\mu,\nu)
    :=\inf_{\gamma\in\Gamma(\mu,\nu)}\wspo(\mu,\nu;\gamma).
    \tag{Optimistic DF divergence}\label{WDFO}
\end{equation}
\end{definition}

The optimistic DF divergence is the minimum decision regret incurred when the population changes from $\mu$ to $\nu$ without updating the decisions, where the minimum is taken over all individual-level transitions consistent with the two populations. 

\begin{remark}[$\wdfo$ is not a metric]\label{rem:not_metric}
The optimistic DF divergence is generally asymmetric because evaluating decisions derived from $\mu$ under $\nu$ differs from evaluating decisions derived from $\nu$ under $\mu$. Moreover, it violates the triangle inequality, even after symmetrization via $\wdfo(\mu,\nu) + \wdfo(\nu,\mu)$. We therefore refer to $\wdfo$ as a divergence rather than a metric.
\end{remark}

Given $\mu$ and $\nu$,
$\wdfo$ represents the unavoidable regret between these two distributions. Regardless of the true coupling, the expected regret of keeping the decisions is at least $\wdfo(\mu,\nu)$. 
In $\wdfo(\mu,\nu)$, each individual in the initial population $\mu$ receives a decision optimized for their coefficient vector. As the population shifts to $\nu$, a coupling describes the probabilistic transitions between the two populations, allowing many-to-many assignments. Since the regret depends on the initial coefficient vectors only through their optimal decisions, we introduce push-forward measures to exploit this structure.

\begin{definition}[Push-forward measure]\label{def:general_pushforward}
Let $\mu$ be a probability measure on a measurable space $(\Omega_1,\mathcal{F}_1)$, let $(\Omega_2,\mathcal{F}_2)$ be another measurable space, and let $f:\Omega_1\to\Omega_2$ be measurable. The push-forward of $\mu$ by $f$, denoted $f_{\#}\mu$, is the measure on $(\Omega_2,\mathcal{F}_2)$ defined by
\[
(f_{\#}\mu)(B):=\mu\big(f^{-1}(B)\big),\qquad B\in\mathcal{F}_2.
\]
\end{definition}

Taking $f=w^*$, the push-forward $w^*_{\#}\mu=\sum_{i=1}^{|\mathfrak{S}|}\mu(D_i)\,\delta_{w_i}$ is a discrete distribution on $\mathfrak{S}$. If $X\sim\mu$, then $w^*(X)$ is a random decision with distribution $w^*_{\#}\mu$.

We first examine how the optimistic DF divergence changes when either population is perturbed, for example because the population for evaluation is estimated or forecast with error. The following theorem also establishes the existence of an optimal coupling.

\begin{theorem}[Attainment and Lipschitz stability]\label{thm:stability}
Let $\mu,\mu'$ be populations for decision making, and let $\nu,\nu'$ be populations for evaluation. If $\nu$ has a finite first moment, then the infimum in \eqref{WDFO} is attained. For any bounded set $\mathcal{Y}\subset\bbR^d$, write $r_{\mathcal{Y}}:=\sup_{y\in\mathcal{Y}}\|y\|_2$.
\begin{enumerate}
    \item[(i)] If $\nu$ and $\nu'$ have finite first moments, then $\big|\wdfo(\mu,\nu)-\wdfo(\mu,\nu')\big|\le D_W\,W_1(\nu,\nu')$.
    \item[(ii)] If $\nu$ is supported on $\mathcal{Y}$, then $\big|\wdfo(\mu,\nu)-\wdfo(\mu',\nu)\big|\le D_W\,r_{\mathcal{Y}}\,\tv\big(w^*_{\#}\mu,w^*_{\#}\mu'\big)$.
    \item[(iii)] If $\nu$ and $\nu'$ are supported on $\mathcal{Y}$, then
    \[
    \begin{aligned}
    \big|\wdfo(\mu,\nu)-\wdfo(\mu',\nu')\big|
    &\le D_W\,r_{\mathcal{Y}}\,\tv\big(w^*_{\#}\mu,w^*_{\#}\mu'\big)\\
    &\quad+D_W\,W_1(\nu,\nu').
    \end{aligned}
    \]
\end{enumerate}
\end{theorem}

Theorem~\ref{thm:stability} shows that the two populations play different roles. Perturbations of the population for evaluation affect the divergence through the geometric $W_1$ distance, because the SPO loss is Lipschitz in its second argument by \eqref{eq:spo_lipschitz}. Perturbations of the population for decision making affect it only through the distribution of optimal decisions: changes in $\mu$ that leave the decision shares unchanged have no effect. Under the bounded-support assumption in part (ii), reassigning a fraction $\eta$ of the individuals to different decisions changes the divergence by at most $D_W r_{\mathcal{Y}}\eta$. In particular, errors in forecasting the population for evaluation translate linearly into errors in the divergence. Together with Theorem~\ref{ThmSampleComplexity}, this stability supports using the divergence to decide whether re-optimization is warranted.

The attainment and stability results, together with the elementary loss bounds, yield the following basic properties.

\begin{corollary}[Basic properties]\label{prop:basic}
Let $\mu$ and $\nu$ be probability measures on $\bbR^d$ with $\int\|y\|_2\,\nu(dy)<\infty$.
\begin{enumerate}
    \item[(i)] $0\le\wdfo(\mu,\nu)\le R(\mu,\nu)\le D_W\int\|y\|_2\,\nu(dy)$, and the infimum in \eqref{WDFO} is attained.
    \item[(ii)] $\wdfo(\mu,\nu)$ depends on $\mu$ only through $w^*_{\#}\mu$: if $w^*_{\#}\mu=w^*_{\#}\mu'$, then $\wdfo(\mu,\nu)=\wdfo(\mu',\nu)$.
    \item[(iii)] If $w^*_{\#}\mu=w^*_{\#}\nu$, then $\wdfo(\mu,\nu)=0$; in particular, $\wdfo(\mu,\mu)=0$. Conversely, if $\wdfo(\mu,\nu)=0$ and $\nu(\mathcal{T})=0$, then $w^*_{\#}\mu=w^*_{\#}\nu$.
\end{enumerate}
\end{corollary}

Corollary~\ref{prop:basic} shows that $\wdfo$ is nonnegative and vanishes exactly when the two populations induce the same distribution of optimal decisions (when $\nu$ puts no mass on ties). In particular, two populations that are far apart in Wasserstein distance, total variation distance, or KL divergence have zero divergence whenever the population change leaves all decision shares unchanged. This behavior parallels the divergences induced by proper scoring rules, which vanish whenever the optimal acts under the two distributions coincide \citep{dawid2007geometry}. Part (ii) follows from Theorem~\ref{thm:stability}; Section~\ref{sec:calculation} exploits this dependence to reduce the transport problem.

\subsection{A geometric example}\label{sec:subexample}

\begin{figure}[h]
    \centering
    \includegraphics[width = 0.48\linewidth]{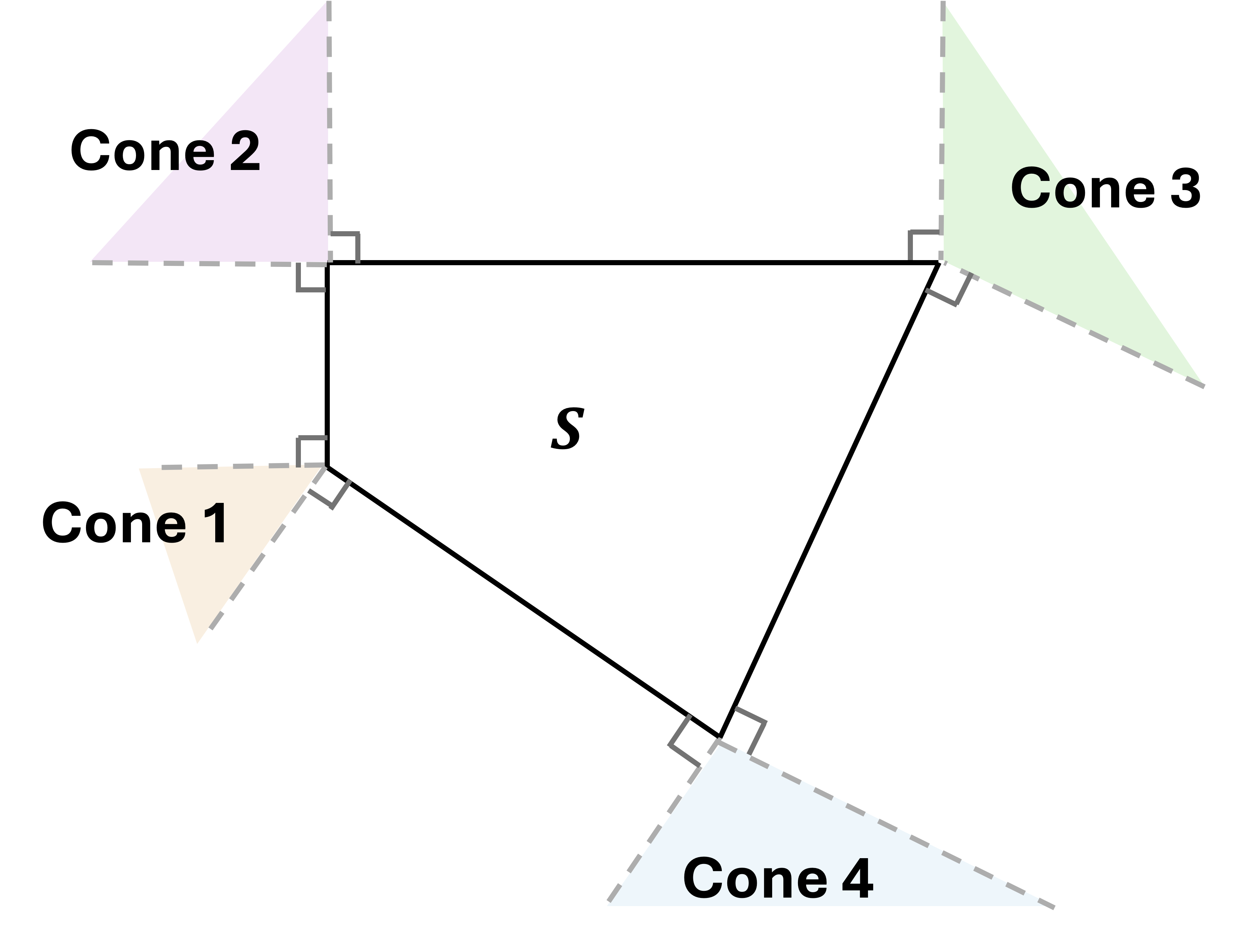}
    \includegraphics[width=0.48\linewidth]{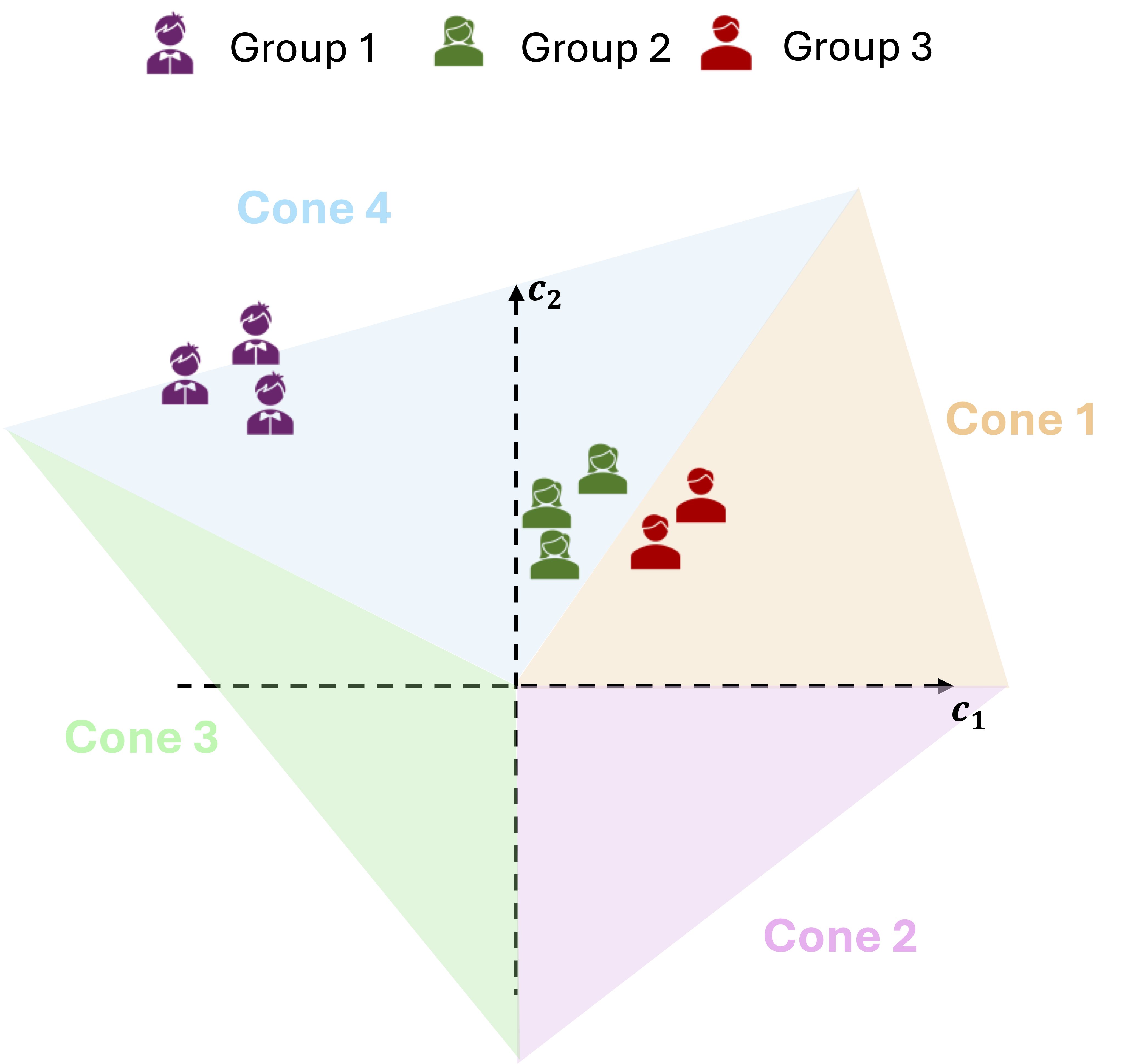}
    \caption{\textbf{Left:} a polyhedral feasible region $S$ with four extreme points and their associated normal cones.
    \textbf{Right:} the point clouds illustrate three example groups of cost vectors. The DF divergence between Groups 1 and 2 is zero, while the DF divergence between Groups 2 and 3 is positive. The discrete points are illustrative; the framework allows \textbf{continuous} distributions whose support spans several decision regions and couplings that move mass \textbf{across decision boundaries}.}
    \label{fig:intro_customer}
\end{figure}

We illustrate the geometry underlying the optimistic DF divergence in a two-dimensional setting ($d=2$). Consider a polyhedral feasible region $S\subset\bbR^2$ with four extreme points, as shown in Fig.~\ref{fig:intro_customer} (left). For each extreme point $w_k$ of $S$, the set of cost vectors for which $w_k$ is an optimizer of \eqref{eq:downstream_oracle} is a polyhedral cone (the normal cone of $S$ at $w_k$), which we label \textit{Cone $k$}. These cones partition the cost space up to boundaries where multiple extreme points are optimal, as depicted in Fig.~\ref{fig:intro_customer} (right), and the oracle map is constant on each cone.

This cone structure makes the SPO loss sensitive to whether $x$ and $y$ fall in the same cone. If $x$ and $y$ belong to the same cone, then $w^*(x)=w^*(y)$ and hence $\lspo(x,y)=0$. Consequently, the coupled regret \eqref{DWD} can be small, or even zero, when the coupling $\gamma$ mostly pairs outcomes within the same cones, even if the paired points are far apart in Euclidean distance. Conversely, moving probability mass between regions with different unique optimal decisions produces positive regret, which by \eqref{eq:spo_cross} is at most $D_W$ times the distance moved.

The right panel of Fig.~\ref{fig:intro_customer} illustrates three point clouds (Groups 1--3), each representing an empirical distribution of cost vectors, for example a population of customers or patients at three points in time. Groups~1 (purple) and~2 (green) lie in Cone~4 and thus induce the same decision, whereas Group~3 (red) lies in Cone~1 and induces a different decision. Purely geometric distances regard Groups~2 and~3 as close and Groups~1 and~2 as far apart, whereas the DF divergence is zero between Groups~1 and~2 and positive between Groups~2 and~3. More generally, if $\mu$ and $\nu$ are supported in the same cone, then $w^*(X)=w^*(Y)$ almost surely for every coupling, so every coupled regret is zero. When $\mu$ and $\nu$ place mass on several cones, the optimistic DF divergence favors couplings that align mass within the same cones.

The optimistic DF divergence can be zero even when $\mu$ and $\nu$ are geometrically far apart. A basic instance follows from the scale invariance $w^*(\kappa x)=w^*(x)$ for $\kappa>0$.

\begin{proposition}[Random rescaling yields zero optimistic DF divergence]\label{thm:rescale_wdfo_zero}
Let $X\in\bbR^d$ be any random cost vector with distribution $\mu$. Let $\kappa$ be any nonnegative random scalar (which may depend on $X$), and define the transformed vector $ Y := \kappa X $, with distribution $\nu$. Then $\wdfo(\mu,\nu)=0 $.
\end{proposition}

Proofs of the results in the main text are provided in Appendix~\ref{append:proofs}.

\section{Computation: Push-Forward Reduction and a Quadratic Transport Formulation}\label{sec:calculation}

Computing $\wdfo(\mu,\nu)$ amounts to solving an optimal transport problem with transport cost $\lspo(x,y)$. When $\mu$ and $\nu$ are replaced by empirical measures with $n_\mu$ and $n_\nu$ atoms, this problem is a linear program with $n_\mu n_\nu$ variables, which can become computationally expensive as the sample sizes grow (See Appendix \ref{append:OT_LP}). If either population is continuous, the original formulation instead requires optimization over an infinite-dimensional space of couplings. This section exploits the structure of the decision loss: by \eqref{eq:spo_ci}, $\lspo(x,y)$ depends on $x$ only through the optimal decision $w^*(x)$. Consequently, individuals assigned the same decision can be grouped together without changing the expected regret. We first explain this reduction, then derive an equivalent quadratic transport formulation, and finally turn these results into an algorithm.

Figure~\ref{fig:reduction} illustrates the reduction through a personalized product-recommendation problem. Each feasible product design is represented by a decision vector $w=(w_1,w_2)\in S$, whose first coordinate represents price and whose second coordinate represents negative quality. The black polygon $S$ therefore describes the feasible combinations of price and quality. An individual's coefficient vector $x=(x_1,x_2)$ records their price sensitivity and quality valuation, respectively, so the objective $x^\top w=x_1w_1+x_2w_2$ balances the cost of the product against its quality. For positive preference coefficients, minimizing this objective favors lower prices and higher quality, subject to the feasible design choices in $S$. The oracle $w^*(x)$ selects an optimal product design for that individual. The populations $\mu$ and $\nu$ describe the distributions of these preference coefficients before and after the population shift, respectively.

We first replace the distribution of individual coefficient vectors by the distribution of their optimal decisions. For $X\sim\mu$, the random decision $W:=w^*(X)$ has distribution $w^*_{\#}\mu$, which we call the \emph{push-forward distribution} of $\mu$ under the oracle map. Writing $D_i=\{x:w^*(x)=w_i\}$ and $p_i:=\mu(D_i)$, we have
\[
w^*_{\#}\mu
=
\sum_{i=1}^{|\mathfrak{S}|}p_i\,\delta_{w_i},
\]
where $\delta_{w_i}$ denotes the point mass at $w_i$. Thus, $p_i$ is the share of individuals whose optimal decision is $w_i$. In the product-recommendation example, all individuals receiving the same product design contribute to the mass at that design. Figure~\ref{fig:reduction_pushforward} illustrates this grouping: the original coefficient vectors are mapped to their corresponding optimal designs, and the probability mass assigned to each design is combined. The resulting distribution is finitely supported even when $\mu$ is continuous.

We next describe how this grouping transforms a coupling. For any $\gamma\in\Gamma(\mu,\nu)$, let $(X,Y)\sim\gamma$ and define the \emph{reduced coupling} $\bar\gamma$ as the joint distribution of $(w^*(X),Y)$. Explicitly, for every extreme point $w_i$ and every Borel set $A\subseteq\bbR^d$,
\[
\bar\gamma(\{w_i\}\times A)
=
\gamma(D_i\times A).
\]
The first marginal of $\bar\gamma$ is $w^*_{\#}\mu$, and its second marginal is $\nu$, so $\bar\gamma\in\Gamma(w^*_{\#}\mu,\nu)$. Whereas $\gamma$ couples the original and subsequent coefficient vectors of individuals, $\bar\gamma$ couples their original decisions with their subsequent coefficient vectors.

For $w\in\mathfrak{S}$, let $c(w,y):=y^\top w-z(y)$ denote the regret of retaining decision $w$ when the coefficient vector is $y$. This agrees with the reduced cost defined earlier: $c(w_i,y)=c_i(y)$. Since $\lspo(x,y)=c(w^*(x),y)$, the original and reduced couplings incur the same expected regret:
\[
\int \lspo(x,y)\,\gamma(dx,dy)
=
\int c(w,y)\,\bar\gamma(dw,dy).
\]
Conversely, any coupling $\bar\gamma\in\Gamma(w^*_{\#}\mu,\nu)$ can be realized by the original populations. First draw $(W,Y)\sim\bar\gamma$. Conditional on $W=w_i$ with $p_i>0$, draw $X$ from $\mu(\,\cdot\,\mid D_i)$ independently of $Y$ given $W$. This construction gives $X\sim\mu$, $Y\sim\nu$, and $w^*(X)=W$ almost surely, and therefore preserves the expected regret. Taking the infimum over the original or reduced couplings consequently gives the same optimal value:
\begin{equation}\label{eq:reduced}
\wdfo(\mu,\nu)
=
\inf_{\bar\gamma\in\Gamma(w^*_{\#}\mu,\nu)}
\int_{\mathfrak{S}\times\bbR^d} c(w,y)\,\bar\gamma(dw,dy).
\end{equation}

The reduced problem therefore depends on the population for decision making only through the shares $p_i$ assigned to its distinct optimal decisions. Its first marginal has at most $|\mathfrak{S}|$ support points, regardless of the coefficient dimension or whether $\mu$ is continuous. When $\nu$ is continuous, \eqref{eq:reduced} is a \emph{semi-discrete} transport problem, with one finitely supported marginal and one continuous marginal. When $\nu$ is empirical, it is a discrete transport problem with one source point for each distinct optimal decision represented in $\mu$.

\begin{figure}[ht]
    \centering
    \begin{subfigure}[t]{0.48\linewidth}
        \centering
        \includegraphics[width=\linewidth]{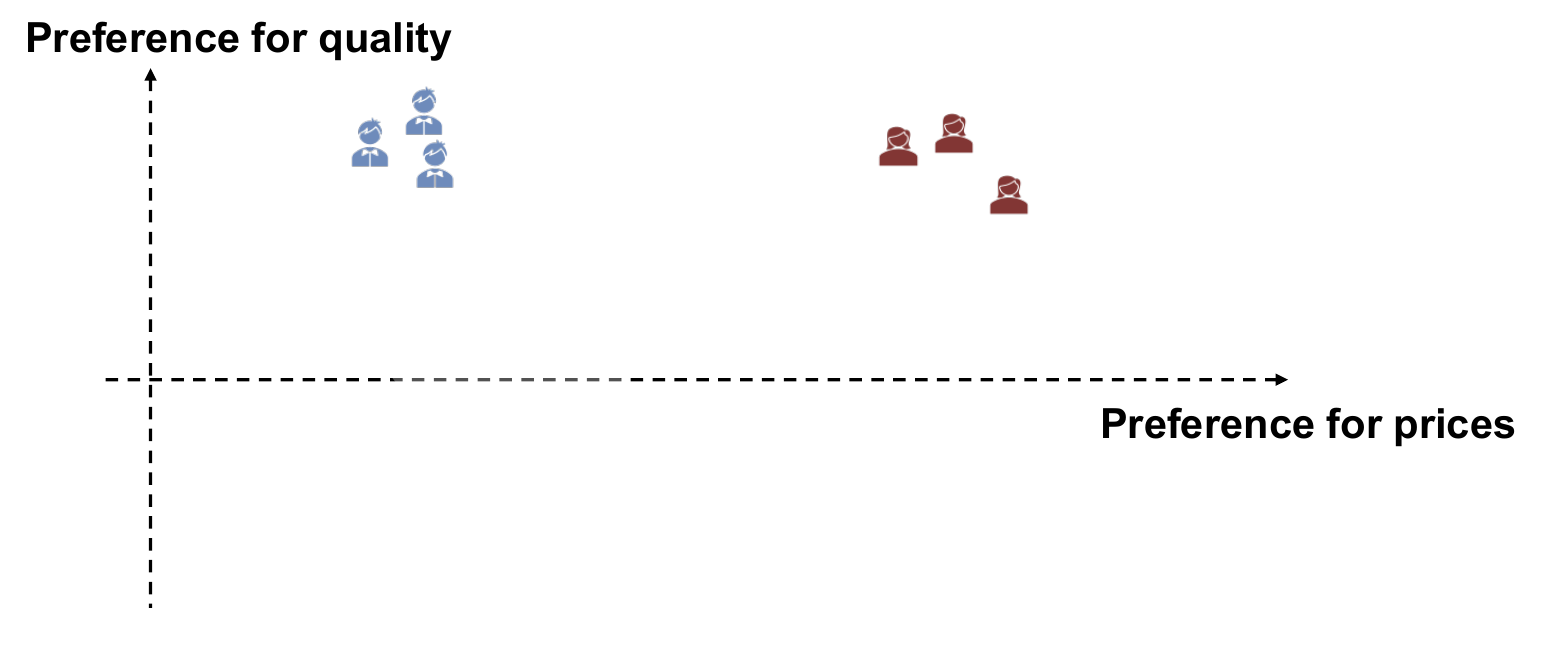}
        \caption{Direct formulation: $\gamma$ couples the coefficient vectors of the two populations.}
        \label{fig:reduction_direct}
    \end{subfigure}
    \hfill
    \begin{subfigure}[t]{0.48\linewidth}
        \centering
        \includegraphics[width=\linewidth]{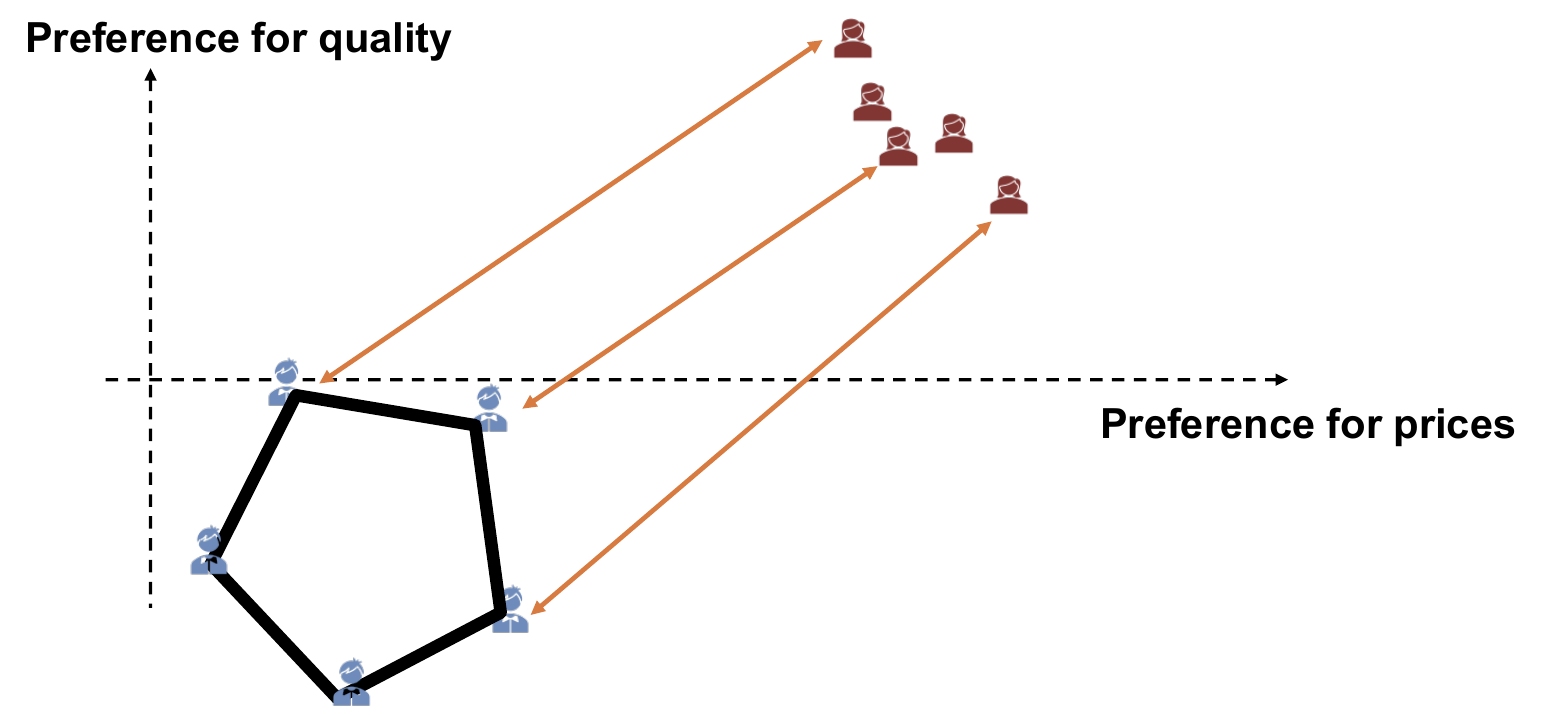}
        \caption{Reduced formulation: $\bar\gamma$ couples the original optimal decisions with the subsequent coefficient vectors.}
        \label{fig:reduction_pushforward}
    \end{subfigure}
    \caption{Push-forward reduction for personalized product recommendation. The black polygon represents feasible product designs, with price as the first coordinate and negative quality as the second. Blue and red markers represent individuals' preference coefficients under $\mu$ and $\nu$, respectively. Panel (a) illustrates a coupling of the original coefficient distributions. In panel (b), the coefficient vectors under $\mu$ are mapped to their optimal designs, giving mass $p_i$ at each extreme point $w_i$; the reduced coupling then pairs these decision masses with the coefficient vectors under $\nu$. Coefficient vectors and product designs have different coordinate meanings and are superimposed schematically. The finite markers are illustrative; the reduction also applies to \textbf{continuous} populations whose support spans multiple decision regions.}
    \label{fig:reduction}
\end{figure}

We next show that, when $\nu$ has a finite second moment, the reduced problem also admits an equivalent quadratic transport formulation.

\begin{theorem}[Quadratic optimal transport formulation]\label{thm:1}
Assume that $\int\|y\|_2^2\,\nu(dy)<\infty$, let $\bar w^*:=-w^*$, and for any probability measure $\rho$ on $\bbR^d$ write $\|\rho\|^2:=\int \|v\|_2^2\,d\rho(v)$. For $\bar\gamma\in\Gamma(w^*_{\#}\mu,\nu)$, the following statements are equivalent:
\begin{enumerate}
    \item[(a)] $\bar\gamma$ is optimal for the reduced problem \eqref{eq:reduced};
    \item[(b)] $\bar\gamma$ maximizes $\int\|w-y\|_2^2\,d\bar\gamma(w,y)$ over $\Gamma(w^*_{\#}\mu,\nu)$;
    \item[(c)] the reflected plan $(-\mathrm{Id}\times\mathrm{Id})_{\#}\bar\gamma$ is an optimal coupling for $W_2\big((\bar w^*)_{\#}\mu,\nu\big)$.
\end{enumerate}
Moreover,
\begin{align}
    \wdfo(\mu,\nu)
    &=\frac{1}{2}\Big[\|w^*_{\#}\mu\|^2+\|\nu\|^2-2\int z\,d\nu-\sup_{\bar\gamma\in\Gamma(w^*_{\#}\mu,\nu)}\int\|w-y\|_2^2\,d\bar\gamma(w,y)\Big]\nonumber\\
    &=\frac{1}{2}\Big[W_2^2\big((\bar w^*)_{\#}\mu,\nu\big)-\|w^*_{\#}\mu\|^2-\|\nu\|^2-2\int y^\top w^*(y)\,d\nu(y)\Big].\label{DWDt}
\end{align}
\end{theorem}

Theorem~\ref{thm:1} shows that the only coupling-dependent term in $\wdfo(\mu,\nu)$ is a quadratic transport cost between the distribution of optimal decisions and the population for evaluation. Because $(\bar w^*)_{\#}\mu$ is supported on the finite set $\{-w_i:\,w_i\in\mathfrak{S}\}$ even when $\mu$ has a continuous density, computing $\wdfo(\mu,\nu)$ reduces to a semi-discrete quadratic optimal transport problem when $\nu$ is continuous, and to a small discrete problem when $\nu$ is empirical.

\begin{remark}[Subtracting decisions from coefficients]\label{rem:subtraction}
At first sight, Theorem~\ref{thm:1} is counterintuitive. A decision vector $w$ may represent dosages, hours, or flows, whereas a coefficient vector $y$ may represent costs, side effects, or travel times, so comparing the two through a Euclidean distance appears unnatural. The reformulation is nevertheless exact, because the coupling enters the SPO loss only through the inner product $y^\top w^*(x)$, while the squared norms of $y$ and $w^*(x)$ are fixed by the marginals. Consequently, the best-case alignment maximizes $\EE\|Y-w^*(X)\|_2^2$, pairing each decision with coefficient vectors that are as far from it as possible; equivalently, it solves the $W_2$ problem between $\nu$ and the negated decisions. Finally, the reformulation does not depend on units. Decision vectors and coefficient vectors may therefore be rescaled to a common scale, which can also improve numerical conditioning.
\end{remark}

After solving the reduced problem, an optimal coupling of the original populations can be recovered by conditioning on the decision regions.

\begin{theorem}[Lifting optimal couplings]\label{thm:2}
Let $\bar\gamma^{*}\in\Gamma(w^*_{\#}\mu,\nu)$ be optimal for the reduced problem \eqref{eq:reduced}. Define a measure $\gamma$ on measurable rectangles by
\begin{equation}\label{eq:thm2_explicit_lift}
\gamma(A\times B)
:=\sum_{k=1}^{|\mathfrak{S}|}\bar\gamma^{*}(\{w_k\}\times B)\,
\frac{\mu(A\cap D_k)}{\mu(D_k)},
\qquad A,B\in\mathcal{B}(\bbR^d),
\end{equation}
with the convention that the $k$th term is $0$ when $\mu(D_k)=0$. Equivalently,
\begin{equation}\label{eq:thm2_explicit_rect}
\gamma(D_k\times B)=\bar\gamma^{*}(\{w_k\}\times B),\qquad k=1,\ldots,|\mathfrak{S}|,
\end{equation}
and, given $w^*(X)=w_k$, the coordinates $X$ and $Y$ are conditionally independent with $X\sim\mu(\cdot\,|\,D_k)$. Then $\gamma\in\Gamma(\mu,\nu)$, $(w^*\times\mathrm{Id})_{\#}\gamma=\bar\gamma^*$, and $\wspo(\mu,\nu;\gamma)=\wdfo(\mu,\nu)$.
\end{theorem}

In Theorem~\ref{thm:2}, $\bar\gamma^{*}(\{w_k\}\times B)$ is the mass that $\bar\gamma^*$ assigns to pairs $(W,Y)$ with $W=w_k$ and $Y\in B$, and the factor $\mu(A\cap D_k)/\mu(D_k)$ is the conditional probability that $X\in A$ given $X\in D_k$. Thus, \eqref{eq:thm2_explicit_lift} samples $(W,Y)\sim\bar\gamma^*$ and, whenever $W=w_k$, samples $X$ from $\mu$ restricted to $D_k$. This ensures $w^*(X)=W$ while preserving the $Y$-marginal.

\begin{remark}[What the divergence identifies]\label{rem:nonunique}
Optimal couplings of $\mu$ and $\nu$ are generally not unique. Because $\lspo(x,y)$ depends on $x$ only through $w^*(x)$, any rearrangement of mass within a decision region preserves optimality, and the lift in Theorem~\ref{thm:2} is one of many optimal couplings. The quantities that the optimistic DF divergence identifies are its value, the decision-level plans that solve \eqref{eq:reduced}, and the dual prices in Proposition~\ref{prop:dual_wdfo} below. This distinction is immaterial for evaluating the regret, but it matters for constructions that depend on the full coupling, such as the interpolation in Appendix~\ref{sec:interpolant}.
\end{remark}

The reduced problem also admits a finite-dimensional dual, which describes the optimal reassignment.

\begin{proposition}[Strong duality for the optimistic DF divergence]\label{prop:dual_wdfo}
Recall the reduced costs $c_i$ in \eqref{eq:reduced_cost}. Let $\Phi_{\mathrm{SPO}}$ be the set of all $(f,g)\in\bbR^{|\mathfrak{S}|}\times L^1(\nu)$ satisfying
\begin{equation}\label{eq:PhiSPO}
    f_i+g(y)\le c_i(y),\qquad \forall i\in\{1,\ldots,|\mathfrak{S}|\},\ \forall y\in\bbR^d,
\end{equation}
where $L^1(\nu)$ denotes the set of measurable functions $g:\bbR^d\to\bbR$ with $\int_{\bbR^d} |g(y)|\,\nu(dy)<\infty$. Then
\begin{equation}\label{eq:WDFO_dual}
    \wdfo(\mu,\nu)=\sup_{(f,g)\in\Phi_{\mathrm{SPO}}}
    \left\{\sum_{i=1}^{|\mathfrak{S}|} f_i\,\mu(D_i)+\int_{\bbR^d} g(y)\,\nu(dy)\right\}.
\end{equation}
Moreover, for any fixed $f\in\bbR^{|\mathfrak{S}|}$, the optimal choice of $g$ in \eqref{eq:WDFO_dual} is
\begin{equation}\label{eq:g_star}
    g_f(y):=\min_{1\le i\le |\mathfrak{S}|}\big(c_i(y)-f_i\big),
\end{equation}
and hence \eqref{eq:WDFO_dual} can be written equivalently as
\begin{equation}\label{eq:WDFO_dual_reduced}
    \wdfo(\mu,\nu)=\sup_{f\in\bbR^{|\mathfrak{S}|}}
    \left\{\sum_{i=1}^{|\mathfrak{S}|} f_i\,\mu(D_i)+\int_{\bbR^d} g_f(y)\,\nu(dy)\right\}.
\end{equation}
\end{proposition}

Proposition~\ref{prop:dual_wdfo} replaces the infinite-dimensional optimization over couplings by a finite-dimensional concave maximization over $f\in\bbR^{|\mathfrak{S}|}$, one variable per decision. This structure underlies the statistical analysis in Section~\ref{sec:sample_complexity}. Corollary~\ref{cor:price_shift} in Appendix~\ref{append:proofs_calculation} characterizes the optimal reassignment.

Algorithm~\ref{alg:dfot} implements the reduction illustrated in Figure~\ref{fig:reduction} for empirical measures $\hat\mu_{n_\mu}=\frac{1}{n_\mu}\sum_{t}\delta_{x_t}$ and $\hat\nu_{n_\nu}=\frac{1}{n_\nu}\sum_j\delta_{y_j}$.

\begin{algorithm}[htbp]
\caption{Optimistic DF divergence by push-forward reduction}\label{alg:dfot}
\begin{algorithmic}[1]
\Require Samples $x_1,\ldots,x_{n_\mu}$ from the population for decision making, samples $y_1,\ldots,y_{n_\nu}$ from the population for evaluation, and the oracle $w^*(\cdot)$.
\State Compute $w^*(x_t)$ for $t=1,\ldots,n_\mu$; set $\hat p_i\gets |\{t:\,w^*(x_t)=w_i\}|/n_\mu$ and $\hat I\gets\{i:\,\hat p_i>0\}$.
\State Compute $z(y_j)=\min_{w\in S}w^\top y_j$ for $j=1,\ldots,n_\nu$.
\State Form the reduced cost matrix $C_{ij}\gets y_j^\top w_i-z(y_j)$ for $i\in\hat I$ and $j=1,\ldots,n_\nu$.
\State Solve $\min_{\Pi\ge 0}\sum_{i\in\hat I}\sum_{j}C_{ij}\Pi_{ij}$ subject to $\sum_j\Pi_{ij}=\hat p_i$ for $i\in\hat I$ and $\sum_{i\in\hat I}\Pi_{ij}=1/n_\nu$ for all $j$, or the equivalent quadratic problem of Theorem~\ref{thm:1}; record an optimal plan $\Pi^*$ and dual prices $f^*$.
\State (Optional) Lift $\Pi^*$ to a coupling of the samples through \eqref{eq:thm2_explicit_lift}.
\State \Return $\wdfo(\hat\mu_{n_\mu},\hat\nu_{n_\nu})=\sum_{i,j}C_{ij}\Pi^*_{ij}$, the decision-level plan $\Pi^*$, and the prices $f^*$.
\end{algorithmic}
\end{algorithm}

Algorithm~\ref{alg:dfot} requires $n_\mu+n_\nu$ calls to the linear optimization oracle, which are also needed to form the cost matrix of the direct linear program, and solves a transportation problem with $|\hat I|\,n_\nu$ variables instead of $n_\mu n_\nu$. Since $|\hat I|\le\min\{n_\mu,|\mathfrak{S}|\}$, the reduced problem is never larger than the direct one, and it is much smaller when the population for decision making uses few distinct decisions. For combinatorial feasible regions in which almost every sample receives a different decision, $|\hat I|$ approaches $n_\mu$ and the reduction offers little gain. When $\nu$ is continuous, the reduced problem is semi-discrete; the concave dual \eqref{eq:WDFO_dual_reduced} in $f$ can then be maximized by stochastic gradient ascent, in which each step draws one sample $y$ and minimizes over $|\hat I|$ decisions \citep{genevay2016stochastic}, or, through the quadratic formulation of Theorem~\ref{thm:1}, by Newton methods for semi-discrete $W_2$ transport \citep{kitagawa2019convergence}, which are most effective in low dimensions. The same reduction applies to the entropy-regularized problem studied in Section~\ref{sec:property} (Proposition~\ref{prop:entropic_reduction}).

\section{Finite-Sample Estimation Error Bound}\label{sec:sample_complexity}

In practice, both populations are observed through samples. We observe samples $X_1,\ldots,X_{n_\mu}\stackrel{\mathrm{i.i.d.}}{\sim}\mu$ and $Y_1,\ldots,Y_{n_\nu}\stackrel{\mathrm{i.i.d.}}{\sim}\nu$ and form the empirical measures $\hat\mu_{n_\mu}$ and $\hat\nu_{n_\nu}$. The two samples need not be independent; this includes i.i.d. paired observations across individuals. The plug-in estimator of the optimistic DF divergence is $\wdfo(\hat\mu_{n_\mu},\hat\nu_{n_\nu})$, computed by Algorithm~\ref{alg:dfot}. We quantify how many samples are needed for this estimator to be within a tolerance $\epsilon$ of $\wdfo(\mu,\nu)$.

For classical Wasserstein distances, plug-in estimation can suffer from the curse of dimensionality: for distributions with $d$-dimensional bounded support, general expected-error bounds scale as $n^{-1/d}$ for Lipschitz costs such as that of $W_1$ when $d>2$, and as $n^{-2/d}$ for semi-concave costs such as that of $W_2^2$ when $d>4$ \citep{ChewiNilesWeedRigollet2025StatOT,hundrieser2024empirical}. \citet{hundrieser2024empirical} show that the rate of the empirical optimal transport cost is governed by the less complex of the two measures and equals $n^{-1/2}$ when one of them is finitely supported. By the reduced problem \eqref{eq:reduced}, the optimistic DF divergence is a transport problem whose first marginal $w^*_{\#}\mu$ is supported on at most $|\mathfrak{S}|$ points, regardless of $d$. The following theorem makes this consequence of the finite-support reduction explicit, with constants depending on the number of decisions used by $\mu$ and a uniform regret bound.

Throughout this section, we assume that $\nu$ is supported on a bounded set $\mathcal{Y}\subset\bbR^d$ and impose no support restriction on $\mu$. Let
\[
I_\mu:=\{i:\ \mu(D_i)>0\},\qquad K_\mu:=|I_\mu|\le|\mathfrak{S}|,\qquad
B:=\max_{i\in I_\mu}\sup_{y\in\mathcal{Y}}c_i(y),
\]
so that $K_\mu$ is the number of decisions used by $\mu$ and $0\le c_i(y)\le B$ for $i\in I_\mu$ and $y\in\mathcal{Y}$. By \eqref{eq:spo_bound}, $B\le D_Wr_{\mathcal{Y}}$, where $r_{\mathcal{Y}}:=\sup_{y\in\mathcal{Y}}\|y\|_2$.

\begin{theorem}[Two-sample concentration for $\wdfo$]\label{ThmSampleComplexity}
For any $\delta\in(0,1)$, with probability at least $1-\delta$,
\begin{equation}\label{eq:sample_complexity_bound}
\Big|\wdfo(\hat{\mu}_{n_\mu},\hat{\nu}_{n_\nu})-\wdfo(\mu,\nu)\Big|
\le
\frac{B}{2}\sqrt{\frac{K_\mu}{n_\mu}}
+
C\,B\sqrt{\frac{K_\mu}{n_\nu}}
+
B\sqrt{\frac{\log(2/\delta)}{2n_\mu}}
+
B\sqrt{\frac{\log(2/\delta)}{2n_\nu}},
\end{equation}
where $C>0$ is a universal constant.
\end{theorem}

Theorem~\ref{ThmSampleComplexity} implies an $\epsilon$-accuracy sample complexity. Ignoring logarithmic factors in $\delta$ and constants, the dominant terms scale as $B\sqrt{K_\mu/n_\mu}$ and $CB\sqrt{K_\mu/n_\nu}$, so that it suffices to take $n_\mu = \tilde{O}(B^2K_\mu/\epsilon^2)$ and $n_\nu = \tilde{O}(B^2K_\mu/\epsilon^2)$, where $\tilde{O}(\cdot)$ suppresses logarithmic dependence on $1/\delta$. The exponent of the sample size does not depend on the dimension $d$, in contrast with the rates $n^{-1/d}$ and $n^{-2/d}$ above.

The constants in Theorem~\ref{ThmSampleComplexity} may depend on the dimension $d$ through $B$ and $K_\mu$, where $K_\mu$ is bounded by the number of extreme points of $S$. For combinatorial feasible regions, the number of extreme points can grow exponentially with the problem size; for example, there are $\binom{2k-2}{k-1}$ monotone paths on a $k\times k$ grid. In such cases, the bound remains favorable when the population for decision making uses relatively few distinct optimal decisions. Furthermore, a parametric rate alone does not distinguish the DF divergence from all decision-blind discrepancies: kernel discrepancies such as the maximum mean discrepancy also admit $n^{-1/2}$ estimation rates in any dimension \citep{gretton2012kernel}. The advantage of the DF divergence is that it combines a parametric estimation rate with a decision-relevant transport cost.

\begin{remark}[High-dimensional coefficients with few decisions]\label{rem:high_dim_few_decisions}
Theorem~\ref{ThmSampleComplexity} is particularly beneficial when the coefficient space is high-dimensional but the feasible region has relatively few extreme points, or, more generally, when the population for decision making uses only a small subset of those extreme points. In such settings, the statistical complexity of estimating the DF divergence is governed by the number $K_\mu$ of distinct optimal decisions rather than the coefficient dimension $d$, provided that the regret bound $B$ remains controlled. In particular, if $K_\mu$ and $B$ remain bounded as $d$ increases, the estimation bound is uniform in $d$.
For example, in personalized resource allocation, each individual may be described by a high-dimensional coefficient vector but receive one of a small number of candidate allocation plans. Taking $S$ to be the convex hull of these plans gives at most as many extreme points as candidate plans, so $K_\mu$ can remain small even as $d$ increases. \hfill \Halmos
\end{remark}


\section{Comparison with Classical Metrics and Entropy Regularization}\label{sec:property}

This section relates the optimistic DF divergence to classical distributional metrics, characterizes its balls as uncertainty sets, and studies its entropy-regularized version. Throughout this section, we assume that $\mu$ and $\nu$ are supported on a bounded set $\mathcal{Y}\subset\bbR^d$, and we write $\diam(\mathcal{Y}):=\sup_{x,y\in\mathcal{Y}}\|x-y\|_2$ and $r_{\mathcal{Y}}:=\sup_{y\in\mathcal{Y}}\|y\|_2$.

\subsection{Comparison with classical metrics}

We first compare $\wdfo$ with the $W_1$ distance, the total variation distance, and the KL divergence. Let $W^O_{\mathrm{sym}}(\mu,\nu):=\wdfo(\mu,\nu)+\wdfo(\nu,\mu)$ denote the additive symmetrization.

\begin{proposition}[Comparison with $W_1$ and total variation]\label{prop:lips_w1}
Assume that $\mu$ and $\nu$ are supported on $\mathcal{Y}$. Then $\wdfo(\mu,\nu)\le W^O_{\mathrm{sym}}(\mu,\nu)\le D_W\,W_1(\mu,\nu)\le D_W\,\diam(\mathcal{Y})\,\tv(\mu,\nu)$.
\end{proposition}

Proposition~\ref{prop:lips_w1}, together with Pinsker's inequality, shows that populations close in $W_1$, total variation, or KL are also close in the DF divergence. The reverse inequalities fail. By Corollary~\ref{prop:basic}(iii) and Proposition~\ref{thm:rescale_wdfo_zero}, the DF divergence vanishes whenever the distribution of optimal decisions does not change, however large the classical discrepancies are. The bound in Proposition~\ref{prop:lips_w1} does not exploit the piecewise-constant structure of the oracle. The next proposition does, by comparing the push-forward measures directly. Intuitively, $w^*_{\#}\mu$ and $w^*_{\#}\nu$ record how $\mu$ and $\nu$ distribute mass across decisions.

\begin{proposition}[Comparison through push-forward measures]\label{prop:lips}
Let $V^2:=\int_{\bbR^d}\|y\|_2^2\,\nu(dy)<\infty$. Then
\[
\wdfo(\mu,\nu)\le V\,W_2\!\left(w^*_{\#}\mu,\;w^*_{\#}\nu\right)
\le V\,D_W\sqrt{\tv\!\left(w^*_{\#}\mu,\;w^*_{\#}\nu\right)}
\le V\,D_W\Big(\tfrac12\,\KL\!\left(w^*_{\#}\mu\|w^*_{\#}\nu\right)\Big)^{1/4}.
\]
If, in addition, $\nu$ is supported on $\mathcal{Y}$, then
\begin{equation}\label{eq:linear_tv}
\wdfo(\mu,\nu)\le D_W\,r_{\mathcal{Y}}\,\tv\!\left(w^*_{\#}\mu,\;w^*_{\#}\nu\right)
\le D_W\,r_{\mathcal{Y}}\sqrt{\tfrac12\,\KL\!\left(w^*_{\#}\mu\|w^*_{\#}\nu\right)}.
\end{equation}
\end{proposition}

All bounds in Proposition~\ref{prop:lips} vanish when the two push-forward measures agree. The bound \eqref{eq:linear_tv} is linear rather than square-root in their total variation distance, and is therefore sharper when they are close.

\subsection{The DF divergence ball as an uncertainty set}\label{sec:df_ball}

The above comparison inequalities concern a given pair of populations. A complementary question arises when a decision maker specifies a regret tolerance in advance: which population shifts can be tolerated without updating the decisions? In this subsection, we assume that $S$ is not a singleton, so that $D_W>0$, and we require the populations to be supported on $\mathcal{Y}$ only where stated. Let $\cP_1(\bbR^d)$ denote the set of probability measures on $\bbR^d$ with finite first moments. For a population for decision making $\mu$ and a tolerance $\Delta\ge0$, we define the \emph{DF divergence ball}
\begin{equation}\label{eq:df_ball}
\mathcal{B}_{\mathrm{DF}}(\mu,\Delta):=\big\{\nu\in\cP_1(\bbR^d):\ \wdfo(\mu,\nu)\le\Delta\big\}.
\end{equation}
To compare this ball with other uncertainty sets, we formalize what it means for a set of population shifts to be compatible with the tolerance.

\begin{definition}[Tolerable set of population shifts]\label{def:tolerable}
A set $\mathcal{U}\subseteq\cP_1(\bbR^d)$ of populations for evaluation is \emph{$\Delta$-tolerable} for $\mu$ if, for every $\nu\in\mathcal{U}$, there exists a coupling $\gamma\in\Gamma(\mu,\nu)$ with $\wspo(\mu,\nu;\gamma)\le\Delta$.
\end{definition}

In words, a set of population shifts is $\Delta$-tolerable if, for every shift in the set, keeping the decisions made for $\mu$ can cost at most $\Delta$. The next theorem shows that the DF divergence ball is the largest $\Delta$-tolerable set, describes its structure, and compares it with the classical balls that are $\Delta$-tolerable.

\begin{theorem}[The DF divergence ball is the largest tolerable set of population shifts]\label{thm:df_ball}
Let $\mu$ be a probability measure on $\bbR^d$ and $\Delta\ge0$.
\begin{enumerate}
    \item[(i)] \emph{Maximality.} The ball $\mathcal{B}_{\mathrm{DF}}(\mu,\Delta)$ is $\Delta$-tolerable for $\mu$, and every set that is $\Delta$-tolerable for $\mu$ is contained in $\mathcal{B}_{\mathrm{DF}}(\mu,\Delta)$.
    \item[(ii)] \emph{Structure.} The ball $\mathcal{B}_{\mathrm{DF}}(\mu,\Delta)$ is convex, is closed with respect to $W_1$, depends on $\mu$ only through $w^*_{\#}\mu$, and satisfies
    \begin{equation}\label{eq:df_ball_dual}
    \mathcal{B}_{\mathrm{DF}}(\mu,\Delta)=\Big\{\nu\in\cP_1(\bbR^d):\ \sum_{i=1}^{|\mathfrak{S}|}f_i\,\mu(D_i)+\int g_f\,d\nu\le\Delta\ \text{ for all } f\in\bbR^{|\mathfrak{S}|}\Big\},
    \end{equation}
    with $g_f$ as in \eqref{eq:g_star}. Moreover, $\mathcal{B}_{\mathrm{DF}}(\mu,0)$ contains every $\nu\in\cP_1(\bbR^d)$ with $w^*_{\#}\nu=w^*_{\#}\mu$, and these are its only elements with $\nu(\mathcal{T})=0$.
    \item[(iii)] \emph{Connections to classical balls.} If $\mu\in\cP_1(\bbR^d)$, then $
    \big\{\nu\in\cP_1(\bbR^d):\ W_1(\mu,\nu)\le\Delta/D_W\big\}\subseteq\mathcal{B}_{\mathrm{DF}}(\mu,\Delta)$.
    If $\mu$ is supported on $\mathcal{Y}$, then every $\nu$ supported on $\mathcal{Y}$ with $\tv(\mu,\nu)\le\Delta/(D_W\,\diam(\mathcal{Y}))$ or with $\KL(\mu\|\nu)\le2\Delta^2/(D_W\,\diam(\mathcal{Y}))^2$ belongs to $\mathcal{B}_{\mathrm{DF}}(\mu,\Delta)$. The radius $\Delta/D_W$ cannot be improved uniformly over problem instances: for every $\Delta>0$ and $\epsilon>0$, there exist a feasible region $S$ and probability measures $\mu$ and $\nu$ with $W_1(\mu,\nu)\le(1+\epsilon)\Delta/D_W$ and $\wdfo(\mu,\nu)>\Delta$.
    \item[(iv)] \emph{Unboundedness.} If $\mu\in\cP_1(\bbR^d)$ and $\mu(\{0\})=0$, then
    \[
    \sup_{\nu\in\mathcal{B}_{\mathrm{DF}}(\mu,0)}W_1(\mu,\nu)=\infty
    \qquad\text{and}\qquad
    \sup_{\nu\in\mathcal{B}_{\mathrm{DF}}(\mu,0)}\tv(\mu,\nu)=1.
    \]
    Consequently, for every $\Delta\ge0$, no $W_1$ ball of finite radius around $\mu$ and no total variation ball of radius less than one around $\mu$ contains $\mathcal{B}_{\mathrm{DF}}(\mu,\Delta)$.
\end{enumerate}
\end{theorem}

Part (i) follows directly from the definition of the divergence as a minimum over couplings; its role is to identify $\mathcal{B}_{\mathrm{DF}}(\mu,\Delta)$ as the exact set of population shifts whose unavoidable regret is at most $\Delta$. Part (ii) shows that this set is tractable: it is convex, and \eqref{eq:df_ball_dual} describes it by constraints that are linear in $\nu$, one for each price vector $f$. Parts (iii) and (iv) compare it with decision-blind uncertainty sets. A Wasserstein, total variation, or KL ball around $\mu$ is $\Delta$-tolerable when its radius is small enough, and $\Delta/D_W$ is the largest Wasserstein radius that is valid for every problem instance. In contrast, the DF divergence ball also contains populations that are arbitrarily far from $\mu$ in the Wasserstein distance, such as every population with the same distribution of optimal decisions as $\mu$. Calibrating a classical ball to the tolerance $\Delta$ therefore excludes many shifts whose unavoidable regret is within the tolerance. The DF divergence ball may also be useful for robust analysis, clustering, and kernel construction. We leave these potential applications for future research.

\subsection{Entropy regularization}\label{sec:entropy}

The optimistic DF divergence describes maximal persistence, and the independent-coupling regret describes none. To obtain intermediate regret values that retain a controlled amount of dependence between an individual's coefficient vectors before and after the shift, we regularize the coupling toward independence. For $\gamma\in\Gamma(\mu,\nu)$, the relative entropy $\KL(\gamma\|\mu\otimes\nu)$ equals the mutual information between $X$ and $Y$ under $\gamma$, so the regularization penalizes dependence. Entropy regularization is also a standard tool for improving the computational tractability and stability of optimal transport \citep{cuturi2013sinkhorn,peyre2019computational}.

\begin{definition}[Entropy-regularized optimistic DF divergence]\label{def:ent_df}
For $\varepsilon>0$, define
\begin{equation}\tag{E-DFO}\label{EWDFO}
    \wdfeo(\mu,\nu;\varepsilon)
    :=\inf_{\gamma\in\Gamma(\mu,\nu)}
    \left\{\int \lspo(x,y)\,d\gamma(x,y)+\varepsilon\,\KL(\gamma\|\mu\otimes\nu)\right\},
\end{equation}
where $\KL(\gamma\|\mu\otimes\nu)$ denotes the KL divergence of $\gamma$ relative to the independent coupling $\mu\otimes\nu$.
\end{definition}

As $\varepsilon$ increases, the minimizer is increasingly encouraged to stay close to the independent coupling. We first show that problem \eqref{EWDFO} admits the same push-forward reduction as the optimistic DF divergence. Consequently, for empirical measures it requires a Sinkhorn computation on a matrix with one row per decision used, as in Algorithm~\ref{alg:dfot}, rather than one row per sample.

\begin{proposition}[Reduction of the entropy-regularized problem]\label{prop:entropic_reduction}
Let $\varepsilon>0$ and $\int\|y\|_2\,\nu(dy)<\infty$. Then
\begin{equation}\label{eq:entropic_reduced}
\wdfeo(\mu,\nu;\varepsilon)=\inf_{\bar\gamma\in\Gamma(w^*_{\#}\mu,\nu)}\left\{\int c\,d\bar\gamma+\varepsilon\,\KL\big(\bar\gamma\,\|\,w^*_{\#}\mu\otimes\nu\big)\right\}.
\end{equation}
The reduced problem \eqref{eq:entropic_reduced} has a unique minimizer $\bar\gamma_\varepsilon$, and problem \eqref{EWDFO} has a unique minimizer $\gamma_\varepsilon$, which is the lift of $\bar\gamma_\varepsilon$ through \eqref{eq:thm2_explicit_lift}. In particular, under $\gamma_\varepsilon$, the coordinates $X$ and $Y$ are conditionally independent given $w^*(X)$, and $\KL(\gamma_\varepsilon\|\mu\otimes\nu)=\KL(\bar\gamma_\varepsilon\|w^*_{\#}\mu\otimes\nu)$.
\end{proposition}

By Proposition~\ref{prop:entropic_reduction}, problem \eqref{EWDFO} has a unique minimizer $\gamma_\varepsilon$, and we call
\begin{equation*}
\rho(\varepsilon):=\wspo(\mu,\nu;\gamma_\varepsilon)
\end{equation*}
the \emph{entropic regret}, i.e., the regret of the entropy-regularized coupling. As an estimate of the regret of not updating decisions, $\rho(\varepsilon)$ is more interpretable than $\wdfeo(\mu,\nu;\varepsilon)$, which also contains the penalty $\varepsilon\,\KL(\gamma_\varepsilon\|\mu\otimes\nu)$. The following basic order relations formalize the idea that entropy regularization discourages extreme couplings and favors couplings closer to the independent coupling $\mu\otimes\nu$.

\begin{proposition}[Monotonicity and sandwich inequalities]\label{prop:erdf}
For fixed $\mu$ and $\nu$, the map $\varepsilon\mapsto \wdfeo(\mu,\nu;\varepsilon)$ is nondecreasing for $\varepsilon>0$. Moreover, for any $\varepsilon> 0$,
\[
\wdfo(\mu,\nu)\le \wdfeo(\mu,\nu;\varepsilon)\le R(\mu,\nu),
\]
where $R(\mu,\nu)$ is the independent-coupling regret defined in \eqref{eq:regret_def}.
\end{proposition}

The next theorem quantifies how close the regularized divergence is to the optimistic DF divergence and shows that the entropic regret $\rho(\varepsilon)$ moves monotonically between the two endpoints of Proposition~\ref{prop:erdf}. Let $\mathrm{H}(w^*_{\#}\mu):=-\sum_{i}\mu(D_i)\log\mu(D_i)$ denote the Shannon entropy of the push-forward measure, with the convention $0\log 0=0$, so that $\mathrm{H}(w^*_{\#}\mu)\le\log K_\mu\le\log|\mathfrak{S}|$ with $K_\mu$ as in Section~\ref{sec:sample_complexity}.

\begin{theorem}[Entropic approximation and monotone interpolation]\label{thm:entropic}
Let $\gamma_\varepsilon$ be the minimizer of \eqref{EWDFO} and $\rho(\varepsilon)=\wspo(\mu,\nu;\gamma_\varepsilon)$.
\begin{enumerate}
    \item[(i)] For every $\varepsilon>0$, we have $\wdfo(\mu,\nu)\le\wdfeo(\mu,\nu;\varepsilon)\le\wdfo(\mu,\nu)+\varepsilon\,\mathrm{H}(w^*_{\#}\mu)\le\wdfo(\mu,\nu)+\varepsilon\log K_\mu$.
    
    \item[(ii)] The map $\varepsilon\mapsto\rho(\varepsilon)$ is nondecreasing, the map $\varepsilon\mapsto\KL(\gamma_\varepsilon\|\mu\otimes\nu)$ is nonincreasing, and
    \[
    \wdfo(\mu,\nu)\le\rho(\varepsilon)\le\wdfeo(\mu,\nu;\varepsilon)\le R(\mu,\nu),
    \qquad
    \KL(\gamma_\varepsilon\|\mu\otimes\nu)\le\frac{R(\mu,\nu)-\wdfo(\mu,\nu)}{\varepsilon}.
    \]
    Moreover, $\rho(\varepsilon)\to\wdfo(\mu,\nu)$ as $\varepsilon\downarrow0$ and $\rho(\varepsilon)\to R(\mu,\nu)$ as $\varepsilon\to\infty$.
    \item[(iii)] For every $\varepsilon>0$, the coupling $\gamma_\varepsilon$ minimizes the coupled regret among all couplings with at most the same mutual information:
    \[
    \rho(\varepsilon)=\min\Big\{\wspo(\mu,\nu;\gamma):\ \gamma\in\Gamma(\mu,\nu),\ \KL(\gamma\|\mu\otimes\nu)\le\KL(\gamma_\varepsilon\|\mu\otimes\nu)\Big\}.
    \]
\end{enumerate}
\end{theorem}

Theorem~\ref{thm:entropic} turns the regularization parameter into a persistence dial. Part (i) shows that the regularized divergence approximates the optimistic DF divergence within $\varepsilon\log K_\mu$, a bound that does not depend on the dimension $d$ or on the sample size. Part (ii) shows that the entropic regret increases monotonically from the optimistic DF divergence, which corresponds to maximal persistence, to the independent-coupling regret, which corresponds to none, while the dependence retained by the coupling decreases. Part (iii) gives each intermediate value a precise meaning: $\rho(\varepsilon)$ is the smallest regret among couplings whose mutual information between an individual's coefficient vectors before and after the change does not exceed that of $\gamma_\varepsilon$. When some dependence is known to be lost over time, the entropic regret therefore provides an intermediate value that is larger than the optimistic DF divergence and smaller than the independent-coupling regret.
The parameter $\varepsilon$ can be calibrated using linked historical or warm-up data separate from final evaluation.

\section{Extension: Worst-Case Regret}\label{sec:pessimistic}

The optimistic DF divergence is the most favorable regret consistent with the two populations. This section introduces the least favorable one, which bounds the regret of not updating decisions from above.

\begin{definition}[Worst-case regret]\label{def:df_pess}
Given a population for decision making $\mu$ and a population for evaluation $\nu$, the \emph{worst-case regret} is
\begin{equation}
    \rwc(\mu,\nu)
    :=\sup_{\gamma\in\Gamma(\mu,\nu)}\wspo(\mu,\nu;\gamma).
    \tag{Worst-case regret}\label{WDFR}
\end{equation}
\end{definition}

By definition, $\wdfo(\mu,\nu)\le R(\mu,\nu)\le\rwc(\mu,\nu)$. The set of coupled regrets consistent with the two populations is exactly the interval $[\wdfo(\mu,\nu),\rwc(\mu,\nu)]$. The worst-case regret is therefore the largest regret that can occur when the dependence between an individual's coefficient vectors before and after the shift is unknown. A decision maker who wishes to express an attitude toward this ambiguity can use a convex combination $\alpha\,\rwc(\mu,\nu)+(1-\alpha)\,\wdfo(\mu,\nu)$ with $\alpha\in[0,1]$, in the spirit of the $\alpha$-maxmin criterion \citep{ghirardato2004differentiating}.

Unlike the optimistic DF divergence, the worst-case regret between a population and itself is generally not zero: $\rwc(\mu,\mu)\ge R(\mu,\mu)>0$ whenever $\mu$ has a finite first moment and assigns positive mass to the interiors of two distinct decision regions. We therefore call $\rwc$ a regret rather than a divergence. Appendix~\ref{append:pess} states and proves its properties. In particular, the worst-case regret depends on $\mu$ only through $w^*_{\#}\mu$. It admits a quadratic transport formulation that mirrors Theorem~\ref{thm:1}: the worst-case alignment minimizes the expected squared distance between decision vectors and coefficient vectors (Theorem~\ref{thm:pess_w2}). Thus, matching each decision with the coefficient vectors closest to it is the least favorable alignment for a minimization problem, because pairing large decision components with large cost coefficients drives cost up. It also admits an entropy-regularized version (Proposition~\ref{prop:erdf_pess}), and its plug-in estimator satisfies the estimation error bound of Theorem~\ref{ThmSampleComplexity} (Corollary~\ref{cor:pess_estimation}).

\section{Numerical Study}\label{sec:numerical}

We first use synthetic experiments to demonstrate the statistical and computational advantages of the optimistic DF divergence established in Sections~\ref{sec:calculation} and~\ref{sec:sample_complexity}. We then use a newsvendor example to illustrate how the DF divergence captures decision-relevant discrepancies between populations that classical distributional measures fail to capture. Finally, using longitudinal patient profiles from the Parkinson's Telemonitoring dataset, we compare regret estimates under different coupling assumptions and investigate the potential benefits of using DF-based regret estimates to schedule reassessments.

\subsection{Estimation accuracy and computational cost}\label{sec:synthetic_estimation}

\paragraph{Estimation with synthetic data.}
We draw independent samples of size $n$ from each of two bounded continuous distributions and compute the exact plug-in estimates of $\wdfo$ and $W_2^2$. The feasible region is a simplex embedded in coefficient space. We vary dimension $d$ by adding coefficient coordinates whose corresponding decision coordinates are always zero. Separately, we vary $K_\mu$, the number of optimal decisions used, by changing how many regions have positive mass. Appendix~\ref{sec:synthetic_estimation_design} specifies the distributions and a positive-divergence experiment.

\begin{figure}[ht]
    \centering
    \includegraphics[width=\linewidth]{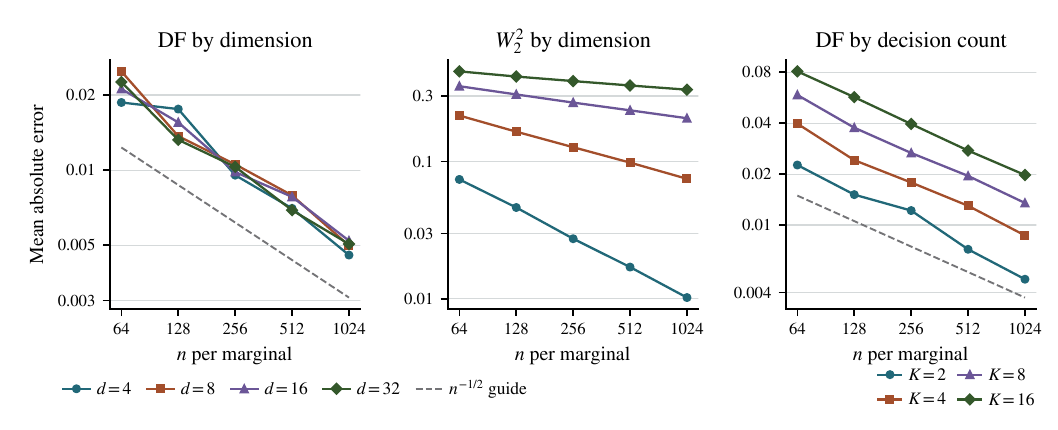}
    \caption{Mean absolute plug-in error over $100$ independent repetitions per setting. The first two panels vary $d$ at $K_\mu=2$. The third varies $K_\mu$ at $d=24$ with $16$ available decisions. The dashed curves have slope $-1/2$ on the log--log axes.}
    \label{fig:estimation}
\end{figure}

Figure~\ref{fig:estimation} shows that the additional decision-irrelevant coordinates have little effect on the estimation error of the optimistic DF divergence. Over the tested sample sizes, the slopes of its estimation error curves on the log--log scale are close to the parametric slope of $-1/2$, consistent with Theorem~\ref{ThmSampleComplexity}. In contrast, the slopes for $W_2^2$ become less negative as the dimension increases, indicating slower convergence due to the curse of dimensionality. The right panel shows that the estimation error of the DF divergence increases with the number of distinct optimal decisions $K_\mu$, consistent with the dependence on $K_\mu$ in our theoretical bound. Additional experimental details and results are provided in Appendix~\ref{append:synthetic_designs}.

\paragraph{Computation.}
We compare the computational efficiency of the direct and reduced empirical DF formulations by solving the corresponding linear programs on identical samples with the same solver settings. For example, when $n_\mu=n_\nu=512$ and $\hat K_\mu=4$, the push-forward reduction decreases the number of decision variables from $262{,}144$ to $2{,}048$, reduces the median solve time from $2.406$ to $0.008$ seconds, and reduces peak memory usage from $419.1$ to $132.4$ MiB. Moreover, the objective values of the two formulations agree within $10^{-9}$ across all benchmark instances, consistent with the exactness of our reduction. Additional computational results are provided in Appendix~\ref{sec:synthetic_computation}.

\subsection{Newsvendor decisions under a change in the customer population}\label{sec:illustrative_examples}

Consider a retailer that determines the order quantity for a store facing three types of customers. Each customer type is associated with a different demand distribution, and the retailer determines the store's order quantity based on the customer type before observing the actual demand. We consider a single product and let $\mathcal Q$ denote the integer order quantities from $0$ to $20$. Let $D_k$ denote the store's demand when facing type-$k$ customers, and let $b$ and $h$ denote the unit shortage and holding costs, respectively. The expected cost of ordering $q$ is
    $c_q^{(k)}:=\EE\big[b(D_k-q)_++h(q-D_k)_+\big], q\in\mathcal Q$.
We represent each order quantity by a vertex of the simplex over $\mathcal Q$ and formulate the downstream linear program as
\begin{equation*}
    \min_{w\geq 0,\ \sum_{q\in\mathcal Q}w_q=1}
    \sum_{q\in\mathcal Q}c_q^{(k)}w_q.
\end{equation*}
We set the unit shortage and holding costs to $3$ and $2$, respectively. An optimal order quantity is the $0.6$ quantile of demand. Figure~\ref{fig:newsvendor} shows the demand distributions associated with the three customer types and the corresponding optimal order quantities of $9$, $10$, and $11$.

\begin{figure}[ht]
    \centering
    \includegraphics[width=0.9\linewidth]{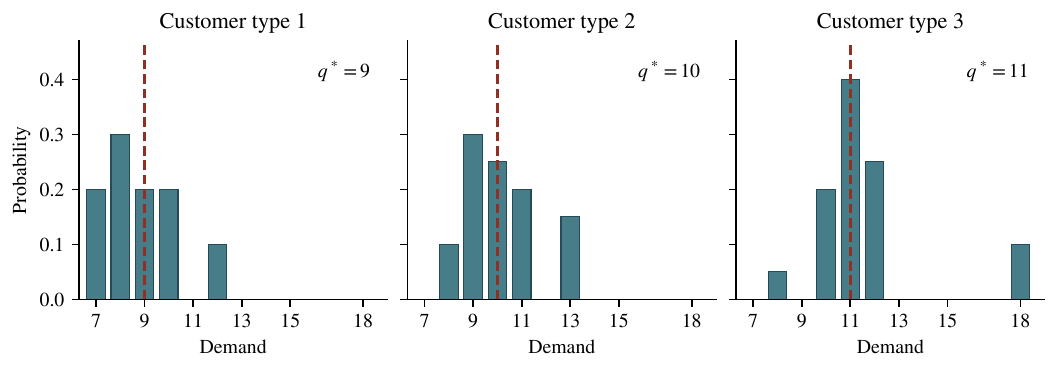}
    \caption{Demand distributions associated with the three customer types. The dashed lines indicate the corresponding optimal order quantities.}
    \label{fig:newsvendor}
\end{figure}

\paragraph{Equal changes in mixture weights.}
Let $\lambda_k$ denote the proportion of type-$k$ customers. The populations for decision making and evaluation are given by $\mu_\lambda$ and $\mu_{\lambda_0}$, respectively, where $\mu_\lambda:=\sum_{k=1}^3\lambda_k\delta_{c^{(k)}}$, with $\lambda_0=(1/3,1/3,1/3)$.
The last two rows of Table~\ref{tab:mixture_distance_to_baseline} consider two population shifts, each involving a decrease of $1/6$ in the proportion of type-2 customers and an equal increase in the proportion of another customer type. Both shifts have the same TV distance and KL divergence.


\begin{table}[ht]
\centering
\begin{tabular}{lrrrrrr}
\hline
$\lambda$ & TV & KL & $W_1$ & $W_2$ & $\wdfo$ & $R$ \\
\hline
$(1/3,1/3,1/3)$ & 0.0000 & 0.0000 & 0.0000 & 0.0000 & 0.0000 & 1.1111 \\
$(1/6,1/2,1/3)$ & 0.1667 & 0.0872 & 2.3970 & 5.8714 & 0.0833 & 0.9306 \\
$(1/3,1/2,1/6)$ & 0.1667 & 0.0872 & 2.5315 & 6.2009 & 0.2917 & 1.1111 \\
\hline
\end{tabular}
\caption{Regret and classical discrepancies between the customer population for decision making $\mu_\lambda$ and the population for evaluation $\mu_{\lambda_0}$. KL denotes $\mathrm{KL}(\mu_\lambda\|\mu_{\lambda_0})$, and $R$ is the independent-coupling regret. Wasserstein distances use the Euclidean ground metric on expected cost vectors in $\mathbb{R}^{21}$.}
\label{tab:mixture_distance_to_baseline}
\end{table}

\paragraph{Changing demand distributions while preserving optimal decisions.}
Starting from the uniform mixture, we increase the largest demand outcome associated with each customer type by six units without changing its probability. The expected cost vectors change substantially, yielding $W_1=8.5843$ and $W_2=8.7265$, while the optimal order quantities for all three customer types remain unchanged. Thus, the optimistic DF divergence is zero by Corollary~\ref{prop:basic}(iii), even though both Wasserstein distances exceed those in Table~\ref{tab:mixture_distance_to_baseline}.

\subsection{Retrospective comparison on Parkinson's telemonitoring profiles}\label{sec:realdata}

We use longitudinal patient profiles to compare regret under different assumptions about how individuals are matched across populations. The Parkinson's Telemonitoring dataset contains $5{,}875$ voice recordings from $42$ patients over a six-month period \citep{tsanaslittle2009uci,tsanas2010telemonitoring}. For each patient profile, we use four features: the motor and total scores of the Unified Parkinson's Disease Rating Scale (UPDRS), which measure motor impairment and overall disease severity, respectively; age; and pitch period entropy (PPE), which measures irregularity in voice pitch. While UPDRS scores are obtained from clinical assessments, PPE is computed directly from each voice recording.

We construct a stylized care-plan problem in which the decision vector $w\in\mathbb{R}^4$ represents the weekly hours allocated to exercise, nursing, rest, and speech therapy, respectively. To construct the objective coefficient vector $y$, we normalize each of the four patient features (motor UPDRS, total UPDRS, age, and PPE) to the interval $[0,1]$ using its minimum and maximum values across the dataset, and then multiply each normalized value by $-1$. This gives $y\in[-1,0]^4$, where a higher feature value corresponds to a greater modeled benefit from allocating time to the corresponding care activity. The downstream problem minimizes $y^\top w$ over the feasible region $S$, which is the convex hull of four candidate care plans, each allocating a total of ten hours per week. Appendix~\ref{sec:parkinsons_setup} provides the details of the candidate plans and the coefficient construction. All regret values reported below are evaluated with respect to this stylized objective.

\paragraph{Patient tracking and comparison couplings.}
We optimize each patient's plan using the day-50 profile and evaluate it at day $t\in\{100,150\}$. For each target day, we select the recording with the closest rounded measurement time within five days. Let $I_t$ contain patients observed in both windows, with profiles $x_i$ at day~50 and $y_{it}$ at day~$t$. The cohorts contain $37$ and $38$ patients, respectively. Matching records by patient identity gives the mean tracked regret
\[
    \widehat R_{\mathrm{tr},t}:=\frac{1}{|I_t|}\sum_{i\in I_t}\lspo(x_i,y_{it}).
\]
All comparison couplings use the same uniform empirical marginals, $\hat\mu_{50\mid t}$ and $\hat\nu_t$. For each coupling $\gamma$, we report mean SPO regret $\wspo(\hat\mu_{50\mid t},\hat\nu_t;\gamma)$ and root-mean-square (RMS) coefficient displacement
\[
    G(\gamma):=\left(\int\|x-y\|_2^2\,\gamma(dx,dy)\right)^{1/2}.
\]
The optimistic DF coupling minimizes regret, the $W_2$ coupling minimizes $G$, and the independent coupling is the product of the marginals. Among DF-optimal couplings, we choose one with the smallest $G$ to make the reported displacement well defined.

For the entropy-regularized optimistic coupling (E-DFO), we choose the KL weight $\varepsilon$ separately at each horizon to match the tracked share of unchanged optimal plans. We solve the reduced problem in Proposition~\ref{prop:entropic_reduction} and lift the coupling to coefficient space. Reported regret excludes the KL penalty in \eqref{EWDFO}; it measures only the loss of keeping the original plans. Appendix~\ref{sec:parkinsons_calibration} gives computational details.

\begin{table}[ht]
\centering
\begin{tabular}{@{}llrrr@{}}
\hline
Transition (patients) & Coupling & RMS displacement & Mean SPO regret & Same plan \\
\hline
$\text{Day } 50\to\text{Day } 100$ (37) & Real patient tracking & 0.1408 & 0.0167 & 89.2\% \\
 & Optimistic DF & 0.2311 & 0.0079 & 91.9\% \\
 & E-DFO ($\varepsilon=0.0772$) & 0.4979 & 0.0143 & 89.2\% \\
 & $W_2$ matching & 0.1342 & 0.0629 & 78.4\% \\
 & Independent & 0.5898 & 0.4956 & 39.4\% \\
\hline
$\text{Day }50\to \text{Day } 150$ (38) & Real patient tracking & 0.1916 & 0.0359 & 86.8\% \\
 & Optimistic DF & 0.2191 & 0.0081 & 89.5\% \\
 & E-DFO ($\varepsilon=0.1055$) & 0.4407 & 0.0210 & 86.8\% \\
 & $W_2$ matching & 0.1633 & 0.0575 & 76.3\% \\
 & Independent & 0.5392 & 0.4631 & 40.0\% \\
\hline
\end{tabular}
\caption{Comparison of couplings for keeping day-50 care plans. At each horizon, all couplings use the same matched cohort and empirical marginals. RMS displacement uses the four normalized coefficients. Same plan is the share of coupled pairs with the same optimal plan; E-DFO is calibrated to match the tracked share.}
\label{tab:parkinsons_matched_regret}
\end{table}

Table~\ref{tab:parkinsons_matched_regret} highlights three observations. (1) The traditional regret evaluation based on independent sampling from the two marginal distributions overestimates the patient-tracked regret by approximately $2{,}868\%$ and $1{,}190\%$ at days~100 and~150, respectively. This substantial overestimation arises from ignoring the persistence of individual patients. Intuitively, a patient's condition tends to evolve gradually from their previous state, so the original care plan is more likely to remain appropriate for the same patient than for a randomly selected patient from the subsequent population. (2) The $W_2$ distance accounts for the coupling between the two populations but matches patients in a decision-blind manner. Consequently, it achieves smaller geometric displacement than patient tracking but yields higher regret at both horizons. This observation illustrates that minimizing geometric displacement does not necessarily recover a coupling that accurately reflects decision regret. (3) The optimistic DF divergence provides a lower bound on the patient-tracked regret, while the entropy-regularized DF coupling, calibrated to the observed share of unchanged optimal plans, yields regret values closer to the patient-tracked regret than either the $W_2$ or independent coupling at both horizons. Appendices~\ref{sec:parkinsons_uncertainty}--\ref{sec:parkinsons_w2_sensitivity} report uncertainty and sensitivity analyses.

\paragraph{Effect of entropy regularization.}
Figure~\ref{fig:parkinsons_kl_calibration} shows the regret $\rho_t(\varepsilon)$ of the entropy-regularized coupling. As $\varepsilon$ increases, the regret increases monotonically from the optimistic DF divergence toward the independent-coupling regret. We adopt an intuitive approach to selecting the regularization parameter $\varepsilon$: we choose its value such that the proportion of patients whose optimal care plans remain unchanged under the entropy-regularized coupling matches the proportion observed under patient tracking. Note that matching these proportions does not necessarily recover the patient-tracked regret, since the expected regret depends not only on how frequently the optimal decisions change but also on the cost of retaining the original decisions when such changes occur.

\begin{figure}[ht]
    \centering
    \includegraphics[width=0.8\linewidth]{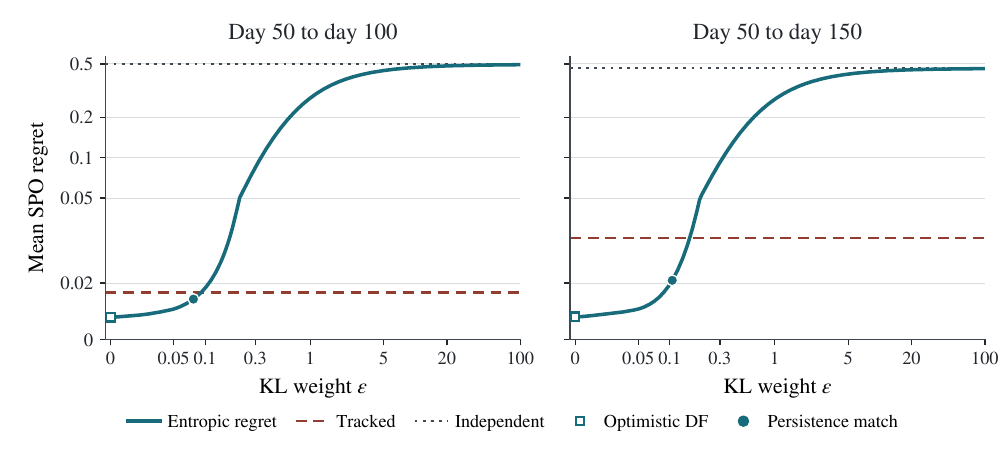}
    \caption{Entropic regret of keeping day-50 care plans. The filled circle in each panel marks the KL weight that matches the share of patients in the linked recording profiles with the same optimal plan. The open square is the unregularized optimistic DF divergence. Both axes are linear near zero and logarithmic at larger values.}
    \label{fig:parkinsons_kl_calibration}
\end{figure}

\subsection{Retrospective comparison of reassessment rules}\label{sec:simulated_reassessment}

Since the independent-coupling regret substantially overestimates the patient-tracked regret in the previous experiment, relying on it may lead to unnecessarily frequent reassessments. To investigate whether DF-based regret estimates can reduce unnecessary reassessments, we examine the trade-off between the cost of reassessing patients and the regret of retaining their original care plans. Using the same patient profiles and downstream objective as in Section~\ref{sec:realdata}, we take enrollment day~10 as $t=0$, construct couplings between the populations at $t=0$ and $t=100$, and consider two candidate reassessment times, $t=50$ and $t=150$. Each time window extends five days on either side of its target. The cohort contains $36$ patients observed in all four windows. In each of $100$ random splits, we use $18$ patients to calibrate the feature normalization and entropy weight, and the remaining $18$ patients for evaluation, with calibration details provided in Appendix~\ref{sec:synthetic_longitudinal_design}.

\paragraph{Estimating regret at candidate reassessment times.}
Each method constructs a coupling $\gamma$ between the evaluation populations at $t=0$ and $t=100$. Let $\gamma_{ij}$ denote the probability mass assigned to the pair of coefficient vectors $x_i$ and $y_j$ at these two endpoints. Following the displacement interpolation discussed in Appendix~\ref{sec:interpolant}, we estimate the intermediate and subsequent coefficient vectors through linear interpolation and extrapolation, respectively. Specifically, the projected coefficient vector and the corresponding estimated mean regret of retaining the baseline plans are
\[
    u_{ij}(t)=\left(1-\frac{t}{100}\right)x_i+\frac{t}{100}y_j,
    \qquad
    \widehat R_\gamma(t)=\sum_{i,j}\gamma_{ij}\,\lspo\big(x_i,u_{ij}(t)\big).
\]
Thus, each coupling induces a distribution of projected coefficient vectors at each candidate reassessment time, over which we evaluate the expected regret. As a reference, patient tracking evaluates the original care plans on the recorded patient profiles at the corresponding times.

\paragraph{Reassessment rules and evaluation.}
Each rule reassesses the evaluation group at most once, at the first candidate time when its estimated mean regret exceeds a tolerance $\tau_s$. The tolerance is common to all rules within split $s$, with a median of $0.0925$ across the $100$ splits. At reassessment, we incur a cost of $c$ per patient and update each patient's care plan using their recorded profile. Let $R_{u,t}$ denote the mean regret at time $t$ of retaining the plans optimized at time $u$, evaluated on the recorded patient profiles. The total cost is $c+R_{50,150}$ if reassessment occurs at $t=50$, $R_{0,50}+c$ if it occurs at $t=150$, and $R_{0,50}+R_{0,150}$ otherwise. To examine the trade-off between reassessment costs and decision regret, we evaluate the same reassessment schedules under two different charges, $c=0.0675$ and $c=0.16$.

\begin{table}[ht]
\centering
\small
\setlength{\tabcolsep}{4pt}
\renewcommand{\arraystretch}{1.05}
\begin{tabular}{@{}lrr@{\hspace{1.4em}}rrr@{\hspace{1.4em}}r@{\hspace{1.4em}}r@{}}
\toprule
 & \multicolumn{2}{c}{\shortstack{Estimated regret of\\keeping baseline plans}}
 & \multicolumn{3}{c}{\shortstack{First reassessment\\(\%)}}
 & \multicolumn{1}{c}{\shortstack{Cost at\\$c=0.0675$}}
 & \multicolumn{1}{c}{\shortstack{Cost at\\$c=0.16$}} \\
\cmidrule(lr){2-3}\cmidrule(lr){4-6}\cmidrule(lr){7-7}\cmidrule(lr){8-8}
Rule & $t=50$ & $t=150$ & $t=50$ & $t=150$ & Neither & Total & Total \\
\midrule
Tracked reference & 0.0834 & 0.1229 & 33 & 15 & 52 & 0.17497 & 0.21937 \\
Calibrated entropy & 0.0133 & 0.3762 & 2 & 91 & 7 & 0.15545 & 0.24148 \\
Optimistic DF & 0.0079 & 0.1411 & 0 & 58 & 42 & 0.17959 & 0.23324 \\
$W_2$ matching & 0.0258 & 0.2448 & 5 & 82 & 13 & 0.16016 & 0.24064 \\
Independent & 0.1320 & 1.3587 & 58 & 42 & 0 & 0.15794 & 0.25044 \\
\midrule
Always at $t=50$ & \textemdash & \textemdash & 100 & 0 & 0 & 0.14791 & 0.24041 \\
Always at $t=150$ & \textemdash & \textemdash & 0 & 100 & 0 & 0.15092 & 0.24342 \\
Never reassess & \textemdash & \textemdash & 0 & 0 & 100 & 0.20636 & 0.20636 \\
\bottomrule
\end{tabular}
\caption{Estimated regret, reassessment frequencies, and mean total costs per patient over $100$ splits. The tracked reference uses recorded profiles, whereas the four coupling-based rules use projected profiles to estimate regret. Dashes indicate fixed schedules that do not require regret estimation. The same reassessment schedules are evaluated under both reassessment charges.}
\label{tab:synthetic_reassessment}
\end{table}

\paragraph{Results.}
Table~\ref{tab:synthetic_reassessment} shows that ignoring patient persistence substantially overestimates regret and leads to more frequent and earlier reassessments. In contrast, the optimistic DF rule reassesses in only $58\%$ of the splits, all at the later candidate time. When reassessment is relatively inexpensive ($c=0.0675$), calibrated entropy achieves the lowest total cost among the coupling-based rules, reducing the cost by approximately $24.7\%$ relative to never reassessing, although always reassessing at $t=50$ performs better. When the reassessment charge increases to $c=0.16$, optimistic DF outperforms the other coupling-based rules, while never reassessing achieves the lowest total cost. These results illustrate how the choice of coupling affects reassessment decisions and how its operational implications depend on the reassessment cost. The comparison is retrospective because the regret estimates at $t=50$ use information from $t=100$. Additional experimental details and sensitivity analyses are provided in Appendix~\ref{sec:synthetic_longitudinal_design}.

\section{Conclusion}\label{sec:Conclusion}

We introduced the optimistic decision-focused divergence to quantify the unavoidable regret of keeping decisions after a population shift in stochastic linear optimization. The independent-coupling regret overstates the true regret when the original decisions perform better on average for their assigned individuals after the shift than for randomly selected individuals from the shifted population. A push-forward reduction yields a semi-discrete quadratic transport formulation and a plug-in estimator with a square-root convergence rate whose exponent does not depend on the coefficient dimension. We established comparisons with classical metrics, Lipschitz stability, and a characterization of DF divergence balls as sets of population shifts compatible with a regret tolerance. Entropy regularization provides intermediate regret values between the optimistic divergence and the independent-coupling regret, while the worst-case regret completes the range consistent with the two populations. Future work includes decision-focused interpolation, optimization over DF divergence balls, robust analysis, and policies for reassessment and re-optimization.

\ACKNOWLEDGMENT{We thank Dr. Soumik Pal for his insightful feedback on this paper during his visit to Chapel Hill. Suhan Liu acknowledges financial support from the National Science Foundation [Grant DMS-2450005].}


\bibliographystyle{informs2014} 
\bibliography{reference} 





%
%
%

\newpage

\begin{APPENDICES}

\OneAndAHalfSpacedXI

\setlength{\abovedisplayskip}{6pt plus 2pt minus 2pt}
\setlength{\belowdisplayskip}{6pt plus 2pt minus 2pt}
\setlength{\abovedisplayshortskip}{3pt plus 1pt minus 1pt}
\setlength{\belowdisplayshortskip}{3pt plus 1pt minus 1pt}
\setlength{\jot}{3pt}

\setlength{\floatsep}{8pt plus 2pt minus 2pt}
\setlength{\textfloatsep}{10pt plus 2pt minus 2pt}
\setlength{\intextsep}{8pt plus 2pt minus 2pt}

 \section{LP Formulation for DF Optimal Transport under Empirical Distributions}\label{append:OT_LP}

Let
\[
\hat{\mu}_{n_\mu}:=\sum_{i=1}^{n_\mu} a_i\,\delta_{x_i},
\qquad
\hat{\nu}_{n_\nu}:=\sum_{j=1}^{n_\nu} b_j\,\delta_{y_j},
\]
where $a_i\ge 0$, $b_j\ge 0$, $\sum_i a_i=\sum_j b_j=1$ (uniform weights correspond to $a_i=1/n_\mu$ and $b_j=1/n_\nu$). Any coupling between $\hat{\mu}_{n_\mu}$ and $\hat{\nu}_{n_\nu}$ can be identified with a nonnegative matrix $\Pi\in\bbR_+^{n_\mu\times n_\nu}$ whose row sums equal $a:=(a_i)$ and whose column sums equal $b:=(b_j)$. The empirical optimistic DF divergence and the empirical worst-case regret are therefore
\begin{align*}
\wdfo(\hat{\mu}_{n_\mu},\hat{\nu}_{n_\nu})
&=\min_{\Pi\ge 0}\ \sum_{i=1}^{n_\mu}\sum_{j=1}^{n_\nu} \lspo(x_i,y_j)\,\Pi_{ij}
\quad\text{s.t.}\quad \Pi\one=a,\ \Pi^\top\one=b, \notag\\
\rwc(\hat{\mu}_{n_\mu},\hat{\nu}_{n_\nu})
&=\max_{\Pi\ge 0}\ \sum_{i=1}^{n_\mu}\sum_{j=1}^{n_\nu} \lspo(x_i,y_j)\,\Pi_{ij}
\quad\text{s.t.}\quad \Pi\one=a,\ \Pi^\top\one=b.
\end{align*}
These linear programs have $n_\mu n_\nu$ variables and are tractable only when both sample sizes are small. By Lemma~\ref{lemma:cons}, both values coincide with the values of the reduced transportation problems between the distribution of optimal decisions $\sum_{k}\big(\sum_{i:\,w^*(x_i)=w_k}a_i\big)\delta_{w_k}$ and $\hat\nu_{n_\nu}$ with cost $c_k(y_j)$, which have at most $\min\{n_\mu,|\mathfrak{S}|\}\,n_\nu$ variables (Algorithm~\ref{alg:dfot}). When the sample sizes are large or when $\nu$ is continuous, we therefore use the methods of Section~\ref{sec:calculation}.

\section{Additional Results and Proofs}\label{append:proofs}

\subsection{Additional results and proofs for Section~\ref{sec:Method}}\label{append:proofs_method}

\begin{proof}{\bfseries Proof of \eqref{eq:spo_cross}--\eqref{eq:spo_lipschitz}}
The first two properties follow from the optimality of $w^*(x)$ for $x$ and of $w^*(y)$ for $y$ together with the Cauchy--Schwarz inequality, and the third holds because $\lspo(x,y)=\max_{w\in S}y^\top\big(w^*(x)-w\big)$ is a maximum of linear functions of $y$ whose gradients have norm at most $D_W$.\hfill\Halmos
\end{proof}

\begin{proof}{\bfseries Proof of Theorem \ref{thm:stability}}
For attainment, Lemma~\ref{lemma:cons} with $h(y)=y$ yields the reduced problem \eqref{eq:reduced}. The set $\Gamma(w^*_{\#}\mu,\nu)$ is weakly compact, and the cost $c(w_i,y)=c_i(y)$ is nonnegative and continuous on the finite set $\mathfrak{S}$ times $\bbR^d$, because $z$ is continuous. Hence, $\bar\gamma\mapsto\int c\,d\bar\gamma$ is lower semicontinuous and attains its minimum over $\Gamma(w^*_{\#}\mu,\nu)$ \citep[Chapter~4]{villani2009optimal}. Lifting a reduced minimizer by the construction in Lemma~\ref{lemma:cons} therefore attains the infimum in \eqref{WDFO}.

(i) Let $\gamma'\in\Gamma(\mu,\nu')$ be optimal for $\wdfo(\mu,\nu')$, whose existence was established above, and let $\pi\in\Gamma(\nu',\nu)$ be an optimal coupling for $W_1(\nu',\nu)$. By the gluing lemma \citep[Chapter~1]{villani2009optimal}, there exist random vectors $(X,Y',Y)$ with $(X,Y')\sim\gamma'$ and $(Y',Y)\sim\pi$. The law of $(X,Y)$ belongs to $\Gamma(\mu,\nu)$, so \eqref{eq:spo_lipschitz} gives
\[
\begin{aligned}
\wdfo(\mu,\nu)
&\le\EE\big[\lspo(X,Y)\big]\\
&\le\EE\big[\lspo(X,Y')\big]+D_W\,\EE\|Y-Y'\|_2\\
&=\wdfo(\mu,\nu')+D_W\,W_1(\nu,\nu').
\end{aligned}
\]
Exchanging the roles of $\nu$ and $\nu'$ proves (i).

(ii) Let $p:=w^*_{\#}\mu$ and $p':=w^*_{\#}\mu'$. By Lemma~\ref{lemma:cons} and attainment established above, take an optimal reduced coupling $(W,Y)$ with marginals $p$ and $\nu$. Couple $W$ with $W'\sim p'$ by matching their common mass at each decision, so that $\bP(W\ne W')=\tv(p,p')$, and glue this coupling with $(W,Y)$ over $W$. Then $(W',Y)$ is feasible for the reduced problem associated with $\mu'$ and $\nu$. Since
\[
c(W',Y)-c(W,Y)=Y^\top(W'-W)
\le D_Wr_{\mathcal{Y}}\,\one\{W'\ne W\},
\]
we obtain
\[
\wdfo(\mu',\nu)\le\wdfo(\mu,\nu)+D_Wr_{\mathcal{Y}}\,\tv(p,p').
\]
Interchanging $\mu$ and $\mu'$ proves (ii).

(iii) Combine (i) and (ii) through the triangle inequality
\[
\begin{aligned}
|\wdfo(\mu,\nu)-\wdfo(\mu',\nu')|
&\le|\wdfo(\mu,\nu)-\wdfo(\mu',\nu)|\\
&\quad+|\wdfo(\mu',\nu)-\wdfo(\mu',\nu')|.
\end{aligned}
\]
\hfill\Halmos
\end{proof}

\begin{proof}{\bfseries Proof of Corollary \ref{prop:basic}}
(i) Nonnegativity follows from \eqref{eq:spo_cross}, the inequality $\wdfo(\mu,\nu)\le R(\mu,\nu)$ follows because $\mu\otimes\nu\in\Gamma(\mu,\nu)$, and $R(\mu,\nu)\le D_W\int\|y\|_2\,\nu(dy)$ follows from \eqref{eq:spo_bound}. Attainment follows from Theorem~\ref{thm:stability}.

(ii) If $\nu$ has bounded support, the assertion follows from Theorem~\ref{thm:stability}(ii). For general $\nu$, let $T_R(y):=y\one\{\|y\|_2\le R\}$ and $\nu_R:=(T_R)_{\#}\nu$. Then
\[
W_1(\nu,\nu_R)\le\int_{\{\|y\|_2>R\}}\|y\|_2\,\nu(dy)\longrightarrow0.
\]
The bounded-support case gives $\wdfo(\mu,\nu_R)=\wdfo(\mu',\nu_R)$, and Theorem~\ref{thm:stability}(i) yields
\[
\big|\wdfo(\mu,\nu)-\wdfo(\mu',\nu)\big|
\le2D_WW_1(\nu,\nu_R)\longrightarrow0.
\]

(iii) If $w^*_{\#}\mu=w^*_{\#}\nu$, then (ii) and the diagonal coupling give $\wdfo(\mu,\nu)=\wdfo(\nu,\nu)=0$. Conversely, suppose $\wdfo(\mu,\nu)=0$ and $\nu(\mathcal{T})=0$. Let $(X,Y)$ have an optimal coupling, which exists by Theorem~\ref{thm:stability}. Nonnegativity of the loss implies $\lspo(X,Y)=0$ almost surely, so $w^*(X)$ is optimal for $Y$. Since the minimizer at $Y$ is unique almost surely, $w^*(X)=w^*(Y)$ almost surely, and hence $w^*_{\#}\mu=w^*_{\#}\nu$.\hfill\Halmos
\end{proof}

\begin{proof}{\bfseries Proof of Remark \ref{rem:not_metric}}
Take $x,y\in\bbR^d$ with $\lspo(x,y)>0$ and $\kappa>1$. Because every coupling of two Dirac measures is their product, $\wdfo(\delta_x,\delta_y)=\lspo(x,y)$. By scale invariance of the oracle and positive homogeneity of the SPO loss in its second argument, $\wdfo(\delta_x,\delta_{\kappa y})=\kappa\,\lspo(x,y)$ and $\wdfo(\delta_y,\delta_{\kappa y})=0$, so
\[
\wdfo(\delta_x,\delta_{\kappa y})=\kappa\,\lspo(x,y)>\lspo(x,y)+0=\wdfo(\delta_x,\delta_y)+\wdfo(\delta_y,\delta_{\kappa y}).
\]
The same Diracs violate the triangle inequality for the symmetrization $\wdfo(\mu,\nu)+\wdfo(\nu,\mu)$, because $\wdfo(\delta_{\kappa y},\delta_x)=\wdfo(\delta_y,\delta_x)$.\hfill\Halmos
\end{proof}

\begin{proof}{\bfseries Proof of Proposition \ref{thm:rescale_wdfo_zero}}
Let $\gamma$ be the joint law of $(X,Y)=(X,\kappa X)$; then $\gamma\in\Gamma(\mu,\nu)$. For any realization, if
$\kappa>0$ then
\[
\arg\min_{w\in S} w^\top (\kappa X)
=\arg\min_{w\in S} \kappa\, w^\top X
=\arg\min_{w\in S} w^\top X,
\]
so by the scale-invariance of the oracle selection we have $w^*(Y)=w^*(X)$. If $\kappa=0$, then $Y=0$ and
$\lspo(X,Y)=0^\top w^*(X)-0^\top w^*(0)=0$. Hence $\lspo(X,Y)=0$ almost surely, and therefore
$\wspo(\mu,\nu;\gamma)=\EE[\lspo(X,Y)]=0$.

Finally, $\lspo(x,y)\ge 0$ for all $x,y$ because $w^*(y)$ minimizes $y^\top w$ over $S$, implying
$y^\top w^*(y)\le y^\top w^*(x)$. Taking the infimum over $\Gamma(\mu,\nu)$ yields $\wdfo(\mu,\nu)=0$.\hfill\Halmos
\end{proof}

\subsubsection{Why the independent-coupling regret overstates the regret}\label{sec:persistence}

The next proposition decomposes the coupled regret into the independent-coupling regret and a covariance term that captures persistence of decisions.

\begin{proposition}[Persistence decomposition]\label{prop:cov_decomp}
Let $\int\|y\|_2\,\nu(dy)<\infty$. For every coupling $\gamma\in\Gamma(\mu,\nu)$ and $(X,Y)\sim\gamma$,
\begin{equation}\label{eq:cov_decomp}
\wspo(\mu,\nu;\gamma)
=R(\mu,\nu)+\sum_{i=1}^{|\mathfrak{S}|}\mathrm{Cov}_\gamma\big(\one\{X\in D_i\},\,c_i(Y)\big).
\end{equation}
Consequently,
\[
\wdfo(\mu,\nu)=R(\mu,\nu)+\min_{\gamma\in\Gamma(\mu,\nu)}\sum_{i=1}^{|\mathfrak{S}|}\mathrm{Cov}_\gamma\big(\one\{X\in D_i\},c_i(Y)\big).
\]
Moreover, if $\int\|x\|_2\,\mu(dx)<\infty$ and $\mu(D_i^\circ)>0$ and $\mu(D_j^\circ)>0$ for some $i\neq j$, then $R(\mu,\mu)>0$, whereas $\wdfo(\mu,\mu)=0$.
\end{proposition}

We call the sum in \eqref{eq:cov_decomp} the \emph{persistence covariance}. It is negative when, after the shift, the original decisions perform better on average for the individuals who received them than for randomly selected individuals from the shifted population. We refer to this tendency as persistence of decisions. The independent-coupling regret drops this term and therefore overstates the regret exactly when the persistence covariance is negative; it even reports positive regret when the population does not change. The optimistic DF divergence makes the persistence covariance as negative as the two populations allow, i.e., it corresponds to maximal persistence. The independent-coupling regret is not an upper bound in general: if individuals' coefficient vectors systematically move to regions in which their earlier decisions are relatively bad, the persistence covariance is positive and the regret exceeds $R(\mu,\nu)$. The worst-case regret in Section~\ref{sec:pessimistic} provides the guaranteed upper bound.

\begin{proof}{\bfseries Proof of Proposition \ref{prop:cov_decomp}}
By \eqref{eq:spo_ci}, $\lspo(X,Y)=\sum_i\one\{X\in D_i\}\,c_i(Y)$, and each $c_i(Y)$ is integrable because $0\le c_i(y)\le D_W\|y\|_2$ by \eqref{eq:spo_bound}. For each $i$,
\[
\EE_\gamma\big[\one\{X\in D_i\}\,c_i(Y)\big]=\mathrm{Cov}_\gamma\big(\one\{X\in D_i\},c_i(Y)\big)+\mu(D_i)\,\EE_\nu\big[c_i(Y)\big].
\]
Summing over $i$ and noting that $\sum_i\mu(D_i)\EE_\nu[c_i(Y)]=\EE_{\mu\otimes\nu}[\lspo(X,Y)]=R(\mu,\nu)$ yields \eqref{eq:cov_decomp}. Taking the infimum over $\gamma$ and using the attainment in Corollary~\ref{prop:basic}(i) yields the expression for $\wdfo(\mu,\nu)$.

For the last statement, $R(\mu,\mu)=\sum_k\mu(D_k)\int c_k\,d\mu\ge\mu(D_i)\int_{D_j^\circ}c_i\,d\mu$, because all terms are nonnegative. For $x\in D_j^\circ$, the extreme point $w_j\neq w_i$ is the unique minimizer, so $c_i(x)=x^\top w_i-z(x)>0$. Since $\mu(D_i)\ge\mu(D_i^\circ)>0$ and $\mu(D_j^\circ)>0$, the right-hand side is positive. Finally, $\wdfo(\mu,\mu)=0$ by Corollary~\ref{prop:basic}(iii).\hfill\Halmos
\end{proof}

\subsection{Additional results and proofs for Section~\ref{sec:calculation}}\label{append:proofs_calculation}

\begin{lemma}[Coupling reduction via oracle push-forward]\label{lemma:cons}
Fix probability measures $\mu,\nu$ on $\bbR^d$, and let $h:\bbR^d\to\bbR^d$ be measurable with $\int\|h(y)\|_2\,\nu(dy)<\infty$. Then
\[
\inf_{\gamma\in\Gamma(\mu,\nu)} \int h(y)^\top w^*(x)\,d\gamma(x,y)
=
\inf_{\bar{\gamma}\in\Gamma(w^*_{\#}\mu,\nu)} \int h(y)^\top w\,d\bar{\gamma}(w,y),
\]
and the same identity holds with both infima replaced by suprema. Moreover, every $\bar\gamma\in\Gamma(w^*_{\#}\mu,\nu)$ equals $(w^*\times\mathrm{Id})_{\#}\gamma$ for some $\gamma\in\Gamma(\mu,\nu)$.
\end{lemma}

\begin{proof}{\bfseries Proof of Lemma \ref{lemma:cons}}
Let $\ff:\bbR^d\times\bbR^d\to S\times\bbR^d$ be $\ff(x,y):=(w^*(x),y)$. The integrals below are finite because $|h(y)^\top w|\le\|h(y)\|_2\max_{w'\in S}\|w'\|_2$ for $w\in S$ and $\int\|h\|_2\,d\nu<\infty$. For any $\gamma\in\Gamma(\mu,\nu)$, the push-forward $\bar{\gamma}:=\ff_{\#}\gamma$ satisfies $\bar{\gamma}\in\Gamma(w^*_{\#}\mu,\nu)$, and by the change-of-variables formula,
\[
\int h(y)^\top w^*(x)\,d\gamma(x,y)=\int h(y)^\top w\,d\bar{\gamma}(w,y).
\]
Hence, the infimum on the left of the claimed identity is at least the infimum on the right, and the supremum on the left is at most the supremum on the right.

For the reverse inequalities, fix any $\bar{\gamma}\in\Gamma(w^*_{\#}\mu,\nu)$. For each $k$ with $\mu(D_k)>0$, let $\mu_k:=\mu(\,\cdot\cap D_k)/\mu(D_k)$ and let $\bar\gamma_k:=\bar\gamma(\{w_k\}\times\cdot\,)$, a measure on $\bbR^d$ with total mass $\mu(D_k)$. Define
\begin{equation}\label{gamma}
\gamma:=\sum_{k:\,\mu(D_k)>0}\mu_k\otimes\bar\gamma_k .
\end{equation}
For measurable $A,B\subseteq\bbR^d$, we have $\gamma(A\times\bbR^d)=\sum_{k:\,\mu(D_k)>0}\mu(A\cap D_k)=\mu(A)$ and $\gamma(\bbR^d\times B)=\sum_{k:\,\mu(D_k)>0}\bar\gamma(\{w_k\}\times B)=\nu(B)$, where the last equality uses $\bar\gamma(\{w_k\}\times B)\le\mu(D_k)=0$ for the omitted indices. Hence, $\gamma\in\Gamma(\mu,\nu)$. Moreover, $\gamma(D_k\times B)=\bar\gamma(\{w_k\}\times B)$ for every $k$, so $\ff_{\#}\gamma=\bar{\gamma}$, which proves the last claim of the lemma. Consequently,
\[
\int h(y)^\top w^*(x)\,d\gamma(x,y)=\int h(y)^\top w\,d\bar{\gamma}(w,y),
\]
and taking the infimum or the supremum over $\bar{\gamma}$ yields the reverse inequalities.\hfill\Halmos
\end{proof}

\begin{proof}{\bfseries Proof of Theorem \ref{thm:1}}
By the reduced problem \eqref{eq:reduced},
\[
\wdfo(\mu,\nu)=\inf_{\bar\gamma\in\Gamma(w^*_{\#}\mu,\nu)}\int y^\top w\,d\bar\gamma(w,y)-\int z\,d\nu .
\]
For every $w,y\in\bbR^d$, the polarization identities give
\[
y^\top w=\tfrac12\Big(\|w\|_2^2+\|y\|_2^2-\|w-y\|_2^2\Big)
=\tfrac12\Big(\|w+y\|_2^2-\|w\|_2^2-\|y\|_2^2\Big).
\]
For every $\bar\gamma\in\Gamma(w^*_{\#}\mu,\nu)$, the integrals of $\|w\|_2^2$ and $\|y\|_2^2$ equal $\|w^*_{\#}\mu\|^2$ and $\|\nu\|^2$, which are finite because $S$ is bounded and $\nu$ has a finite second moment. Hence, minimizing $\int y^\top w\,d\bar\gamma$ over $\Gamma(w^*_{\#}\mu,\nu)$ is equivalent to maximizing $\int\|w-y\|_2^2\,d\bar\gamma$, which proves (a)$\Leftrightarrow$(b) and the first line in \eqref{DWDt}. It is also equivalent to minimizing $\int\|w+y\|_2^2\,d\bar\gamma=\int\|(-w)-y\|_2^2\,d\bar\gamma$. Under the reflection $w\mapsto -w$, the set $\Gamma(w^*_{\#}\mu,\nu)$ is mapped one-to-one onto $\Gamma((\bar w^*)_{\#}\mu,\nu)$, so the minimum equals $W_2^2((\bar w^*)_{\#}\mu,\nu)$ and is attained exactly at reflections of the optimal plans in (a). This proves (a)$\Leftrightarrow$(c) and the second line in \eqref{DWDt}, where we use $\int z\,d\nu=\int y^\top w^*(y)\,d\nu(y)$.

For any invertible matrix $A\in\bbR^{d\times d}$, $(Ay)^\top(A^{-\top}w)=y^\top w$, so replacing $(y,w)$ with $(Ay,A^{-\top}w)$ leaves the optimal couplings unchanged.\hfill\Halmos
\end{proof}

\begin{proof}{\bfseries Proof of Theorem \ref{thm:2}}
The measure defined by \eqref{eq:thm2_explicit_lift} is $\gamma=\sum_{k:\,\mu(D_k)>0}\mu_k\otimes\bar\gamma^*(\{w_k\}\times\cdot\,)$, i.e., the construction \eqref{gamma} in the proof of Lemma~\ref{lemma:cons} applied to $\bar\gamma^*$. Hence, $\gamma\in\Gamma(\mu,\nu)$, \eqref{eq:thm2_explicit_rect} holds, $(w^*\times\mathrm{Id})_{\#}\gamma=\bar\gamma^*$, and the conditional independence of $X$ and $Y$ given $w^*(X)$ follows from the product form. Finally, by \eqref{eq:spo_ci} and the change-of-variables formula,
\[
\wspo(\mu,\nu;\gamma)=\int\big(y^\top w^*(x)-z(y)\big)\,d\gamma(x,y)=\int c(w,y)\,d\bar\gamma^*(w,y)=\wdfo(\mu,\nu),
\]
where the last equality uses the optimality of $\bar\gamma^*$ for \eqref{eq:reduced}.\hfill\Halmos
\end{proof}

\begin{proof}{\bfseries Proof of Proposition \ref{prop:dual_wdfo}}
By Definition~\ref{def:df_dist} and $\lspo(x,y)=y^\top w^*(x)-y^\top w^*(y)$,
\[
\wdfo(\mu,\nu)
=\inf_{\gamma\in\Gamma(\mu,\nu)}\int_{\bbR^d\times\bbR^d}\big(y^\top w^*(x)-z(y)\big)\,d\gamma(x,y).
\]
Applying Lemma~\ref{lemma:cons} with $h(y)=y$ yields the reduction
\[
\wdfo(\mu,\nu)
=\inf_{\bar\gamma\in\Gamma(w^*_{\#}\mu,\nu)}\int_{\mathfrak{S}\times\bbR^d}\big(y^\top w-z(y)\big)\,d\bar\gamma(w,y).
\]
Since $w^*_{\#}\mu$ is supported on the finite set $\mathfrak{S}$, we may write
$w^*_{\#}\mu=\sum_{i=1}^{|\mathfrak{S}|}\mu(D_i)\,\delta_{w_i}$, and the above becomes an optimal transport problem
between the discrete measure $w^*_{\#}\mu$ and $\nu$ with cost $c(w_i,y)=c_i(y)$.

The Kantorovich duality theorem for optimal transport (which applies here since the first space is finite and the
cost is measurable and integrable) gives
\[
\inf_{\bar\gamma\in\Gamma(w^*_{\#}\mu,\nu)}\int c(w,y)\,d\bar\gamma(w,y)
=
\sup_{\varphi,\psi}\left\{\int \varphi(w)\,d(w^*_{\#}\mu)(w)+\int \psi(y)\,d\nu(y):\ \varphi(w)+\psi(y)\le c(w,y)\right\},
\]
where $\varphi$ is a function on $\mathfrak{S}$ and $\psi\in L^1(\nu)$. Identifying $\varphi(w_i)$ with $f_i$ and
$\psi$ with $g$ yields the constraints \eqref{eq:PhiSPO} and the dual representation \eqref{eq:WDFO_dual}.

Finally, for any fixed $f$, the constraint \eqref{eq:PhiSPO} implies $g(y)\le \min_i(c_i(y)-f_i)$ pointwise. Since
the objective in \eqref{eq:WDFO_dual} is increasing in $g$, the best feasible choice is $g=g_f$ as defined in
\eqref{eq:g_star}, which gives \eqref{eq:WDFO_dual_reduced}.\hfill\Halmos
\end{proof}

\begin{corollary}[Price-adjusted reassignment]\label{cor:price_shift}
Suppose that $\nu$ is supported on a bounded set. Then the supremum in \eqref{eq:WDFO_dual_reduced} is attained at some $f^*\in\bbR^{|\mathfrak{S}|}$, and the following statements hold.
\begin{enumerate}
    \item[(i)] Every optimal plan $\bar\gamma^*$ for \eqref{eq:reduced} satisfies, for $\bar\gamma^*$-almost every $(w_i,y)$,
    \begin{equation}\label{eq:price_rule}
    y^\top w_i-f^*_i=\min_{1\le j\le|\mathfrak{S}|}\big(y^\top w_j-f^*_j\big).
    \end{equation}
    Equivalently, $\bar\gamma^*$ assigns $y$ to decision $w_i$ only if $y$ lies in the price-adjusted region $A_i(f^*):=\{y\in\bbR^d:\ y^\top(w_i-w_j)\le f^*_i-f^*_j\ \text{for all } j\}$.
    \item[(ii)] If $w^*_{\#}\mu=w^*_{\#}\nu$, then $f^*=0$ is optimal, and $A_i(0)$ coincides with $D_i$ up to the tie set $\mathcal{T}$.
    \item[(iii)] For every probability vector $p'$ on $\mathfrak{S}$, writing $\wdfo(p',\nu)$ for the divergence of any population whose optimal decisions have distribution $p'$, we have $\wdfo(p',\nu)\ge\wdfo(\mu,\nu)+\sum_i f^*_i\big(p'_i-\mu(D_i)\big)$.
\end{enumerate}
\end{corollary}

Corollary~\ref{cor:price_shift} gives a geometric description of the optimal reassignment. The price-adjusted regions $A_i(f^*)$ have the same facet normals $w_i-w_j$ as the decision regions $D_i$, which coincide with $A_i(0)$ up to ties, so the optimal reassignment keeps the orientation of the original decision rule and shifts each boundary in parallel by the price difference $f^*_i-f^*_j$. Intuitively, the adjustments make a decision more attractive when its share under $\mu$ exceeds the share that the original decision rule would assign to it under $\nu$, and less attractive in the opposite case, until each decision's share matches its share under $\mu$. When the two populations induce the same distribution of optimal decisions, no shift is needed. By part (iii), the prices are subgradients of the divergence with respect to the distribution of optimal decisions, so $f^*_i-f^*_j$ bounds from below the rate at which the divergence changes when a unit of share moves from decision $w_j$ to decision $w_i$.

\begin{proof}{\bfseries Proof of Corollary \ref{cor:price_shift}}
Let $\mathcal{Y}$ be a bounded set containing the support of $\nu$, let $r_{\mathcal{Y}}:=\sup_{y\in\mathcal{Y}}\|y\|_2$, and let $\bar B:=D_W\,r_{\mathcal{Y}}$. By \eqref{eq:spo_bound}, $0\le c_i(y)\le\bar B$ for all $i$ and $y\in\mathcal{Y}$. Write $p_i:=\mu(D_i)$ and $J(f):=\sum_if_ip_i+\int g_f\,d\nu$. The argument in Step~1 of the proof of Theorem~\ref{ThmSampleComplexity}, applied with all indices and with $B$ replaced by $\bar B$, shows that the supremum of $J$ over $\bbR^{|\mathfrak{S}|}$ equals its supremum over the compact set $\{f:\ \max_if_i=0,\ \min_if_i\ge-\bar B\}$. Because $|g_f-g_{f'}|\le\|f-f'\|_\infty$, the function $J$ is continuous, so it attains its supremum at some $f^*$.

(i) Let $\bar\gamma^*$ be optimal for \eqref{eq:reduced}, which exists by Corollary~\ref{prop:basic}(i). Because $\bar\gamma^*$ has marginals $w^*_{\#}\mu$ and $\nu$, strong duality (Proposition~\ref{prop:dual_wdfo}) gives
\[
\int\big(c(w,y)-f^*(w)-g_{f^*}(y)\big)\,d\bar\gamma^*(w,y)=\int c\,d\bar\gamma^*-J(f^*)=\wdfo(\mu,\nu)-\wdfo(\mu,\nu)=0,
\]
where $f^*(w_i):=f^*_i$. The integrand is nonnegative on $\mathfrak{S}\times\mathcal{Y}$ by the definition of $g_{f^*}$, so it vanishes $\bar\gamma^*$-almost everywhere, i.e., $c_i(y)-f^*_i=\min_j(c_j(y)-f^*_j)$. Substituting $c_i(y)=y^\top w_i-z(y)$ and canceling the common term $-z(y)$ yields \eqref{eq:price_rule}, which is equivalent to $y\in A_i(f^*)$.

(ii) For every $y$, $\min_ic_i(y)=0$, so $J(0)=\int\min_ic_i\,d\nu=0$. If $w^*_{\#}\mu=w^*_{\#}\nu$, then $\wdfo(\mu,\nu)=0$ by Corollary~\ref{prop:basic}(iii), so $f^*=0$ is optimal. Moreover, $A_i(0)=\{y:\ w_i\in\arg\min_{w\in S}w^\top y\}$ contains $D_i$, and $A_i(0)\setminus D_i\subseteq\mathcal{T}$.

(iii) Proposition~\ref{prop:dual_wdfo} applies to every distribution of optimal decisions $p'$ with the same function $g_f$, so $\wdfo(p',\nu)\ge\sum_if^*_ip'_i+\int g_{f^*}\,d\nu=J(f^*)+\sum_if^*_i(p'_i-p_i)=\wdfo(\mu,\nu)+\sum_if^*_i(p'_i-p_i)$.\hfill\Halmos
\end{proof}

\subsection{Proofs for Section~\ref{sec:sample_complexity}}

\begin{proof}{\bfseries Proof of Theorem \ref{ThmSampleComplexity}}
If $B=0$, both values vanish. Otherwise, we proceed in three steps.

\noindent\emph{Step 1: Dual representations on a normalized set.}
The push-forward $w^*_{\#}\mu$ is supported on $\{w_i:\ i\in I_\mu\}$, and so is $w^*_{\#}\hat\mu_{n_\mu}$ almost surely, because $\mu(D_i)=0$ implies that no sample falls in $D_i$ almost surely. Applying Proposition~\ref{prop:dual_wdfo} with $\mathfrak{S}$ replaced by $\{w_i:\ i\in I_\mu\}$, whose proof only requires that this finite set contain the support of $w^*_{\#}\mu$, we may write
\begin{equation}\label{eq:dual_pop}
\wdfo(\mu,\nu)
=
\sup_{f\in\bbR^{I_\mu}}
\left\{\sum_{i\in I_\mu} f_i\,\mu(D_i)+\EE\big[g_f(Y)\big]\right\},
\end{equation}
where $Y\sim\nu$ and $g_f(y):=\min_{i\in I_\mu}(c_i(y)-f_i)$. Likewise,
\begin{equation}\label{eq:dual_emp}
\wdfo(\hat\mu_{n_\mu},\hat\nu_{n_\nu})
=
\sup_{f\in\bbR^{I_\mu}}
\left\{\sum_{i\in I_\mu} f_i\,\hat\mu_{n_\mu}(D_i)+\frac{1}{n_\nu}\sum_{j=1}^{n_\nu} g_f(Y_j)\right\}.
\end{equation}
Both objectives are invariant under the shift $f\mapsto f+t\one$, because $\sum_{i\in I_\mu} \mu(D_i)=\sum_{i\in I_\mu} \hat\mu_{n_\mu}(D_i)=1$, so we may assume $\max_i f_i=0$. Suppose that $f_i<-B$ for some $i$, and let $j$ satisfy $f_j=0$. For every $y\in\mathcal{Y}$, $c_i(y)-f_i>c_i(y)+B\ge B\ge c_j(y)-f_j$, so the $i$th term never attains the minimum in $g_f(y)$, and raising $f_i$ to $-B$ leaves $g_f$ unchanged on $\mathcal{Y}$ while weakly increasing $\sum_i f_i\mu(D_i)$ and $\sum_i f_i\hat\mu_{n_\mu}(D_i)$. Because $\nu$ and $\hat\nu_{n_\nu}$ are supported on $\mathcal{Y}$, we may therefore restrict the suprema in \eqref{eq:dual_pop}--\eqref{eq:dual_emp} to
\[
\mathcal{F}:=\Big\{f\in\bbR^{I_\mu}:\ \max_{i\in I_\mu} f_i=0,\ \min_{i\in I_\mu} f_i\ge -B\Big\}
\]
without changing either value. For $f\in\mathcal{F}$ and $y\in\mathcal{Y}$, we have $|f_i|\le B$ and $0\le\min_i c_i(y)\le g_f(y)\le c_j(y)-f_j\le B$, where $j$ satisfies $f_j=0$.

\smallskip
\noindent\emph{Step 2: Decomposing the estimation error.}
For each $f\in\mathcal{F}$, define
\[
\hat J(f):=\sum_{i\in I_\mu} f_i\,\hat\mu_{n_\mu}(D_i)+\frac{1}{n_\nu}\sum_{j=1}^{n_\nu} g_f(Y_j),
\qquad
J(f):=\sum_{i\in I_\mu} f_i\,\mu(D_i)+\EE[g_f(Y)],
\]
so that $\wdfo(\hat\mu_{n_\mu},\hat\nu_{n_\nu})=\sup_{f\in\mathcal{F}} \hat J(f)$ and $\wdfo(\mu,\nu)=\sup_{f\in\mathcal{F}} J(f)$. Using $\sup_f \hat J(f)-\sup_f J(f)\le \sup_f(\hat J(f)-J(f))$ and the triangle inequality,
\begin{align*}
\wdfo(\hat\mu_{n_\mu},\hat\nu_{n_\nu})-\wdfo(\mu,\nu)
&\le \sup_{f\in\mathcal{F}}\big(\hat J(f)-J(f)\big)\\
&\le \sup_{f\in\mathcal{F}}
\left|\sum_{i\in I_\mu} f_i\big(\hat\mu_{n_\mu}(D_i)-\mu(D_i)\big)\right|
+\sup_{f\in\mathcal{F}}
\left|\frac{1}{n_\nu}\sum_{j=1}^{n_\nu} g_f(Y_j)-\EE[g_f(Y)]\right|.
\end{align*}
The same argument applied to $\sup_f J(f)-\sup_f \hat J(f)\le\sup_f(J(f)-\hat J(f))$ yields the same upper bound for the reverse difference. Hence,
\begin{equation}\label{eq:gap_decomp}
\Big|\wdfo(\hat\mu_{n_\mu},\hat\nu_{n_\nu})-\wdfo(\mu,\nu)\Big|
\le \Delta_\mu+\Delta_\nu,
\end{equation}
where
\[
\Delta_\mu:=\sup_{f\in\mathcal{F}}
\left|\sum_{i\in I_\mu} f_i\big(\hat\mu_{n_\mu}(D_i)-\mu(D_i)\big)\right|,
\qquad
\Delta_\nu:=\sup_{f\in\mathcal{F}}
\left|\frac{1}{n_\nu}\sum_{j=1}^{n_\nu} g_f(Y_j)-\EE[g_f(Y)]\right|.
\]

\smallskip
\noindent\emph{Step 3: Concentration for $\Delta_\mu$ and $\Delta_\nu$.}
For $\Delta_\mu$, let $p_i:=\mu(D_i)$ and $\hat p_i:=\hat\mu_{n_\mu}(D_i)$ for $i\in I_\mu$. Since $\sum_i(\hat p_i-p_i)=0$ and $|f_i+B/2|\le B/2$ for $f\in\mathcal{F}$,
\[
\Delta_\mu
\le\frac{B}{2}\|\hat p-p\|_1,
\qquad
\EE[\Delta_\mu]
\le\frac{B}{2}\sqrt{K_\mu}\,\Big(\EE\|\hat p-p\|_2^2\Big)^{1/2}.
\]
Since $\hat p_i=\frac{1}{n_\mu}\sum_{t=1}^{n_\mu}\one\{X_t\in D_i\}$ is unbiased for $p_i$ with $\Var(\hat p_i)=p_i(1-p_i)/n_\mu$,
\[
\EE\|\hat p-p\|_2^2
=\sum_{i\in I_\mu}\Var(\hat p_i)
=\frac{1}{n_\mu}\sum_{i\in I_\mu}p_i(1-p_i)
\le \frac{1}{n_\mu},
\]
and therefore
\begin{equation}\label{eq:Emu_bound}
\EE[\Delta_\mu]\le\frac{B}{2}\sqrt{\frac{K_\mu}{n_\mu}}.
\end{equation}
Bounding each coordinate separately gives the refinement
\[
\EE[\Delta_\mu]
\le\frac{B}{2\sqrt{n_\mu}}\sum_{i\in I_\mu}\sqrt{p_i(1-p_i)}
\le\frac{B}{2\sqrt{n_\mu}}\sum_{i\in I_\mu}\sqrt{p_i}.
\]
Thus, $K_\mu$ can be replaced by $(\sum_i\sqrt{p_i})^2\le K_\mu$ in the source-sample term of \eqref{eq:sample_complexity_bound}; the target-sample term below is unchanged. Replacing one source observation changes each linear functional by at most $|f_a-f_b|/n_\mu\le B/n_\mu$, where $a$ and $b$ are its old and new decision indices. Taking absolute values and the supremum preserves this bound. McDiarmid's inequality yields that with probability at least $1-\delta/2$,
\begin{equation}\label{eq:mu_tail}
\Delta_\mu \le \EE[\Delta_\mu] + B\sqrt{\frac{\log(2/\delta)}{2n_\mu}}.
\end{equation}

For $\Delta_\nu$, consider the class $\mathcal{G}:=\{g_f:\ f\in\mathcal{F}\}$ of functions on $\mathcal{Y}$ with values in $[0,B]$. Because $g_f$ is a minimum of functions that are $1$-Lipschitz in $f$, we have $\sup_{y}|g_f(y)-g_{f'}(y)|\le\|f-f'\|_\infty$. Covering the cube $[-B,0]^{I_\mu}\supseteq\mathcal{F}$ by a grid shows that the covering number of $\mathcal{G}$ in the supremum norm satisfies $N(\tau,\mathcal{G},\|\cdot\|_\infty)\le (3B/\tau)^{K_\mu}$ for $0<\tau\le B$. By symmetrization \citep[Chapter~4]{wainwright2019high},
\[
\EE[\Delta_\nu]\le 2\,\EE\Big[\sup_{f\in\mathcal{F}}\Big|\frac{1}{n_\nu}\sum_{j=1}^{n_\nu}\sigma_j g_f(Y_j)\Big|\Big],
\]
where $\sigma_1,\ldots,\sigma_{n_\nu}$ are independent Rademacher variables. Fix $f_0\in\mathcal{F}$. The supremum on the right is at most $|\frac{1}{n_\nu}\sum_j\sigma_jg_{f_0}(Y_j)|$, whose expectation is at most $B/\sqrt{n_\nu}$, plus the supremum of the increments $|\frac{1}{n_\nu}\sum_j\sigma_j(g_f-g_{f_0})(Y_j)|$, which Dudley's entropy integral bound \citep[Chapter~5]{wainwright2019high} controls by
\[
\frac{C_0}{\sqrt{n_\nu}}\int_0^{B}\sqrt{\log N(\tau,\mathcal{G},\|\cdot\|_\infty)}\,d\tau
\le\frac{C_0}{\sqrt{n_\nu}}\,B\sqrt{K_\mu}\int_0^1\sqrt{\log(3/u)}\,du
\le C_1B\sqrt{\frac{K_\mu}{n_\nu}}
\]
for universal constants $C_0,C_1>0$. Since $K_\mu\ge1$, combining the two terms gives
\begin{equation}\label{eq:Enu_bound}
\EE[\Delta_\nu]\le C\,B\sqrt{\frac{K_\mu}{n_\nu}}
\end{equation}
for a universal constant $C>0$. Moreover, changing one sample $Y_t$ changes the empirical average $\frac{1}{n_\nu}\sum_{j=1}^{n_\nu} g_f(Y_j)$ by at most $B/n_\nu$ uniformly over $f$, hence changes $\Delta_\nu$ by at most $B/n_\nu$. McDiarmid's inequality yields that with probability at least $1-\delta/2$,
\begin{equation}\label{eq:nu_tail}
\Delta_\nu \le \EE[\Delta_\nu] + B\sqrt{\frac{\log(2/\delta)}{2n_\nu}}.
\end{equation}
Combining \eqref{eq:gap_decomp} with \eqref{eq:Emu_bound}--\eqref{eq:nu_tail} and applying a union bound over the events in \eqref{eq:mu_tail} and \eqref{eq:nu_tail} yields \eqref{eq:sample_complexity_bound}, without requiring independence between the samples.

Finally, the transport value is jointly convex in its marginals, so Jensen's inequality gives $\EE[\wdfo(\hat\mu_{n_\mu},\hat\nu_{n_\nu})]\ge\wdfo(\mu,\nu)$. This does not require independence between the two samples.\hfill\Halmos
\end{proof}

\subsection{Additional results and proofs for Section~\ref{sec:property}}\label{append:proofs_property}

\begin{proof}{\bfseries  Proof of Proposition \ref{prop:lips_w1}}
The inequality $\wdfo(\mu,\nu)\le W^O_{\mathrm{sym}}(\mu,\nu)$ is immediate since
$W^O_{\mathrm{sym}}(\mu,\nu)=\wdfo(\mu,\nu)+\wdfo(\nu,\mu)$ and $\wdfo(\nu,\mu)\ge 0$.

To bound $W^O_{\mathrm{sym}}$, fix any coupling $\gamma\in\Gamma(\mu,\nu)$ and let $\gamma^{\top}$ denote its
transpose. By feasibility of $\gamma$ and $\gamma^{\top}$,
\[
\begin{aligned}
\wdfo(\mu,\nu)
&\le \int \lspo(x,y)\,\gamma(dx,dy),\\
\wdfo(\nu,\mu)
&\le \int \lspo(x,y)\,\gamma^{\top}(dx,dy)
 = \int \lspo(y,x)\,\gamma(dx,dy).
\end{aligned}
\]
Adding the two inequalities and using $\lspo(x,y)+\lspo(y,x)=(x-y)^\top(w^*(y)-w^*(x))$, we obtain
\[
\begin{aligned}
W^O_{\mathrm{sym}}(\mu,\nu)
&\le \int (x-y)^\top\!\big(w^*(y)-w^*(x)\big)\,\gamma(dx,dy)\\
&\le \int \|x-y\|_2\,\|w^*(y)-w^*(x)\|_2\,\gamma(dx,dy).
\end{aligned}
\]
Since $w^*(x),w^*(y)\in S$, we have $\|w^*(y)-w^*(x)\|_2\le D_W$, and therefore
\[
W^O_{\mathrm{sym}}(\mu,\nu)\le D_W\int \|x-y\|_2\,\gamma(dx,dy).
\]
Taking the infimum over $\gamma\in\Gamma(\mu,\nu)$ yields
$W^O_{\mathrm{sym}}(\mu,\nu)\le D_W\,W_1(\mu,\nu)$.

Finally, on a bounded set $\mathcal{Y}$, we have $W_1(\mu,\nu)\le \diam(\mathcal{Y})\,\tv(\mu,\nu)$, completing the proof. As a consequence, Pinsker's inequality gives
\[
W^O_{\mathrm{sym}}(\mu,\nu)\le D_W\,\diam(\mathcal{Y})\sqrt{\tfrac12\,\KL(\mu\|\nu)}
\]
whenever $\KL(\mu\|\nu)<\infty$.\hfill\Halmos
\end{proof}

\begin{proof}{\bfseries Proof of Proposition \ref{prop:lips}}
By definition,
\[
\wdfo(\mu,\nu)= \inf_{\gamma\in\Gamma(\mu,\nu)}\int y^\top\!\big(w^*(x)-w^*(y)\big)\,d\gamma(x,y).
\]
Applying Cauchy--Schwarz to the integrand yields
\[
\wdfo(\mu,\nu)\le
\inf_{\gamma\in\Gamma(\mu,\nu)}
\left(\int \|y\|_2^2\,d\gamma(x,y)\right)^{1/2}
\left(\int\|w^*(x)-w^*(y)\|_2^2\,d\gamma(x,y)\right)^{1/2}.
\]
The first factor depends only on the second marginal of $\gamma$, hence
\[
\left(\int \|y\|_2^2\,d\gamma(x,y)\right)^{1/2}
=
\left(\int\|y\|_2^2 \,d\nu(y)\right)^{1/2}
=:V.
\]

To handle the second factor, push any $\gamma\in\Gamma(\mu,\nu)$ forward through the oracle map in both
coordinates. Specifically, let $\tilde{\gamma}:=(w^*\times w^*)_{\#}\gamma$. Then
$\tilde{\gamma}\in\Gamma(w^*_{\#}\mu,w^*_{\#}\nu)$ and
\[
\int\|w^*(x)-w^*(y)\|_2^2\,d\gamma(x,y)
=
\int\|u-v\|_2^2\,d\tilde{\gamma}(u,v).
\]
Conversely, given any $\tilde{\gamma}\in\Gamma(w^*_{\#}\mu,w^*_{\#}\nu)$, one can construct a coupling
$\gamma\in\Gamma(\mu,\nu)$ whose oracle-level push-forward equals $\tilde{\gamma}$ by sampling
$(U,V)\sim\tilde{\gamma}$ and then, conditional on $(U,V)=(w_i,w_j)$, sampling
$X\sim \mu(\,\cdot\,|D_i)$ and $Y\sim \nu(\,\cdot\,|D_j)$. This construction ensures
$(w^*(X),w^*(Y))\sim\tilde{\gamma}$ and preserves the above quadratic cost identity. Therefore,
\begin{equation} \label{eq:pf1}
\left(\inf_{\gamma\in\Gamma(\mu,\nu)}\int\|w^*(x)-w^*(y)\|_2^2\,d\gamma(x,y)\right)^{1/2}
=
W_2(w^*_{\#}\mu,w^*_{\#}\nu).
\end{equation}
Combining the pieces gives $\wdfo(\mu,\nu)\le V\,W_2(w^*_{\#}\mu,w^*_{\#}\nu)$.

To upper bound $W_2(w^*_{\#}\mu,w^*_{\#}\nu)$ in terms of $\tv$, match the common mass at each atom $w_i$
at zero cost, and transport the remaining mass (of total size $\tv(w^*_{\#}\mu,w^*_{\#}\nu)$) within $S$ at
cost at most $D_W^2$ per unit. This yields
\[
W_2^2(w^*_{\#}\mu,w^*_{\#}\nu)\le D_W^2\,\tv(w^*_{\#}\mu,w^*_{\#}\nu),
\quad\text{hence}\quad
W_2(w^*_{\#}\mu,w^*_{\#}\nu)\le D_W\sqrt{\tv(w^*_{\#}\mu,w^*_{\#}\nu)}.
\]
Pinsker's inequality $\tv(P,Q)\le \sqrt{\tfrac12\,\KL(P\|Q)}$ gives the KL control.

For \eqref{eq:linear_tv}, suppose that $\nu$ is supported on $\mathcal{Y}$, write $p_i:=\mu(D_i)$ and $q_i:=\nu(D_i)$, and let $\nu_i:=\nu(\,\cdot\cap D_i)/q_i$ whenever $q_i>0$. We construct a plan in $\Gamma(w^*_{\#}\mu,\nu)$ for the reduced problem \eqref{eq:reduced}. For each $i$, assign the mass $m_i:=\min\{p_i,q_i\}$ of $m_i\nu_i$ to the atom $w_i$; because $c_i(y)=0$ for $y\in D_i$, this part costs nothing. The unassigned masses of the atoms, $(p_i-m_i)_i$, and the unassigned part of $\nu$, $\sum_i(q_i-m_i)\nu_i$, both have total mass $\sum_i(p_i-q_i)_+=\tv(w^*_{\#}\mu,w^*_{\#}\nu)$, and we couple them by any coupling, for example their normalized product. Since $c_i(y)\le D_W\|y\|_2\le D_W\,r_{\mathcal{Y}}$ for $y\in\mathcal{Y}$ by \eqref{eq:spo_bound}, this part costs at most $D_W\,r_{\mathcal{Y}}\,\tv(w^*_{\#}\mu,w^*_{\#}\nu)$. By \eqref{eq:reduced}, $\wdfo(\mu,\nu)$ is at most the total cost of this plan, and Pinsker's inequality gives the KL version.\hfill\Halmos
\end{proof}

\begin{proof}{\bfseries Proof of Theorem \ref{thm:df_ball}}
(i) Let $\nu\in\mathcal{B}_{\mathrm{DF}}(\mu,\Delta)$. By Corollary~\ref{prop:basic}(i), the infimum in \eqref{WDFO} is attained at some $\gamma^*\in\Gamma(\mu,\nu)$, and $\wspo(\mu,\nu;\gamma^*)=\wdfo(\mu,\nu)\le\Delta$. Hence, $\mathcal{B}_{\mathrm{DF}}(\mu,\Delta)$ is $\Delta$-tolerable. Conversely, let $\mathcal{U}$ be $\Delta$-tolerable and $\nu\in\mathcal{U}$. There exists $\gamma\in\Gamma(\mu,\nu)$ with $\wspo(\mu,\nu;\gamma)\le\Delta$, so $\wdfo(\mu,\nu)\le\wspo(\mu,\nu;\gamma)\le\Delta$ and $\nu\in\mathcal{B}_{\mathrm{DF}}(\mu,\Delta)$.

(ii) For convexity, let $\nu_1,\nu_2\in\mathcal{B}_{\mathrm{DF}}(\mu,\Delta)$ and $\lambda\in[0,1]$, and let $\gamma_m\in\Gamma(\mu,\nu_m)$ be optimal for $\wdfo(\mu,\nu_m)$, $m=1,2$. Then $\lambda\gamma_1+(1-\lambda)\gamma_2\in\Gamma(\mu,\lambda\nu_1+(1-\lambda)\nu_2)$, so
\[
\wdfo\big(\mu,\lambda\nu_1+(1-\lambda)\nu_2\big)\le\lambda\,\wdfo(\mu,\nu_1)+(1-\lambda)\,\wdfo(\mu,\nu_2)\le\Delta.
\]
Closedness with respect to $W_1$ holds because $\nu\mapsto\wdfo(\mu,\nu)$ is $D_W$-Lipschitz with respect to $W_1$ by Theorem~\ref{thm:stability}(i). The dependence on $\mu$ only through its distribution of optimal decisions is Corollary~\ref{prop:basic}(ii). By Proposition~\ref{prop:dual_wdfo}, $\wdfo(\mu,\nu)$ is the supremum over $f\in\bbR^{|\mathfrak{S}|}$ of the left-hand side of the inequality in \eqref{eq:df_ball_dual}, so $\wdfo(\mu,\nu)\le\Delta$ if and only if every $f$ satisfies this inequality. The statement about $\mathcal{B}_{\mathrm{DF}}(\mu,0)$ is Corollary~\ref{prop:basic}(iii).

(iii) By \eqref{eq:spo_cross}, every $\gamma\in\Gamma(\mu,\nu)$ satisfies $\wspo(\mu,\nu;\gamma)\le D_W\int\|x-y\|_2\,d\gamma(x,y)$, and taking the infimum over $\gamma$ gives $\wdfo(\mu,\nu)\le D_W\,W_1(\mu,\nu)$, which proves the first inclusion. The total variation statement follows from Proposition~\ref{prop:lips_w1}, and the KL statement also uses Pinsker's inequality. For the last claim, let $d=1$ and $S=[0,D]$ with $D>0$, so that $D_W=D$, $w^*(x)=0$ for $x>0$, and $w^*(x)=D$ for $x<0$. Let $\mu=\delta_{\eta}$ and $\nu=\delta_{-r}$ with $\eta:=\epsilon\Delta/(2D)$ and $r:=(1+\epsilon/2)\Delta/D$. Then $W_1(\mu,\nu)=\eta+r=(1+\epsilon)\Delta/D_W$, whereas $\wdfo(\mu,\nu)=\lspo(\eta,-r)=rD=(1+\epsilon/2)\Delta>\Delta$.

(iv) Let $X\sim\mu$ and, for $\kappa>0$, let $\nu_\kappa$ denote the law of $\kappa X$. Then $\nu_\kappa\in\cP_1(\bbR^d)$ and $\wdfo(\mu,\nu_\kappa)=0$ by Proposition~\ref{thm:rescale_wdfo_zero}, so $\nu_\kappa\in\mathcal{B}_{\mathrm{DF}}(\mu,0)$. Because the Euclidean norm is $1$-Lipschitz, every pair $(X',Y')$ with $X'\sim\mu$ and $Y'\sim\nu_\kappa$ satisfies $\EE\|X'-Y'\|_2\ge\big|\EE\|Y'\|_2-\EE\|X'\|_2\big|$, so
\[
W_1(\mu,\nu_\kappa)\ge|\kappa-1|\int\|x\|_2\,\mu(dx),
\]
which tends to infinity as $\kappa\to\infty$ because $\int\|x\|_2\,\mu(dx)>0$ when $\mu(\{0\})=0$. For the total variation distance, let $A_\rho:=\{x:\ \|x\|_2\le\rho\}$ for $\rho>0$. Then $\tv(\mu,\nu_\kappa)\ge\mu(A_\rho)-\nu_\kappa(A_\rho)=\mu(A_\rho)-\mu(A_{\rho/\kappa})$, and $\mu(A_{\rho/\kappa})\to\mu(\{0\})=0$ as $\kappa\to\infty$. Hence, $\sup_{\kappa>0}\tv(\mu,\nu_\kappa)\ge\mu(A_\rho)$ for every $\rho>0$, and letting $\rho\to\infty$ shows that the supremum equals one. The last statement follows because $\mathcal{B}_{\mathrm{DF}}(\mu,0)\subseteq\mathcal{B}_{\mathrm{DF}}(\mu,\Delta)$.\hfill\Halmos
\end{proof}

\begin{proof}{\bfseries Proof of Proposition \ref{prop:entropic_reduction}}
Write $p:=w^*_{\#}\mu$ and $\ff(x,y):=(w^*(x),y)$, and note that $\ff_{\#}(\mu\otimes\nu)=p\otimes\nu$. For any $\gamma\in\Gamma(\mu,\nu)$, let $\bar\gamma:=\ff_{\#}\gamma\in\Gamma(p,\nu)$. By \eqref{eq:spo_ci} and the change-of-variables formula, $\int\lspo\,d\gamma=\int c\,d\bar\gamma$, and by the data-processing inequality for the KL divergence, $\KL(\bar\gamma\|p\otimes\nu)\le\KL(\gamma\|\mu\otimes\nu)$. Hence, the infimum in \eqref{EWDFO} is at least the infimum in \eqref{eq:entropic_reduced}.

Conversely, let $\bar\gamma\in\Gamma(p,\nu)$ satisfy $\KL(\bar\gamma\|p\otimes\nu)<\infty$, and let $\gamma$ be its lift \eqref{gamma}. For each $k$ with $p_k:=\mu(D_k)>0$, the measure $\bar\gamma(\{w_k\}\times\cdot\,)$ is absolutely continuous with respect to $\nu$, and $r(w_k,y):=\frac{1}{p_k}\,\frac{d\bar\gamma(\{w_k\}\times\cdot\,)}{d\nu}(y)$ is the density of $\bar\gamma$ with respect to $p\otimes\nu$. By the product form of the lift, the density of $\gamma$ with respect to $\mu\otimes\nu$ is $(x,y)\mapsto r(w^*(x),y)$. Therefore,
\begin{align*}
\KL(\gamma\|\mu\otimes\nu)
&=\sum_{k:\,p_k>0}\int_{D_k}\int r(w_k,y)\log r(w_k,y)\,\nu(dy)\,\mu(dx)\\
&=\sum_{k:\,p_k>0}p_k\int r(w_k,y)\log r(w_k,y)\,\nu(dy)=\KL(\bar\gamma\|p\otimes\nu),
\end{align*}
and $\int\lspo\,d\gamma=\int c\,d\bar\gamma$ by the same change of variables. Hence, the infimum in \eqref{eq:entropic_reduced} is at least the infimum in \eqref{EWDFO}, which proves \eqref{eq:entropic_reduced}.

The set $\Gamma(p,\nu)$ is weakly compact, $\bar\gamma\mapsto\int c\,d\bar\gamma$ is lower semicontinuous because $c$ is continuous and nonnegative \citep[Chapter~4]{villani2009optimal}, and $\bar\gamma\mapsto\KL(\bar\gamma\|p\otimes\nu)$ is lower semicontinuous and strictly convex on its domain. The objective in \eqref{eq:entropic_reduced} is finite at $\bar\gamma=p\otimes\nu$, where it equals $R(\mu,\nu)<\infty$. Hence, \eqref{eq:entropic_reduced} has a unique minimizer $\bar\gamma_\varepsilon$. By the previous paragraph, its lift $\gamma_\varepsilon$ attains the infimum in \eqref{EWDFO} and satisfies $\KL(\gamma_\varepsilon\|\mu\otimes\nu)=\KL(\bar\gamma_\varepsilon\|p\otimes\nu)$, and the conditional independence follows from the product form of the lift. Finally, the objective in \eqref{EWDFO} is strictly convex on the set where it is finite, because $\gamma\mapsto\int\lspo\,d\gamma$ is linear and the KL divergence is strictly convex, so $\gamma_\varepsilon$ is its unique minimizer.\hfill\Halmos
\end{proof}

Proposition~\ref{prop:entropic_reduction} also clarifies how the entropy-regularized coupling reassigns decisions. Under standard conditions, for example when $\nu$ has bounded support, the reduced minimizer has the Gibbs form $d\bar\gamma_\varepsilon/d(w^*_{\#}\mu\otimes\nu)(w_i,y)=\exp\big((f_i+g(y)-c_i(y))/\varepsilon\big)$ for some potentials $(f,g)$ \citep{peyre2019computational}. With $p_i:=\mu(D_i)$, an individual whose coefficient vector under $\nu$ is $y$ therefore keeps decision $w_i$ with conditional probability
\[
\frac{p_i\exp\big((f_i-c_i(y))/\varepsilon\big)}{\sum_{j}p_j\exp\big((f_j-c_j(y))/\varepsilon\big)},
\]
a multinomial-logit assignment with decision-specific adjustments that preserve the decision shares. As $\varepsilon\to\infty$, the minimizer approaches the independent coupling (Theorem~\ref{thm:entropic}), i.e., random assignment in proportions $p$. As $\varepsilon\downarrow0$, the assignment becomes increasingly concentrated and, under standard conditions on the convergence of the potentials, approaches the price-adjusted rule \eqref{eq:price_rule}.

\begin{proof}{\bfseries Proof of Proposition \ref{prop:erdf}}
Fix $0<\varepsilon<\varepsilon'$ and any $\gamma\in\Gamma(\mu,\nu)$. Since $\KL(\gamma\|\mu\otimes\nu)\ge 0$,
\[
\int \lspo(x,y)\,d\gamma(x,y)+\varepsilon'\KL(\gamma\|\mu\otimes\nu)
\;\ge\;
\int \lspo(x,y)\,d\gamma(x,y)+\varepsilon\KL(\gamma\|\mu\otimes\nu).
\]
Taking the infimum over $\gamma\in\Gamma(\mu,\nu)$ shows that $\wdfeo(\mu,\nu;\varepsilon')\ge \wdfeo(\mu,\nu;\varepsilon)$. Because the KL term is nonnegative, $\wdfeo(\mu,\nu;\varepsilon)\ge \wdfo(\mu,\nu)$. On the other hand, choosing the independent coupling $\gamma=\mu\otimes\nu$, for which $\KL(\gamma\|\mu\otimes\nu)=0$, in \eqref{EWDFO} yields $\wdfeo(\mu,\nu;\varepsilon)\le R(\mu,\nu)$.\hfill\Halmos
\end{proof}

\begin{proof}{\bfseries Proof of Theorem \ref{thm:entropic}}
Write $p:=w^*_{\#}\mu$ and $\kappa_\varepsilon:=\KL(\gamma_\varepsilon\|\mu\otimes\nu)$.

(i) The lower bound is Proposition~\ref{prop:erdf}. For the upper bound, let $\bar\gamma^*$ be optimal for \eqref{eq:reduced} and let $\gamma^*$ be its lift, so that $\wspo(\mu,\nu;\gamma^*)=\wdfo(\mu,\nu)$ by Theorem~\ref{thm:2}. By the proof of Proposition~\ref{prop:entropic_reduction}, $\KL(\gamma^*\|\mu\otimes\nu)=\KL(\bar\gamma^*\|p\otimes\nu)$. For each $k$ with $p_k>0$, let $q_k(y):=d\bar\gamma^*(\{w_k\}\times\cdot\,)/d\nu(y)$, so that $\sum_kq_k(y)=1$ for $\nu$-almost every $y$ and $\int q_k\,d\nu=p_k$. Then
\[
\KL(\bar\gamma^*\|p\otimes\nu)=\int\sum_{k}q_k(y)\log\frac{q_k(y)}{p_k}\,\nu(dy)=-\int\Big(-\sum_kq_k(y)\log q_k(y)\Big)\nu(dy)+\mathrm{H}(p)\le\mathrm{H}(p),
\]
because the Shannon entropy of the probability vector $(q_k(y))_k$ is nonnegative. Evaluating the objective of \eqref{EWDFO} at $\gamma^*$ gives $\wdfeo(\mu,\nu;\varepsilon)\le\wdfo(\mu,\nu)+\varepsilon\,\mathrm{H}(p)$, and $\mathrm{H}(p)\le\log K_\mu$ because $p$ has $K_\mu$ atoms.

(ii) Let $0<\varepsilon_1<\varepsilon_2$, and write $\rho_m:=\rho(\varepsilon_m)$ and $\kappa_m:=\kappa_{\varepsilon_m}$. Optimality of $\gamma_{\varepsilon_1}$ and $\gamma_{\varepsilon_2}$ gives
\[
\rho_1+\varepsilon_1\kappa_1\le\rho_2+\varepsilon_1\kappa_2,
\qquad
\rho_2+\varepsilon_2\kappa_2\le\rho_1+\varepsilon_2\kappa_1.
\]
Adding the two inequalities yields $(\varepsilon_2-\varepsilon_1)(\kappa_2-\kappa_1)\le0$, so $\kappa_2\le\kappa_1$, and the first inequality then gives $\rho_1-\rho_2\le\varepsilon_1(\kappa_2-\kappa_1)\le0$. The chain $\wdfo\le\rho(\varepsilon)\le\wdfeo(\mu,\nu;\varepsilon)\le R$ follows because $\gamma_\varepsilon$ is a coupling, because $\kappa_\varepsilon\ge0$, and by Proposition~\ref{prop:erdf}. Comparing the objective at $\gamma_\varepsilon$ with its value $R(\mu,\nu)$ at $\mu\otimes\nu$ gives $\rho(\varepsilon)+\varepsilon\kappa_\varepsilon\le R(\mu,\nu)$, hence $\kappa_\varepsilon\le(R(\mu,\nu)-\wdfo(\mu,\nu))/\varepsilon$.

As $\varepsilon\downarrow0$, part (i) gives $\wdfo(\mu,\nu)\le\rho(\varepsilon)\le\wdfeo(\mu,\nu;\varepsilon)\le\wdfo(\mu,\nu)+\varepsilon\log K_\mu$, so $\rho(\varepsilon)\to\wdfo(\mu,\nu)$. As $\varepsilon\to\infty$, $\kappa_\varepsilon\to0$, and Pinsker's inequality gives $\tv(\gamma_\varepsilon,\mu\otimes\nu)\le\sqrt{\kappa_\varepsilon/2}\to0$. Because $\mu$ and $\nu$ are supported on $\mathcal{Y}$, the SPO loss takes values in $[0,D_W\,\diam(\mathcal{Y})]$ on the supports by \eqref{eq:spo_cross}, so $|\rho(\varepsilon)-R(\mu,\nu)|\le D_W\,\diam(\mathcal{Y})\,\tv(\gamma_\varepsilon,\mu\otimes\nu)\to0$.

(iii) If some $\gamma\in\Gamma(\mu,\nu)$ satisfied $\KL(\gamma\|\mu\otimes\nu)\le\kappa_\varepsilon$ and $\wspo(\mu,\nu;\gamma)<\rho(\varepsilon)$, then the objective of \eqref{EWDFO} at $\gamma$ would be strictly smaller than $\rho(\varepsilon)+\varepsilon\kappa_\varepsilon=\wdfeo(\mu,\nu;\varepsilon)$, a contradiction. Since $\gamma_\varepsilon$ is feasible for the constrained problem, the minimum is attained at $\gamma_\varepsilon$.\hfill\Halmos
\end{proof}

\section{Results for the Worst-Case Regret}\label{append:pess}

This appendix states and proves the properties of the worst-case regret $\rwc$ introduced in Section~\ref{sec:pessimistic}. We first collect its basic properties.

\begin{proposition}[Properties of the worst-case regret]\label{prop:pess}
Let $\int\|y\|_2\,\nu(dy)<\infty$.
\begin{enumerate}
    \item[(i)] $\wdfo(\mu,\nu)\le R(\mu,\nu)\le\rwc(\mu,\nu)\le D_W\int\|y\|_2\,\nu(dy)$, the supremum in \eqref{WDFR} is attained, and
    \[
    \rwc(\mu,\nu)=R(\mu,\nu)+\max_{\gamma\in\Gamma(\mu,\nu)}\sum_{i=1}^{|\mathfrak{S}|}\mathrm{Cov}_\gamma\big(\one\{X\in D_i\},c_i(Y)\big).
    \]
    \item[(ii)] $\rwc(\mu,\nu)$ depends on $\mu$ only through $w^*_{\#}\mu$.
    \item[(iii)] If $\int\|x\|_2\,\mu(dx)<\infty$ and $\mu(D_i^\circ)>0$ and $\mu(D_j^\circ)>0$ for some $i\neq j$, then $\rwc(\mu,\mu)\ge R(\mu,\mu)>0$. Hence, $\rwc$ is not a divergence, and we call it a regret.
    \item[(iv)] For fixed $\nu$, write $\wdfo(p,\nu)$ and $\rwc(p,\nu)$ for the values at any population whose optimal decisions have distribution $p$. Then $p\mapsto\rwc(p,\nu)$ is concave and $p\mapsto\wdfo(p,\nu)$ is convex on the probability simplex over $\mathfrak{S}$.
\end{enumerate}
\end{proposition}

Part (iv) explains why $\rwc$ behaves differently from a divergence. As a concave function of the distribution of optimal decisions, it is minimized at a vertex of the simplex, i.e., at a population that assigns every individual the same decision, rather than at $w^*_{\#}\nu$.

\begin{proof}{\bfseries Proof of Proposition \ref{prop:pess}}
(i) The inequality $R\le\rwc$ holds because $\mu\otimes\nu\in\Gamma(\mu,\nu)$, and $\rwc(\mu,\nu)\le D_W\int\|y\|_2\,\nu(dy)$ follows from \eqref{eq:spo_bound}. By Lemma~\ref{lemma:cons},
\[
\rwc(\mu,\nu)=\sup_{\bar\gamma\in\Gamma(w^*_{\#}\mu,\nu)}\int c\,d\bar\gamma.
\]
The function $-c$ is continuous and bounded below by the $\nu$-integrable function $-D_W\|y\|_2$, so $\bar\gamma\mapsto\int(-c)\,d\bar\gamma$ is lower semicontinuous on the weakly compact set $\Gamma(w^*_{\#}\mu,\nu)$ \citep[Chapter~4]{villani2009optimal}. Hence, the supremum is attained, and the lift of a maximizer (Theorem~\ref{thm:2}, whose proof applies verbatim to any reduced plan) attains the supremum in \eqref{WDFR}. The covariance representation follows from \eqref{eq:cov_decomp} by taking the maximum over $\gamma$.

(ii) This follows from the reduced representation in the proof of part (i).

(iii) $\rwc(\mu,\mu)\ge R(\mu,\mu)$ by part (i), and $R(\mu,\mu)>0$ by Proposition~\ref{prop:cov_decomp}.

(iv) Let $p^1,p^2$ be distributions of optimal decisions, $\lambda\in[0,1]$, and $p^\lambda:=\lambda p^1+(1-\lambda)p^2$. If $\bar\gamma^m\in\Gamma(p^m,\nu)$ for $m=1,2$, then $\lambda\bar\gamma^1+(1-\lambda)\bar\gamma^2\in\Gamma(p^\lambda,\nu)$, and $\int c\,d(\lambda\bar\gamma^1+(1-\lambda)\bar\gamma^2)=\lambda\int c\,d\bar\gamma^1+(1-\lambda)\int c\,d\bar\gamma^2$. Taking suprema over $\bar\gamma^1$ and $\bar\gamma^2$ gives $\rwc(p^\lambda,\nu)\ge\lambda\rwc(p^1,\nu)+(1-\lambda)\rwc(p^2,\nu)$, and taking infima gives $\wdfo(p^\lambda,\nu)\le\lambda\wdfo(p^1,\nu)+(1-\lambda)\wdfo(p^2,\nu)$.

Because $\Gamma(\mu,\nu)$ is convex and $\gamma\mapsto\wspo(\mu,\nu;\gamma)$ is linear, and because both extreme values are attained (Corollary~\ref{prop:basic} and Proposition~\ref{prop:pess}), the set of coupled regrets consistent with the two populations is exactly the interval $[\wdfo(\mu,\nu),\rwc(\mu,\nu)]$.\hfill\Halmos
\end{proof}

The worst-case regret also admits a quadratic transport formulation, which mirrors Theorem~\ref{thm:1}.

\begin{theorem}[Quadratic optimal transport formulation of the worst-case regret]\label{thm:pess_w2}
Assume that $\int\|y\|_2^2\,\nu(dy)<\infty$. Then
\[
\rwc(\mu,\nu)=\frac12\Big[\|w^*_{\#}\mu\|^2+\|\nu\|^2-2\int z\,d\nu-W_2^2\big(w^*_{\#}\mu,\nu\big)\Big],
\]
and a plan $\bar\gamma\in\Gamma(w^*_{\#}\mu,\nu)$ maximizes $\int c\,d\bar\gamma$ if and only if it is an optimal coupling for $W_2(w^*_{\#}\mu,\nu)$. Consequently, the worst-case alignment minimizes $\EE\|Y-w^*(X)\|_2^2$.
\end{theorem}

Together, Theorems~\ref{thm:1} and~\ref{thm:pess_w2} show that the two extreme alignments are mirror images: the worst case matches each decision with the coefficient vectors closest to it in Euclidean distance, and the best case matches each decision with the coefficient vectors farthest from it, or, equivalently, closest to its negation. Both are computed by the same semi-discrete $W_2$ solver, and the lifting construction of Theorem~\ref{thm:2} applies verbatim to the worst-case plan.

\begin{proof}{\bfseries Proof of Theorem \ref{thm:pess_w2}}
By the reduced representation in the proof of Proposition~\ref{prop:pess}(i),
$\rwc(\mu,\nu)=\sup_{\bar\gamma\in\Gamma(w^*_{\#}\mu,\nu)}\int y^\top w\,d\bar\gamma(w,y)-\int z\,d\nu$. By the polarization identity $y^\top w=\tfrac12(\|w\|_2^2+\|y\|_2^2-\|w-y\|_2^2)$ and because the integrals of $\|w\|_2^2$ and $\|y\|_2^2$ depend only on the marginals,
\begin{align*}
\sup_{\bar\gamma\in\Gamma(w^*_{\#}\mu,\nu)}\int y^\top w\,d\bar\gamma
&=\frac12\Big(\|w^*_{\#}\mu\|^2+\|\nu\|^2-\inf_{\bar\gamma\in\Gamma(w^*_{\#}\mu,\nu)}\int\|w-y\|_2^2\,d\bar\gamma\Big)\\
&=\frac12\Big(\|w^*_{\#}\mu\|^2+\|\nu\|^2-W_2^2\big(w^*_{\#}\mu,\nu\big)\Big),
\end{align*}
and a plan attains the supremum on the left if and only if it attains the infimum on the right.\hfill\Halmos
\end{proof}

\paragraph{Entropy regularization.}
For $\varepsilon>0$, define the entropy-regularized worst-case regret
\begin{equation}\tag{E-WC}\label{EWDFR}
    \rwcer(\mu,\nu;\varepsilon)
    :=\sup_{\gamma\in\Gamma(\mu,\nu)}
    \left\{\int \lspo(x,y)\,d\gamma(x,y)
    -\varepsilon\,\KL(\gamma\|\mu\otimes\nu)\right\}.
\end{equation}

\begin{proposition}[Entropy-regularized worst-case regret]\label{prop:erdf_pess}
For fixed $\mu$ and $\nu$, the map $\varepsilon\mapsto\rwcer(\mu,\nu;\varepsilon)$ is nonincreasing for $\varepsilon>0$, and for every $\varepsilon>0$,
\[
\max\big\{R(\mu,\nu),\ \rwc(\mu,\nu)-\varepsilon\,\mathrm{H}(w^*_{\#}\mu)\big\}\le\rwcer(\mu,\nu;\varepsilon)\le\rwc(\mu,\nu).
\]
\end{proposition}

\begin{proof}{\bfseries Proof of Proposition \ref{prop:erdf_pess}}
Monotonicity follows as in the proof of Proposition~\ref{prop:erdf}, since the regularization term enters \eqref{EWDFR} with a minus sign. Evaluating \eqref{EWDFR} at $\gamma=\mu\otimes\nu$ gives $\rwcer(\mu,\nu;\varepsilon)\ge R(\mu,\nu)$, and dropping the nonpositive term $-\varepsilon\KL(\gamma\|\mu\otimes\nu)$ gives $\rwcer(\mu,\nu;\varepsilon)\le\rwc(\mu,\nu)$. Finally, let $\bar\gamma^P$ attain the reduced supremum in the proof of Proposition~\ref{prop:pess}(i) and let $\gamma^P$ be its lift. As in the proof of Theorem~\ref{thm:entropic}(i), $\KL(\gamma^P\|\mu\otimes\nu)=\KL(\bar\gamma^P\|w^*_{\#}\mu\otimes\nu)\le\mathrm{H}(w^*_{\#}\mu)$, so evaluating \eqref{EWDFR} at $\gamma^P$ gives $\rwcer(\mu,\nu;\varepsilon)\ge\rwc(\mu,\nu)-\varepsilon\,\mathrm{H}(w^*_{\#}\mu)$.\hfill\Halmos
\end{proof}

\paragraph{Estimation.}
The statistical analysis of Section~\ref{sec:sample_complexity} carries over to the worst-case regret.

\begin{corollary}[Two-sample concentration for $\rwc$]\label{cor:pess_estimation}
Under the assumptions of Section~\ref{sec:sample_complexity}, the bound \eqref{eq:sample_complexity_bound} holds with $\wdfo$ replaced by $\rwc$.
\end{corollary}

\begin{proof}{\bfseries Proof of Corollary \ref{cor:pess_estimation}}
By Lemma~\ref{lemma:cons} and Kantorovich duality for the maximization problem, $\rwc(\mu,\nu)=\inf_{f\in\bbR^{I_\mu}}\{\sum_{i\in I_\mu}f_i\mu(D_i)+\int g^P_f\,d\nu\}$ with $g^P_f(y):=\max_{i\in I_\mu}(c_i(y)-f_i)$, and likewise for the empirical measures. The objective is invariant under $f\mapsto f+t\one$, so we may assume $\min_if_i=0$. If $f_i>B$ for some $i$ and $f_j=0$, then $c_i(y)-f_i<c_i(y)-B\le0\le c_j(y)-f_j$ for $y\in\mathcal{Y}$, so lowering $f_i$ to $B$ leaves $g^P_f$ unchanged on $\mathcal{Y}$ and weakly decreases the objective. Hence, the infimum may be restricted to $\{f:\ \min_if_i=0,\ \max_if_i\le B\}$, on which $|f_i|\le B$, $0\le g^P_f\le B$ on $\mathcal{Y}$, and $|g^P_f-g^P_{f'}|\le\|f-f'\|_\infty$. Centering the prices at $B/2$, Steps~2 and~3 of the proof of Theorem~\ref{ThmSampleComplexity} apply with suprema replaced by infima and $g_f$ replaced by $g^P_f$.\hfill\Halmos
\end{proof}

\section{Decision-Focused Interpolation}\label{sec:interpolant}

This appendix studies how a coupling between two populations induces intermediate distributions. We retain the decision-focused objective when selecting the coupling and distinguish interpolation between the original populations from interpolation obtained through the quadratic transport reduction. Section~\ref{sec:simulated_reassessment} uses the coefficient-space displacement map for a retrospective reassessment experiment. Prospective forecasting and population modeling remain topics for future research.

\subsection{McCann interpolation from an optimal coupling}\label{sec:mccann}

Let $\mu_0,\mu_1\in\cP_2(\bbR^d)$, where $\cP_2(\bbR^d)$ denotes the probability measures with finite second moments. For a coupling $\gamma\in\Gamma(\mu_0,\mu_1)$, define $\pi_t(x,y):=(1-t)x+ty$ for $t\in[0,1]$. The push-forward $(\pi_t)_{\#}\gamma$ moves the coupled points along straight lines and allows mass to split across destinations. When $\gamma$ is optimal for the quadratic transport cost, this construction is the McCann displacement interpolation, a constant-speed geodesic under $W_2$ \citep{mccann1997convexity}. We use the same interpolation map with couplings selected by decision regret.

\subsection{Decision-focused interpolations}\label{sec:df_interpolations}

\begin{definition}[Decision-focused interpolant induced by a coupling]\label{def:df_interpolants}
Let $\mu_0,\mu_1\in\cP_2(\bbR^d)$ and $\gamma\in\Gamma(\mu_0,\mu_1)$. The \emph{decision-focused interpolant} induced by $\gamma$ is $\mu_t^\gamma:=(\pi_t)_{\#}\gamma$ for $t\in[0,1]$.
\end{definition}

We obtain an optimistic interpolant by choosing a coupling optimal for $\wdfo(\mu_0,\mu_1)$ and call its midpoint $\mu_{1/2}^\gamma$ a \emph{decision-focused average}. Couplings attaining the worst-case regret $\rwc$ or solving the entropy-regularized problem \eqref{EWDFO} yield other interpolation paths. Each path connects $\mu_0$ and $\mu_1$, but it need not be a $W_2$ geodesic.

\begin{remark}[Selection among optimal couplings]\label{rem:interp_selection}
Optimal DF couplings need not be unique, and the interpolant depends on the full coupling, not only on its decision-level plan. One selection rule minimizes quadratic displacement among the DF-optimal couplings. For a fixed optimal reduced plan, this amounts to solving a quadratic transport problem within each decision region; selecting that reduced plan is an additional step in the global lexicographic problem. If the interpolant is used as a forecast, its coupling should also reflect assumptions about how individuals' coefficient vectors evolve.
\end{remark}

The next result bounds the regret of retaining the original decisions along any coupling-induced path.

\begin{theorem}[A one-sided bound along DF interpolants]\label{thm:dfo_mccann_one_sided}
Let $\mu,\nu$ have finite first moments, let $\gamma\in\Gamma(\mu,\nu)$, and define $\nu_t^\gamma:=(\pi_t)_{\#}\gamma$. For $(X,Y)\sim\gamma$, let $\gamma_t$ be the joint distribution of $(X,\pi_t(X,Y))$. Then, for every $t\in[0,1]$,
\[
\wspo(\mu,\nu_t^\gamma;\gamma_t)
\le t\,\wspo(\mu,\nu;\gamma).
\]
In particular, if $\gamma^*$ is optimal for $\wdfo(\mu,\nu)$ and $\nu_t:=(\pi_t)_{\#}\gamma^*$, then $\wdfo(\mu,\nu_t)\le t\,\wdfo(\mu,\nu)$.
\end{theorem}

\begin{proof}{\bfseries Proof of Theorem \ref{thm:dfo_mccann_one_sided}}
By construction, $\gamma_t\in\Gamma(\mu,\nu_t^\gamma)$ and
\begin{equation}\label{eq:step2_pushforward}
\wspo(\mu,\nu_t^\gamma;\gamma_t)
=
\int\lspo(u,\pi_t(u,v))\,\gamma(du,dv).
\end{equation}
Since $z$ is concave,
\begin{equation}\label{eq:step4_concave}
z((1-t)u+tv)\ge(1-t)z(u)+tz(v).
\end{equation}
Using $u^\top w^*(u)=z(u)$ and \eqref{eq:step4_concave}, we obtain
\begin{align}
\lspo(u,\pi_t(u,v))
&=(1-t)z(u)+tv^\top w^*(u)-z((1-t)u+tv) \label{eq:step3_expand}\\
&\le t\big(v^\top w^*(u)-z(v)\big)
=t\,\lspo(u,v). \label{eq:step4_pointwise}
\end{align}
All terms are integrable because $S$ is compact and the marginals have finite first moments. Integrating the pointwise inequality gives the first claim. For an optimal $\gamma^*$, the induced coupling is feasible for $\wdfo(\mu,\nu_t)$, so taking the infimum over couplings gives the second claim.\hfill\Halmos
\end{proof}

Theorem~\ref{thm:dfo_mccann_one_sided} provides a linear upper bound on regret from the original population. A corresponding bound from the right endpoint need not hold. For example, let $S=[0,1]$, $\mu=\delta_{-2}$, and $\nu=\delta_1$. The unique coupling gives $\nu_{1/2}=\delta_{-1/2}$ and $\wdfo(\nu_{1/2},\nu)=1>\tfrac12\wdfo(\mu,\nu)=\tfrac12$.

\subsection{McCann interpolation via the $W_2$ reduction}\label{sec:mccann_df_reduction}

Theorem~\ref{thm:1} provides a second construction in the reduced space. Let $\mu,\nu\in\cP_2(\bbR^d)$, write $\bar w^*:=-w^*$, and set $\alpha:=(\bar w^*)_{\#}\mu$. For $\gamma\in\Gamma(\mu,\nu)$ and $\eta:=(\bar w^*\times\mathrm{Id})_{\#}\gamma$, the quadratic reduction gives
\begin{equation}\label{eq:wdfo_reduction_general}
\wspo(\mu,\nu;\gamma)
=
\frac12\int\|a-y\|_2^2\,\eta(da,dy)+C(\mu,\nu),
\end{equation}
where $C(\mu,\nu):=-\frac12\int\|a\|_2^2\,\alpha(da)-\frac12\int\|y\|_2^2\,\nu(dy)-\int z(y)\,\nu(dy)$ depends only on the marginals. Every coupling of $\alpha$ and $\nu$ admits a lift to $\Gamma(\mu,\nu)$, so
\begin{equation}\label{eq:wdfo_w2_reduction}
\wdfo(\mu,\nu)=\frac12W_2^2(\alpha,\nu)+C(\mu,\nu).
\end{equation}
We can therefore construct a McCann interpolation from $\alpha$ to $\nu$ and lift its transport plans to couplings with first marginal $\mu$.

\begin{proposition}[McCann geodesic in decision space and consistent lifts]\label{prop:mccann_lift}
Let $\eta^*\in\Gamma(\alpha,\nu)$ be optimal for $W_2^2(\alpha,\nu)$. For $(A,Y)\sim\eta^*$, define $Y_t:=(1-t)A+tY$, let $\nu_t$ be the distribution of $Y_t$, and let $\eta_t$ be the joint distribution of $(A,Y_t)$. For Borel sets $E,F\subseteq\bbR^d$, define
\begin{equation}\label{eq:gamma_t_lift}
\gamma_t(E\times F)
:=
\sum_{i:\,\mu(D_i)>0}
\frac{\mu(E\cap D_i)}{\mu(D_i)}
\,\eta_t(\{-w_i\}\times F).
\end{equation}
Then the following hold for $s,t\in[0,1]$.
\begin{enumerate}
\item[(i)] The path $\{\nu_t\}_{t\in[0,1]}$ is a constant-speed $W_2$ geodesic from $\alpha$ to $\nu$, with $W_2(\nu_s,\nu_t)=|t-s|W_2(\alpha,\nu)$.
\item[(ii)] The displacements satisfy
\[
\EE\|Y_t-A\|_2^2=t^2W_2^2(\alpha,\nu),
\qquad
\EE\|Y_t-Y_s\|_2^2=(t-s)^2W_2^2(\alpha,\nu).
\]
\item[(iii)] Each $\gamma_t$ belongs to $\Gamma(\mu,\nu_t)$ and satisfies $(\bar w^*\times\mathrm{Id})_{\#}\gamma_t=\eta_t$.
\item[(iv)] Each $\gamma_t$ is optimal for $\wdfo(\mu,\nu_t)$. More generally, lifting any quadratic-optimal coupling between $\alpha$ and $\nu_t$ by \eqref{eq:gamma_t_lift} gives a DF-optimal coupling between $\mu$ and $\nu_t$.
\end{enumerate}
\end{proposition}

\begin{proof}{\bfseries Proof of Proposition \ref{prop:mccann_lift}}
For (i)--(ii), write $L:=W_2(\alpha,\nu)$. Since $Y_t-Y_s=(t-s)(Y-A)$ and $\EE\|Y-A\|_2^2=L^2$, the displacement identities hold and $W_2(\nu_s,\nu_t)\le|t-s|L$. For $0\le s\le t\le1$, the triangle inequality gives
\[
\begin{aligned}
L
&\le W_2(\alpha,\nu_s)+W_2(\nu_s,\nu_t)+W_2(\nu_t,\nu)\\
&\le sL+(t-s)L+(1-t)L=L.
\end{aligned}
\]
All bounds are therefore equalities. In particular, $\eta_t$ is quadratic-optimal between $\alpha$ and $\nu_t$, since its cost is $t^2L^2=W_2^2(\alpha,\nu_t)$.

For (iii), $\eta_t(\{-w_i\}\times\bbR^d)=\mu(D_i)$, so \eqref{eq:gamma_t_lift} gives $\gamma_t(E\times\bbR^d)=\mu(E)$ and $\gamma_t(\bbR^d\times F)=\nu_t(F)$. Moreover, $\gamma_t(D_i\times F)=\eta_t(\{-w_i\}\times F)$, which proves the push-forward identity. Finally, \eqref{eq:wdfo_reduction_general} makes minimizing DF regret equivalent to minimizing the quadratic cost of the reflected decision-level coupling. Thus, every quadratic-optimal coupling between $\alpha$ and $\nu_t$, including $\eta_t$, has a DF-optimal lift, proving (iv).\hfill\Halmos
\end{proof}

The two constructions have different starting marginals. Definition~\ref{def:df_interpolants} interpolates between the original populations $\mu$ and $\nu$, whereas Proposition~\ref{prop:mccann_lift} starts from the distribution of negated decisions $\alpha$. Lifting supplies optimal couplings with first marginal $\mu$ but does not change the interpolated marginal $\nu_t$. The direct construction therefore describes intermediate populations; the reduced construction establishes the associated quadratic transport geometry.

\section{Details of the Numerical Study}\label{append:numerical}

\subsection{Parkinson's data and care-plan model}\label{sec:parkinsons_setup}

\paragraph{Dataset.}
Each released row is a voice recording identified by \texttt{subject\#} and \texttt{test\_time}. The motor and total UPDRS values are piecewise-linear interpolations of clinician assessments at baseline, three months, and six months. The rows supply neither a clinical-visit indicator nor separate scores from the original assessments. We analyze the profiles supplied in the release and link repeated measurements by patient.

\paragraph{Time snapshots via windowing.}
We center the windows at days $t\in\{50,100,150\}$ and use observations with integer-rounded measurement time in $[t-5,t+5]$. Within each window and patient, we select the record whose rounded time is closest to $t$. We break ties by the first occurrence in the dataset. There are $41$, $38$, and $39$ patient representatives at days~50, 100, and~150, respectively. There are $37$ patients in both the day-50 and day-100 windows, $38$ in both the day-50 and day-150 windows, and $34$ in all three windows. Section~\ref{sec:realdata} uses the matched patients at each horizon as the two empirical marginals. The interpolation experiment in Appendix~\ref{sec:parkinsons_interpolation} also uses the full window samples, as specified there.

\paragraph{Cost vectors.}
To evaluate decision-focused discrepancy, we embed each patient observation into a four-dimensional cost vector
\[
y = (y_{\mathrm{ex}},y_{\mathrm{nurse}},y_{\mathrm{rest}},y_{\mathrm{speech}})\in\mathbb{R}^4,
\]
constructed from two severity measures and two additional covariates:
\[
(\texttt{motor\_UPDRS},\texttt{total\_UPDRS},\texttt{age},\texttt{PPE}).
\]
We apply min--max normalization to each coordinate (using the full dataset's range) and then negate the result:
\[
y_k := -\frac{f_k-\min(f_k)}{\max(f_k)-\min(f_k)} \in [-1,0].
\]
This normalization yields a stylized but interpretable model: greater disease burden or higher-risk covariate values
correspond to more negative coefficients, i.e., a larger marginal benefit from allocating hours to that intervention
type. (Our objective is to \emph{minimize} $y^\top w$, so more negative coefficients incentivize allocating more
hours.)

\paragraph{Decisions.}
The decision is a weekly care plan
\[
w=(h_{\mathrm{exercise}},h_{\mathrm{nursing}},h_{\mathrm{rest}},h_{\mathrm{speech}})\in\mathbb{R}^4,
\]
with each coordinate interpreted as hours per week allocated to a care modality: physical exercise/physiotherapy,
nursing/care coordination, sleep/rest optimization, and speech therapy.

\paragraph{Feasible region.}
The four extreme plans each allocate ten hours per week:
\[
\begin{aligned}
w_1&=(6,2,2,0) &&\text{(exercise-focused)}\\
w_2&=(2,6,2,0) &&\text{(nursing-focused)}\\
w_3&=(2,2,6,0) &&\text{(rest-focused)}\\
w_4&=(3,3,2,2) &&\text{(speech add-on)}.
\end{aligned}
\]
The balanced plan $(3,3,3,1)$ equals $\frac18 w_1+\frac18 w_2+\frac14 w_3+\frac12 w_4$. It is feasible in $S=\mathrm{conv}\{w_1,w_2,w_3,w_4\}$ and cannot be a unique optimal plan for a linear objective. The oracle in \eqref{eq:downstream_oracle} selects an extreme point. The minimizing plan is unique on the original mapped profiles used in the matched-patient analysis.

\paragraph{Interpretation.}
A more negative cost coefficient favors allocating more hours to the corresponding care modality. For example, higher motor UPDRS increases the modeled benefit of exercise. The objective $y^\top w$ is a stylized proxy for the value of a care plan.

\subsection{Patient-level uncertainty and coefficient-map sensitivity}\label{sec:parkinsons_uncertainty}

\paragraph{Paired resampling.}
For each horizon, we draw $10{,}000$ bootstrap samples of linked patients, holding the coefficient map fixed. In each sample we recompute mean tracked regret and independent-coupling regret, using the product of the two resampled marginals for the latter. Table~\ref{tab:parkinsons_paired_uncertainty} reports pointwise percentile intervals. They describe sampling variation over linked recordings under the chosen coefficient map. The intervals for the tracked-minus-independent gap stay below zero at both horizons.

\begin{table}[ht]
\centering
\begin{tabular}{lrrrr}
\hline
Follow-up & Tracked & 95\% interval & Tracked minus independent & 95\% interval \\
\hline
Day 100 & 0.0167 & $[0.0009,0.0411]$ & -0.4790 & $[-0.5801,-0.3348]$ \\
Day 150 & 0.0359 & $[0.0048,0.0775]$ & -0.4272 & $[-0.5335,-0.2871]$ \\
\hline
\end{tabular}
\caption{Patient-level bootstrap intervals for the public telemonitoring profiles. Baseline and follow-up records for each patient are resampled together.}
\label{tab:parkinsons_paired_uncertainty}
\end{table}

\paragraph{Decision shares and optimistic DF divergence.}
Under the no-ties condition, the zero-divergence null in Corollary~\ref{prop:basic} is equality of the decision shares. For each of the four plans, equality of its baseline and follow-up share implies equal probabilities of entering and leaving that plan. Holding the original coefficient map fixed, we examine each plan using an exact binomial test conditional on the number of discordant pairs, with a Holm adjustment across the four tests at each horizon. The smallest unadjusted $p$-values are $0.500$ and $0.375$ at days~100 and~150, respectively, and the smallest adjusted $p$-value is $1$ at each horizon.

We also construct a conservative interval for the population optimistic DF divergence under independent sampling of patients, allowing arbitrary dependence within a linked pair. We split the $42$ patient identifiers into $14$ training patients and $28$ evaluation patients with a fixed random seed. We fit the feature ranges on the training patients and clip mapped coefficients to $[-1,0]^4$ for both groups. This gives a known bound $0\le c_k(Y)\le B=4$, since all plans have the same total hours and the largest pairwise $\ell_1$ distance between plans is $8$. For a dual potential $f$ with $\max_k f_k=0$ and $f_k\in[-B,0]$, define the patient-level dual score
\[
    H_f(X,Y)=f_{i(X)}+\min_j\{c_j(Y)-f_j\},
    \qquad w^*(X)=w_{i(X)}.
\]
The score lies in $[-B,B]$. The population and empirical DF divergences are the suprema of the respective means of $H_f$, using the bounded potentials in Section~\ref{sec:sample_complexity}. We fit a potential $\tilde f$ on the training pairs and denote its mean score on the $m$ evaluation pairs by $\overline H_{\tilde f}$. For a population-optimal potential $f^*$, the evaluation plug-in estimate $\widehat W$ is at least the empirical mean of $H_{f^*}$. Applying the two one-sided Hoeffding inequalities to the linked-pair scores gives the pointwise $95\%$ interval
\[
    \big[\max\{0,\overline H_{\tilde f}-r_m\},\
    \min\{B,\widehat W+r_m\}\big],
    \qquad r_m=2B\sqrt{\frac{\log(2/0.05)}{2m}}.
\]
The training data fix the scale and the potential used in the lower endpoint. Table~\ref{tab:parkinsons_df_bounds} gives the evaluation estimates and intervals. The bounds include zero and are wide relative to the estimated divergences.

\begin{table}[ht]
\centering
\begin{tabular}{lrrrr}
\hline
Follow-up & Training pairs & Evaluation pairs & Evaluation $\widehat W$ & 95\% interval \\
\hline
Day 100 & 12 & 25 & 0.0249 & $[0.000,2.198]$ \\
Day 150 & 13 & 25 & 0.0129 & $[0.000,2.186]$ \\
\hline
\end{tabular}
\caption{Conservative pointwise intervals from linked evaluation pairs with a separately trained coefficient map.}
\label{tab:parkinsons_df_bounds}
\end{table}

\paragraph{Sensitivity to the coefficient map.}
We set every patient's age coefficient to its mean at day~50, freeze each patient's PPE coefficient at its own day-50 value, and separately halve or double the weight of PPE. Each change is applied individually. We compute the $W_2$ match in the resulting coordinates and evaluate all regrets under the resulting linear objective. Table~\ref{tab:parkinsons_coefficient_sensitivity} reports the change in plan selection and regret. Freezing PPE leaves tracked regret and optimistic DF unchanged at day~150. Doubling its weight increases both relative to the original map.

\begin{table}[ht]
\centering
\begin{tabular}{lrrrrrr}
\hline
Map & Day & Plan changes & Tracked & $W_2$ match & Optimistic DF & Independent \\
\hline
Original map & 100 & 4 & 0.0167 & 0.0629 & 0.0079 & 0.4956 \\
Common age & 100 & 3 & 0.0153 & 0.0352 & 0.0117 & 0.4851 \\
Baseline PPE & 100 & 5 & 0.0164 & 0.0164 & 0.0050 & 0.4914 \\
Half PPE weight & 100 & 6 & 0.0199 & 0.0630 & 0.0143 & 0.4434 \\
Double PPE weight & 100 & 5 & 0.0301 & 0.0549 & 0.0244 & 0.5770 \\
Original map & 150 & 5 & 0.0359 & 0.0575 & 0.0081 & 0.4631 \\
Common age & 150 & 8 & 0.0367 & 0.0298 & 0.0097 & 0.4308 \\
Baseline PPE & 150 & 5 & 0.0359 & 0.0492 & 0.0081 & 0.4622 \\
Half PPE weight & 150 & 7 & 0.0381 & 0.0597 & 0.0103 & 0.4141 \\
Double PPE weight & 150 & 10 & 0.0847 & 0.0797 & 0.0271 & 0.5289 \\
\hline
\end{tabular}
\caption{Sensitivity to the mapping of features to the care-plan objective. Plan changes count patients whose optimal plan at follow-up differs from the day-50 plan. The four rightmost columns give mean SPO regrets under the stated coupling.}
\label{tab:parkinsons_coefficient_sensitivity}
\end{table}

For illustration, we compare the first checkpoint at which each estimate exceeds $0.02$ on the corresponding objective scale. Under the original map, tracked regret first crosses this level at day~150, $W_2$ matching and independence cross at day~100, and the optimistic DF divergence stays below it at both checkpoints. Freezing PPE delays the crossing under $W_2$ matching to day~150. Doubling its weight moves the tracked and optimistic DF crossings to day~100. The other variants give the same crossing times as the original map. These changes show how the stylized specification affects the retrospective threshold comparison.

\subsection{Retrospective entropy calibration}\label{sec:parkinsons_calibration}

This appendix gives computational details and additional results for the entropy calibration in Section~\ref{sec:realdata} and Figure~\ref{fig:parkinsons_kl_calibration}.

For each horizon, we compute the day-50 decision shares and the reduced cost matrix with entries $c_k(y_{jt})$ from \eqref{eq:reduced_cost}. The optimistic DF divergence is the solution of the reduced transportation linear program in Section~\ref{sec:calculation}. For $\varepsilon>0$, we solve the reduced entropy-regularized problem by fitting the Gibbs potentials to the day-50 decision shares. We use a logarithmically spaced grid from $0.001$ to $100$ and refine each matching weight locally within a bracketing interval, without assuming global monotonicity of the same-plan share. To compute $G(\gamma_{\varepsilon,t})$ in Table~\ref{tab:parkinsons_matched_regret}, we lift the reduced coupling by dividing the mass assigned to each day-50 plan equally among the baseline profiles that use that plan, as in Proposition~\ref{prop:entropic_reduction}.

Let
\[
K_t(\varepsilon)
:=\KL\!\left(\bar\gamma_{\varepsilon,t}\,\middle\|\,w^*_{\#}\hat\mu_{50\mid t}\otimes\hat\nu_t\right)
\]
denote the KL divergence of the reduced coupling from the product of its marginals. Proposition~\ref{prop:entropic_reduction} gives the same KL divergence for the lifted coupling. Table~\ref{tab:parkinsons_kl_weights} reports both the KL divergence and the weighted penalty $\varepsilon K_t(\varepsilon)$. The regularized objective is the sum of the entropic regret and the weighted penalty.

\begin{table}[ht]
\centering
\begin{tabular}{llrrrr}
\hline
Follow-up & Matching target & $\varepsilon$ & $\rho_t(\varepsilon)$ & $K_t(\varepsilon)$ & $\varepsilon K_t(\varepsilon)$ \\
\hline
Day 100 & Same-plan share & 0.0772 & 0.0143 & 0.8497 & 0.0656 \\
Day 100 & Tracked regret & 0.0912 & 0.0167 & 0.8222 & 0.0750 \\
Day 150 & Same-plan share & 0.1055 & 0.0210 & 0.8005 & 0.0845 \\
Day 150 & Tracked regret & 0.1547 & 0.0359 & 0.6854 & 0.1060 \\
\hline
\end{tabular}
\caption{Retrospective KL weights and penalties for the two matched-patient cohorts. The same-plan target uses the observed share of unchanged optimal plans. The tracked-regret target fits the observed mean SPO loss on the same cohort.}
\label{tab:parkinsons_kl_weights}
\end{table}

We also examine calibration on the $34$ patients observed in all three windows. At day~150, the observed same-plan share is $30/34$, above the maximum fitted share of approximately $0.853$ on the grid. The ability to match this statistic depends on the cohort and on the range of assignments generated by the regularized model.

\subsection{Sensitivity of the Wasserstein comparison}\label{sec:parkinsons_w2_sensitivity}

The two empirical marginals in each comparison assign equal mass to the same number of observations, so both matching problems admit solutions given by one-to-one assignments. For the optimistic DF match, we first solve the SPO assignment problem and then minimize squared displacement over assignments supported on edges with zero reduced cost in that problem.

\paragraph{Common cohort.}
Table~\ref{tab:parkinsons_w2_common} repeats the comparison on the $34$ patients observed in all three windows, so the baseline population and the distribution of optimal decisions are identical across horizons. The $W_2$ match incurs more regret than tracking at both horizons. Its regret increases slightly on the common cohort, while it decreases slightly across the horizon-specific cohorts in Table~\ref{tab:parkinsons_matched_regret}. The direction of the change in regret under the $W_2$ coupling therefore depends on the cohort.

\begin{table}[ht]
\centering
\begin{tabular}{llrrr}
\hline
Transition & Coupling & RMS displacement & Mean SPO regret & Same plan \\
\hline
$50\to 100$ & Patient tracking & 0.1411 & 0.0070 & 91.2\% \\
 & $W_2$ matching & 0.1339 & 0.0573 & 79.4\% \\
 & Optimistic DF & 0.1515 & 0.0038 & 94.1\% \\
 & Independent & 0.5601 & 0.4929 & 38.9\% \\
\hline
$50\to 150$ & Patient tracking & 0.1947 & 0.0368 & 88.2\% \\
 & $W_2$ matching & 0.1680 & 0.0610 & 76.5\% \\
 & Optimistic DF & 0.2109 & 0.0173 & 85.3\% \\
 & Independent & 0.5532 & 0.4861 & 36.9\% \\
\hline
\end{tabular}
\caption{Geometric displacement and regret when both horizons use the same $34$ patients. All measures use the definitions in Table~\ref{tab:parkinsons_matched_regret}.}
\label{tab:parkinsons_w2_common}
\end{table}

\paragraph{Severity-only matching.}
We also compute a $W_2$ coupling using Euclidean distance on the two normalized UPDRS coordinates. We evaluate its displacement in the original four normalized coefficients and its regret in the same care-plan problem. Table~\ref{tab:parkinsons_w2_severity} shows that matching on the two severity coordinates yields higher mean SPO regret than matching on all four coordinates at both horizons and for both cohort definitions. The choice of ground metric changes the resulting assignments and can therefore affect the regret attributed to the same population shift.

\begin{table}[ht]
\centering
\begin{tabular}{llrrr}
\hline
Cohort & Transition & Severity-only $W_2$ & Four-coefficient RMS & Mean SPO regret \\
\hline
Horizon-specific & $50\to 100$ & 0.0688 & 0.2651 & 0.2271 \\
 & $50\to 150$ & 0.1030 & 0.2963 & 0.2135 \\
\hline
Common (34) & $50\to 100$ & 0.0730 & 0.2330 & 0.1450 \\
 & $50\to 150$ & 0.1074 & 0.2725 & 0.1431 \\
\hline
\end{tabular}
\caption{Sensitivity to matching on the two normalized severity coordinates. The projected $W_2$ value describes displacement in those coordinates. Four-coefficient RMS displacement and SPO regret are evaluated on the resulting coupling using the full care-plan model.}
\label{tab:parkinsons_w2_severity}
\end{table}

\subsection{Decision-focused interpolation on Parkinson's data}\label{sec:parkinsons_interpolation}

We examine interpolation between the day-50 and day-150 profiles constructed from the public release. The coupling-induced interpolants are defined in Appendix~\ref{sec:interpolant}. We compare densities on the two severity coordinates and evaluate the induced decisions on the supplied day-100 profiles. These profiles already incorporate the trial's interpolation of the clinical assessments.

\paragraph{Severity distributions in the release.}
Figure~\ref{fig:true} shows Gaussian kernel density estimates for motor and total UPDRS using all observations in each time window. The patient-representative marginal distributions are used for the geometric and decision-focused distances reported below.

\begin{figure}[ht]
    \centering
    \includegraphics[width=\linewidth]{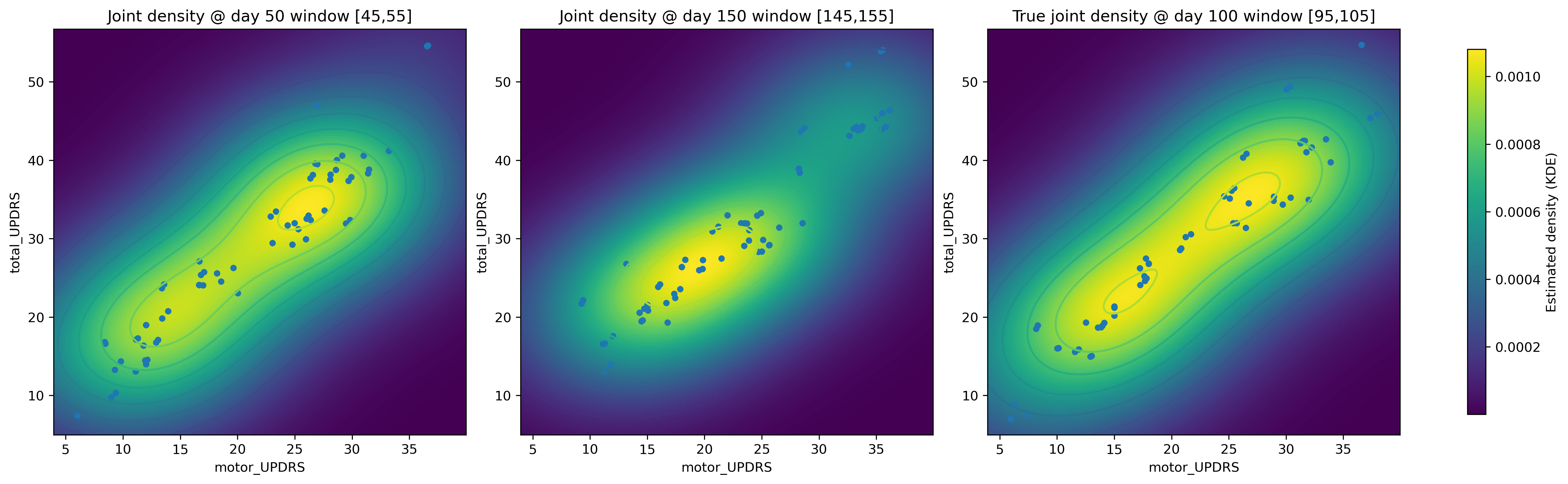}
    \caption{Joint densities of the motor and total UPDRS values supplied with voice recordings in the three time windows. These values are interpolated from the clinical assessments.}
    \label{fig:true}
\end{figure}

\subsubsection{Geometric vs decision-focused discrepancy between day 50 and day 150}

\paragraph{Decision-blind Wasserstein distances on severity.}
In the unnormalized two-dimensional severity space $z=(\texttt{motor\_UPDRS},\texttt{total\_UPDRS})$, the day-$50$ and day-$150$ empirical
distributions are separated by $W_1 \approx 4.339$ and $W_2 \approx 5.090$. These values reflect a nontrivial shift
in severity over approximately $100$ days and motivate constructing intermediate distributions via interpolation.

\paragraph{Decision-focused regret on care-plan costs.}
Let $\hat\mu_{50}$ and $\hat\nu_{150}$ denote the empirical distributions of the four-dimensional cost vectors built from all
patient representatives at days~50 and~150, respectively. Using the SPO loss $\lspo$ as the transport cost, we compute the optimistic DF divergence, the independent-coupling regret, and the worst-case regret:
\[
\wdfo(\hat\mu_{50},\hat\nu_{150}) \approx 0.0061,\qquad
R(\hat\mu_{50},\hat\nu_{150}) \approx 0.4566,\qquad
\rwc(\hat\mu_{50},\hat\nu_{150}) \approx 0.8605.
\]
The gap between $\wdfo$ and $\rwc$ quantifies how much expected regret can vary across couplings of the same marginals. An optimistic coupling keeps the day-50 plans nearly optimal on average, while a worst-case coupling incurs a substantially larger loss.

\medskip

\begin{table}[ht]
\centering
\begin{tabular}{lcc}
\hline
Quantity (day~50 vs day~150) & Value & Interpretation\\
\hline
$W_1$ on (motor,total) & 4.3391 & geometric shift in severity\\
$W_2$ on (motor,total) & 5.0900 & quadratic transport shift\\
$\wdfo(\hat\mu_{50},\hat\nu_{150})$ & 0.0061 & optimistic DF divergence (best-case alignment)\\
$R(\hat\mu_{50},\hat\nu_{150})$ & 0.4566 & i.i.d.\ regret (independent coupling)\\
$\rwc(\hat\mu_{50},\hat\nu_{150})$ & 0.8605 & worst-case regret\\
\hline
\end{tabular}
\caption{Geometric distances on the two severity coordinates and decision-focused regret on the four care-plan coefficients. The empirical marginals use all patient representatives at days~50 and~150.}
\label{tab:parkinsons_dfgeo}
\end{table}

\subsubsection{Coupling-induced interpolations and evaluation at day 100}

\paragraph{Interpolated severity densities.}
Each coupling between the day-50 and day-150 windows induces a day-100 distribution through the midpoint map $\pi_{1/2}(x,y)=(x+y)/2$. We compare optimistic DF, worst-case, and independent couplings using up to $250$ observations from each endpoint window. We sample from each coupling, take the midpoints, and use the motor and total UPDRS coordinates to estimate the joint density. Figure~\ref{fig:interplo} shows the resulting densities and the density of the day-100 scores in the release.

\begin{figure}[ht]
    \centering
    \includegraphics[width=\linewidth]{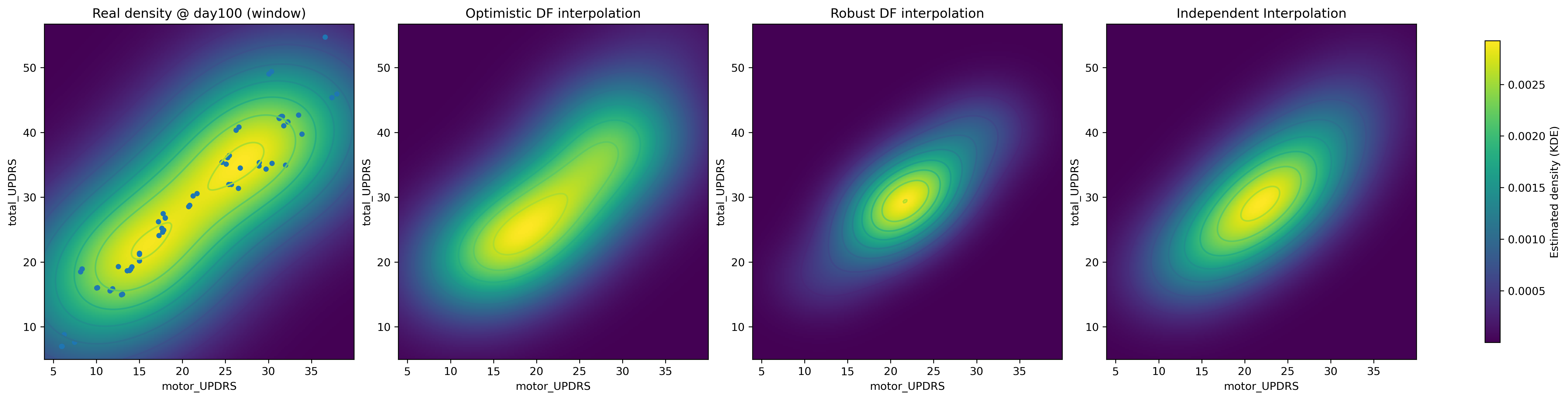}
    \caption{The day-100 severity density in the release and midpoint interpolants obtained from the optimistic DF, worst-case, and independent couplings.}
    \label{fig:interplo}
\end{figure}

\paragraph{Patient-tracked evaluation.}
For each endpoint coupling $\gamma$, we construct a day-100 coefficient prediction for patient $i$ as
\[
    \widehat{Y}_{100}(i)
    :=\tfrac12 Y_{50}(i)+\tfrac12\EE_\gamma[Y_{150}\mid Y_{50}=Y_{50}(i)].
\]
Here the conditional mean selects a single plan. Section~\ref{sec:simulated_reassessment} averages the loss of a fixed baseline plan over the coupled paths. The couplings use all patient representatives at the endpoints. We evaluate the induced plan at the day-100 profile supplied by the release for each of the $37$ patients present at days~50 and~100. The reference is the plan chosen from that patient's own day-50 coefficients. We report the number of predictions that change this plan and the paired mean regret difference
\[
    \frac{1}{37}\sum_i
    \left[\lspo\big(\widehat{Y}_{100}(i),Y_{100}(i)\big)
        -\lspo\big(Y_{50}(i),Y_{100}(i)\big)\right].
\]

\begin{table}[ht]
\centering
\begin{tabular}{lrrrr}
\hline
Method & Plan changes & Mean SPO & Difference & Same-patient mass \\
\hline
Keep day-50 plan & 0 & 0.0167 & 0 & \textemdash \\
Optimistic DF & 0 & 0.0167 & 0.0000 & 4.5\% \\
$W_2$, raw & 4 & 0.0231 & 0.0064 & 47.7\% \\
$W_2$, normalized & 5 & 0.0207 & 0.0040 & 49.5\% \\
$W_2$, normalized without age & 12 & 0.1153 & 0.0987 & 24.7\% \\
$W_2$, raw UPDRS & 11 & 0.1165 & 0.0998 & 23.0\% \\
Independent & 7 & 0.0945 & 0.0778 & 2.4\% \\
Worst-case & 16 & 0.2471 & 0.2305 & 0.1\% \\
\hline
\end{tabular}
\caption{Day-100 prediction regret relative to keeping the day-50 plan. Same-patient mass is the fraction of the day-50 to day-150 coupling assigned to profiles with the same patient identifier. The regret and plan-change columns concern the $37$ evaluated patients.}
\label{tab:parkinsons_tracked_baselines}
\end{table}

\paragraph{Preservation under the optimistic coupling.}
The optimistic predictions keep the day-50 plan for all $41$ patients in the day-50 window, including the $37$ evaluated patients. The SPO loss therefore equals the no-update loss for each evaluated patient. This holds both for the direct LP solution and for the reduced solution lifted using Theorem~\ref{thm:2}. Under the lift, the conditional mean for patients assigned plan $k$ is the mean $m_k$ of the follow-up mass assigned to that plan. We verify that $m_k\in D_k$ for each of the four used plans. Since each $D_k$ is convex, the midpoint of a baseline point in $D_k$ and $m_k$ also lies in $D_k$. This explains preservation for the lifted solution in this experiment.

\paragraph{Sensitivity of the geometric match.}
We compute $W_2$ matchings on all four raw coordinates, on the normalized coordinates, on the normalized coordinates excluding age, and on the two raw UPDRS coordinates. We use the full normalized costs to calculate the predictions and their SPO losses. In the raw four-coordinate match, PPE accounts for $0.025\%$ of the squared transport cost. Same-patient mass is $47.7\%$ for this match and $49.5\%$ after normalization, compared with $2.4\%$ under independence. Excluding age lowers same-patient mass to $24.7\%$ and raises mean regret. The paired bootstrap interval for the normalized match's mean difference is $[-0.0106,0.0180]$. When an endpoint problem has multiple optimal couplings, the predictions can depend on the chosen optimizer. We report the solutions from the specified transportation solver for these comparisons.

\subsubsection{Entropy regularization in the interpolation experiment}

We compute the entropy-regularized optimistic and worst-case couplings in \eqref{EWDFO} and \eqref{EWDFR} with log-domain Sinkhorn iterations. Table~\ref{tab:parkinsons_eps} compares their induced decisions with keeping the day-50 plan. The optimistic predictions at weights $0.1$ and $0.2$ retain all $37$ evaluated plans. At larger weights, the prediction regret for both methods moves toward the independent-coupling prediction regret over the range of weights in Figure~\ref{fig:regular}.

\begin{figure}[ht]
    \centering
    \includegraphics[width=0.8\linewidth]{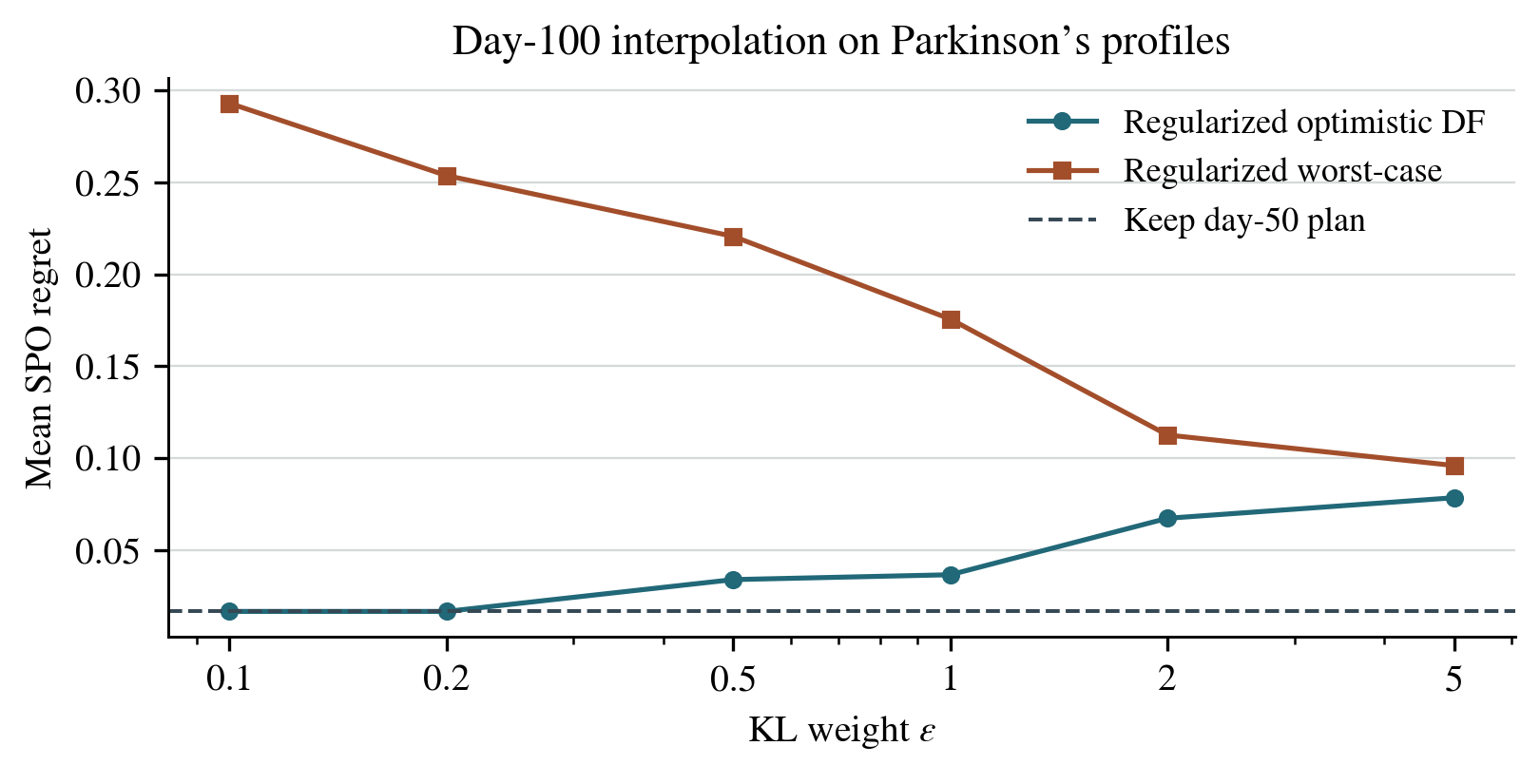}
    \caption{Mean SPO regret against the supplied day-100 profiles for entropy-regularized endpoint couplings and the day-50 no-update baseline.}
    \label{fig:regular}
\end{figure}

\begin{table}[ht]
\centering
\begin{tabular}{crrrrrr}
\hline
 & \multicolumn{3}{c}{Regularized optimistic DF} & \multicolumn{3}{c}{Regularized worst-case} \\
\cline{2-4}\cline{5-7}
$\varepsilon$ & Changes & Mean SPO & Difference & Changes & Mean SPO & Difference \\
\hline
0.1 & 0 & 0.0167 & 0.0000 & 19 & 0.2930 & 0.2764 \\
0.2 & 0 & 0.0167 & 0.0000 & 18 & 0.2537 & 0.2371 \\
0.5 & 2 & 0.0340 & 0.0173 & 15 & 0.2206 & 0.2040 \\
1 & 3 & 0.0366 & 0.0200 & 12 & 0.1756 & 0.1590 \\
2 & 5 & 0.0673 & 0.0507 & 9 & 0.1126 & 0.0960 \\
5 & 6 & 0.0785 & 0.0618 & 8 & 0.0960 & 0.0793 \\
\hline
\end{tabular}
\caption{Plan changes and paired mean regret differences relative to keeping the day-50 plans.}
\label{tab:parkinsons_eps}
\end{table}

Figure~\ref{fig:epislon} shows how the interpolated severity densities change as the endpoint couplings approach independence.

\begin{figure}[H]
    \centering
    \includegraphics[width=\linewidth]{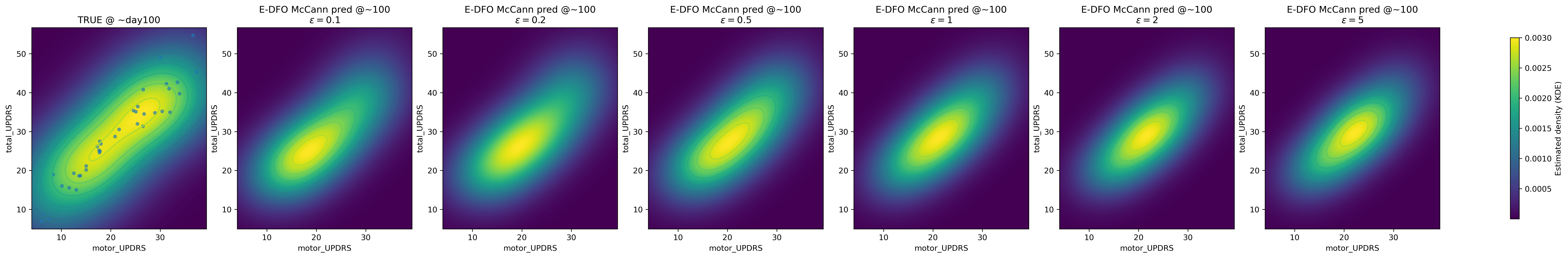}
    \includegraphics[width=\linewidth]{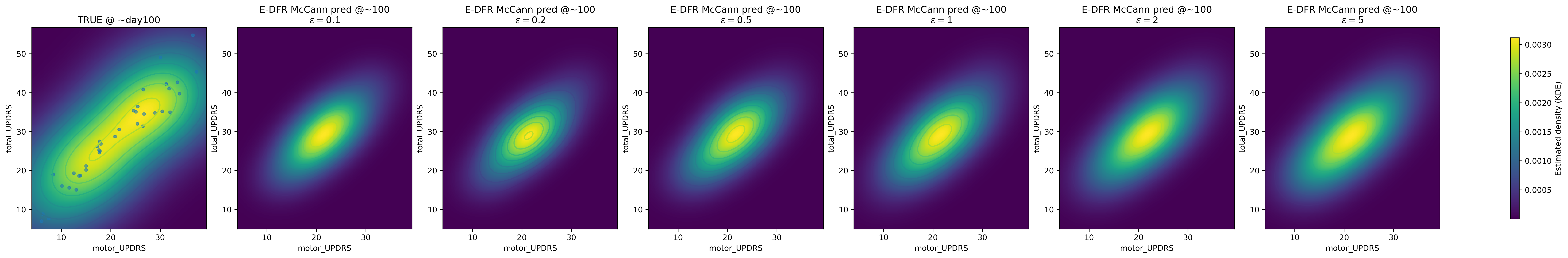}
    \caption{Day-100 severity densities induced by the optimistic DF couplings (top) and worst-case couplings (bottom) for different values of $\varepsilon$.}
    \label{fig:epislon}
\end{figure}

\subsection{Reassessment design and sensitivity analyses}\label{sec:synthetic_longitudinal_design}

We use enrollment days $10$, $60$, $110$, and $160$, which have $40$, $42$, $38$, and $40$ patient representatives, respectively. Their intersection has $36$ patients. Day~10 gives an early starting point with substantial overlap at the four elapsed times. In each split, we fit the feature ranges on the calibration patients' profiles at $t=0$ and $t=100$ and apply the same affine map to all profiles in that split. The patients are split as intact trajectories. The frequencies summarize resampling within this fixed cohort.

Write $p_i=\sum_j\gamma_{ij}=1/18$ for the baseline mass. Conditional on $x_i$, the path ending at $y_j$ has probability $\gamma_{ij}/p_i$. We evaluate the baseline plan at $u_{ij}(t)$, average over these paths, and then average over baseline patients. For $0\leq t\leq 100$ the coefficient distribution is the push-forward of $\gamma$ under the displacement map $(x_i,y_j)\mapsto u_{ij}(t)$. The $W_2$-optimal assignment yields a constant-speed $W_2$ geodesic on this interval. The other couplings use the same map with their own endpoint weights. At $t=150$ the map gives a linear extrapolation along each pair. The endpoint optimization of the DF coupling gives no ordering of checkpoint regrets across couplings.

Among assignments attaining the minimum endpoint decision loss, the optimistic benchmark minimizes squared geometric displacement. For the entropy coupling we distribute the reduced coupling mass of each baseline decision equally among patients assigned that decision. Independence gives mass $1/18^2$ to every pair. All four normalized coefficients, including age, follow the displacement map, so $u_{ij}(100)=y_j$. Cross-patient matches can pair different ages, and extrapolation can take coefficients outside the observed range. These paths describe distributions of objective coefficients. Recorded ages are constant within patients.

To see why the order of averaging matters, let $\bar u_i(t)=\sum_j(\gamma_{ij}/p_i)u_{ij}(t)$. By linearity of the cost of the baseline plan and concavity of the optimal value $z$, the conditional Jensen gap is
\[
 \sum_j\frac{\gamma_{ij}}{p_i}\lspo(x_i,u_{ij}(t))-\lspo(x_i,\bar u_i(t))
 = z(\bar u_i(t))-\sum_j\frac{\gamma_{ij}}{p_i}z(u_{ij}(t))\ \geq\ 0 .
\]
At $t=50$ and $t=150$, the gaps averaged over patients and splits are $0.01257$ and $0.34172$ for calibrated entropy and $0.10745$ and $0.70854$ for independence. The gaps vanish for both assignment couplings because each baseline patient has a single matched endpoint. At the fitted endpoint the mean regrets for optimistic DF and calibrated entropy are $0.05119$ and $0.11899$. The DF assignment minimizes endpoint regret in every split. At $t=50$ the cost matrix changes, and calibrated entropy has lower expected regret than optimistic DF in $27$ splits.

As a sensitivity check, we fix each path's age coefficient at the baseline patient's value and vary its other three coordinates as above. Under this convention the endpoint can differ from $y_j$. Holding that convention fixed, replacing loss at the conditional mean by expected loss raises the entropy estimates from $0.00733$ to $0.03981$ at $t=50$ and from $0.15600$ to $0.48062$ at $t=150$. At $\tau=0.09$ and $c=0.12$, this rule reassesses at $t=50$ in $5$ splits and at $t=150$ in $95$, at a mean total cost of $0.20428$ per patient.

The calibration target is the mean linked SPO loss $\widehat r_s$ at $t=100$ in the calibration group. We fit the entropy-regularized endpoint coupling by matching its mean SPO loss to this target over $\varepsilon\in[0.01,10]$, using a log-spaced grid and refining a bracketing interval by root finding. We use the nearest attainable boundary loss when the target lies outside the range. The target is attained in $93$ of $100$ splits. The $10$th, $50$th, and $90$th percentiles of the fitted weights are $0.1055$, $0.3400$, and $0.6055$, respectively. The corresponding percentiles of $\widehat r_s$ are $0.0606$, $0.1320$, and $0.2108$.

\paragraph{Varying the charge with schedules held fixed.}
For a given set of tolerances, write $p_m$ for the fraction of splits in which rule $m$ reassesses and $L_m$ for its mean realized regret over both checkpoints. Its mean total cost at charge $c$ is $C_m(c)=L_m+c\,p_m$. The checkpoint estimates are common to all charge scenarios. Table~\ref{tab:reassessment_tolerance_sensitivity} gives these estimates and the two quantities that determine total cost. The first tolerance row uses the split-specific $\tau_s$ from Table~\ref{tab:synthetic_reassessment}. The other rows use the tracked endpoint calibration regret $\widehat r_s$ or a fixed tolerance shared by all splits.

\begin{table}[H]
\centering
\small
\begin{tabular}{@{}lrrrrr@{}}
\hline
 & Tracking & Calibrated entropy & Optimistic DF & $W_2$ & Independent \\
\hline
\multicolumn{6}{c}{Estimated mean baseline-plan regret at each checkpoint} \\
$t=50$ & 0.0834 & 0.0133 & 0.0079 & 0.0258 & 0.1320 \\
$t=150$ & 0.1229 & 0.3762 & 0.1411 & 0.2448 & 1.3587 \\
\hline
\multicolumn{6}{c}{Mean realized regret $L_m$ by tolerance} \\
Table~\ref{tab:synthetic_reassessment} tolerances & 0.1426 & 0.0927 & 0.1404 & 0.1014 & 0.0904 \\
$\widehat r_s$ & 0.1202 & 0.0857 & 0.1406 & 0.1000 & 0.0781 \\
$0.09$ & 0.0859 & 0.0879 & 0.1100 & 0.0838 & 0.0824 \\
$0.10$ & 0.0919 & 0.0879 & 0.1146 & 0.0868 & 0.0815 \\
$0.105$ & 0.0972 & 0.0879 & 0.1206 & 0.0884 & 0.0818 \\
$0.13$ & 0.1186 & 0.0879 & 0.1332 & 0.0933 & 0.0820 \\
$0.16$ & 0.1435 & 0.0901 & 0.1563 & 0.1011 & 0.0830 \\
\hline
\multicolumn{6}{c}{Splits with reassessment (\%) by tolerance} \\
Table~\ref{tab:synthetic_reassessment} tolerances & 48 & 93 & 58 & 87 & 100 \\
$\widehat r_s$ & 47 & 98 & 47 & 79 & 100 \\
$0.09$ & 73 & 96 & 77 & 99 & 100 \\
$0.10$ & 67 & 96 & 72 & 94 & 100 \\
$0.105$ & 64 & 96 & 66 & 92 & 100 \\
$0.13$ & 47 & 96 & 54 & 87 & 100 \\
$0.16$ & 31 & 94 & 35 & 79 & 100 \\
\hline
\end{tabular}
\caption{Baseline-plan regret at each checkpoint, realized regret under the implemented plans, and reassessment frequency. All estimates use the same $100$ patient splits.}
\label{tab:reassessment_tolerance_sensitivity}
\end{table}

Under the schedules in Table~\ref{tab:synthetic_reassessment}, calibrated entropy incurs $0.04776$ less realized regret per patient than optimistic DF and reassesses in an additional $35\%$ of splits. Their total costs are equal at a charge of approximately $0.13646$. For calibrated entropy and never reassessing, the corresponding break-even charge is $0.12224$. Table~\ref{tab:reassessment_charge_sensitivity} shows how this latter comparison changes with the tolerance. For example, at fixed tolerance $\tau=0.09$ and charge $c=1.5\tau=0.135$, entropy costs $0.21746$ per patient and never reassessing costs $0.20636$.

\begin{table}[ht]
\centering
\small
\setlength{\tabcolsep}{3pt}
\begin{tabular}{@{}lrrrrrrr@{}}
\hline
 & \multicolumn{6}{c}{Reassessment charge $c$} & \\
\cline{2-7}
Tolerance & $0.045$ & $0.0675$ & $0.09$ & $0.12$ & $0.135$ & $0.16$ & Break-even $c$ \\
\hline
Table~\ref{tab:synthetic_reassessment} tolerances & 0.0718 & 0.0509 & 0.0300 & 0.0021 & $-0.0119$ & $-0.0351$ & 0.1222 \\
$\widehat r_s$ & 0.0766 & 0.0545 & 0.0325 & 0.0031 & $-0.0116$ & $-0.0361$ & 0.1232 \\
$0.09$ & 0.0753 & 0.0537 & 0.0321 & 0.0033 & $-0.0111$ & $-0.0351$ & 0.1234 \\
$0.10$ & 0.0753 & 0.0537 & 0.0321 & 0.0033 & $-0.0111$ & $-0.0351$ & 0.1234 \\
$0.105$ & 0.0753 & 0.0537 & 0.0321 & 0.0033 & $-0.0111$ & $-0.0351$ & 0.1234 \\
$0.13$ & 0.0753 & 0.0537 & 0.0321 & 0.0033 & $-0.0111$ & $-0.0351$ & 0.1234 \\
$0.16$ & 0.0740 & 0.0529 & 0.0317 & 0.0035 & $-0.0106$ & $-0.0341$ & 0.1237 \\
\hline
\end{tabular}
\caption{Mean savings per patient from calibrated entropy relative to never reassessing. Each row holds its schedules fixed while the charge varies. Positive entries favor entropy.}
\label{tab:reassessment_charge_sensitivity}
\end{table}

\paragraph{Allowing the tolerance to vary with the charge.}
Table~\ref{tab:reassessment_retuned} gives a separate cost scenario in each column. At each charge, the split-specific tolerance $\tau_s=c/k_s$ is based on the calibration patients and is common to all threshold rules on the held-out patients. The column for $c=0.0675$ reproduces Table~\ref{tab:synthetic_reassessment}. At $c=0.09$, independence costs $0.18209$ per patient and calibrated entropy costs $0.18249$. At $c=0.12$, never reassessing costs $0.20636$ and entropy costs $0.20748$. Both the reassessment frequencies and their unit charge can change across these columns.

\begin{table}[ht]
\centering
\small
\begin{tabular}{@{}lrrrrrr@{}}
\hline
 & \multicolumn{6}{c}{Reassessment charge $c$} \\
\cline{2-7}
 & $0.045$ & $0.0675$ & $0.09$ & $0.105$ & $0.12$ & $0.135$ \\
\hline
Median $k_s$ & 0.755 & 0.730 & 0.585 & 0.565 & 0.490 & 0.410 \\
Median $\tau_s$ & 0.0596 & 0.0925 & 0.1541 & 0.1859 & 0.2449 & 0.3293 \\
\hline
Tracked reference & 0.14781 & 0.17497 & 0.19142 & 0.19976 & 0.20432 & 0.20269 \\
Calibrated entropy & 0.13426 & 0.15545 & 0.18249 & 0.19543 & 0.20748 & 0.21671 \\
Optimistic DF & 0.15141 & 0.17959 & 0.19789 & 0.20488 & 0.20790 & 0.21178 \\
$W_2$ matching & 0.13465 & 0.16016 & 0.18573 & 0.19961 & 0.20593 & 0.21211 \\
Independent & 0.13457 & 0.15794 & 0.18209 & 0.19799 & 0.21310 & 0.22290 \\
\hline
Always at $t=50$ & 0.12541 & 0.14791 & 0.17041 & 0.18541 & 0.20041 & 0.21541 \\
Always at $t=150$ & 0.12842 & 0.15092 & 0.17342 & 0.18842 & 0.20342 & 0.21842 \\
Never reassess & 0.20636 & 0.20636 & 0.20636 & 0.20636 & 0.20636 & 0.20636 \\
\hline
\end{tabular}
\caption{Mean total cost per patient under fee-specific tolerances. Each column uses one charge for every rule and the same $100$ held-out groups.}
\label{tab:reassessment_retuned}
\end{table}

\section{Synthetic Data Designs and Additional Results}\label{append:synthetic_designs}

\subsection{Estimation against known population values}\label{sec:synthetic_estimation_design}

\paragraph{Zero-divergence models.}
Let $S=\mathrm{conv}\{e_1,\ldots,e_M\}\subseteq\bbR^d$, and let $J$ be uniform on $\{1,\ldots,K_\mu\}$, where $K_\mu\le M<d$. We draw independent coordinates $U_i$ uniform on $[-0.04,0.04]$ for $i\le M$ and on $[-1/\sqrt{d-M},1/\sqrt{d-M}]$ for $i>M$. A profile from either marginal is $U-0.35e_J$. Each cluster lies in the region of $e_J$, and the nuisance variation has bounded Euclidean norm as $d$ changes. The two populations coincide, so $\wdfo=W_2^2=0$. Finite samples can have different decision shares. In this construction, an off-diagonal match between decisions has a loss in $[0.27,0.43]$. Therefore the plug-in DF divergence is bounded above and below by fixed multiples of the total variation between empirical decision shares. At least the unmatched mass must be assigned off the diagonal, and there is an assignment that moves exactly that mass.

In the dimension sweep, we take $M=K_\mu=2$ and $d\in\{4,8,16,32\}$. In the decision-count sweep, we fix $(d,M)=(24,16)$ and take $K_\mu\in\{2,4,8,16\}$. For each setting and each $n\in\{64,128,256,512,1024\}$, we draw $100$ independent pairs of marginal samples of size $n$. We compute the plug-in DF divergence by a reduced transportation LP and the plug-in $W_2^2$ by an exact linear assignment with the squared Euclidean ground cost.

\paragraph{Positive-divergence model.}
We use $M=2$ and $d=8$. Let the first two coordinates of $X$ be $(A+V/2,A-V/2)$, where $A$ and $V$ are independent uniforms on $[-0.1,0.1]$ and $[-1,1]$. The remaining coordinates are independent uniforms on $[-1/\sqrt{6},1/\sqrt{6}]$. The distribution of $Y$ is the translation of this law by $e_1-e_2$, so $W_2^2=\|e_1-e_2\|_2^2=2$. Half the decision population uses plan $e_1$, and every profile in the evaluation population prefers plan $e_2$. The loss of using $e_1$ in the evaluation population is uniform on $[1,3]$. An optimal coupling assigns its $e_1$ mass to the lowest half of those losses. This gives $\wdfo=\int_0^{1/2}(1+2u)\,du=0.75$.

\begin{table}[ht]
\centering
\small
\begin{tabular}{lrrrrr}
\hline
Estimator & $n$ & MAE & Mean error (MC SE) & Error quantiles $(10\%,90\%)$ & Under (\%) \\
\hline
Optimistic DF & 64 & 0.1145 & 0.0088 (0.0145) & $[-0.1755,0.1871]$ & 48 \\
 & 128 & 0.0795 & 0.0098 (0.0099) & $[-0.1059,0.1346]$ & 51 \\
 & 256 & 0.0561 & -0.0003 (0.0070) & $[-0.1047,0.0752]$ & 48 \\
 & 512 & 0.0371 & -0.0005 (0.0045) & $[-0.0553,0.0546]$ & 47 \\
 & 1024 & 0.0254 & 0.0028 (0.0032) & $[-0.0376,0.0482]$ & 49 \\
\hline
Quadratic $W_2^2$ & 64 & 0.2941 & 0.2743 (0.0228) & $[0.0164,0.5525]$ & 9 \\
 & 128 & 0.2294 & 0.2210 (0.0150) & $[0.0550,0.4303]$ & 8 \\
 & 256 & 0.1771 & 0.1704 (0.0110) & $[0.0103,0.3045]$ & 8 \\
 & 512 & 0.1323 & 0.1308 (0.0075) & $[0.0304,0.2328]$ & 2 \\
 & 1024 & 0.1136 & 0.1133 (0.0053) & $[0.0438,0.1830]$ & 1 \\
\hline
\end{tabular}
\caption{Positive-divergence experiment. Each row uses $100$ independent trials. Errors are the estimates minus the exact population values, MC SE is the Monte Carlo standard error of the mean error, and the last column gives the share of negative errors.}
\label{tab:synthetic_positive_error}
\end{table}

Table~\ref{tab:synthetic_positive_error} shows that both positive and negative DF estimation errors occur around the positive target. The absolute values of the mean DF errors are within one Monte Carlo standard error in the settings shown. The zero-divergence designs in Figure~\ref{fig:estimation} directly exhibit upward bias because every positive estimate there is an overestimate.

\subsection{Direct and reduced empirical computation}\label{sec:synthetic_computation}

We draw benchmark instances from the decision-count design with $d=24$, varying the common marginal sample size and the number of occupied decision regions. Both implementations use the SciPy interface to the HiGHS dual simplex solver with primal and dual feasibility tolerances of $10^{-8}$. Construction time includes oracle evaluation, cost matrix construction, and assembly of the marginal constraints. We time the solve separately and measure peak process memory during construction and solution. Each solve runs in a fresh process. Table~\ref{tab:df_benchmark} reports medians over three instances per setting on a macOS ARM host with numerical-library thread counts set to one.

\begin{table}[ht]
\centering
\begin{tabular}{rrrrrrrr}
\hline
 & & \multicolumn{2}{c}{Construction (ms)} & \multicolumn{2}{c}{Solve (s)} & \multicolumn{2}{c}{Peak RSS (MiB)} \\
\cline{3-4}\cline{5-6}\cline{7-8}
$n$ & $\hat K_\mu$ & Direct & Reduced & Direct & Reduced & Direct & Reduced \\
\hline
128 & 4 & 1.03 & 0.52 & 0.037 & 0.002 & 148.7 & 130.2 \\
256 & 4 & 2.84 & 0.55 & 0.239 & 0.003 & 203.2 & 130.8 \\
512 & 2 & 10.73 & 0.55 & 3.232 & 0.005 & 461.3 & 130.7 \\
512 & 4 & 10.83 & 0.60 & 2.406 & 0.008 & 419.1 & 132.4 \\
512 & 8 & 10.75 & 0.67 & 1.745 & 0.012 & 423.7 & 134.4 \\
512 & 16 & 10.67 & 0.79 & 1.080 & 0.019 & 453.9 & 141.7 \\
\hline
\end{tabular}
\caption{Empirical DF linear programs on identical inputs. The direct LP has $n^2$ variables, and the reduced LP has $\hat K_\mu n$. Memory is the peak resident size of the solver process including its libraries.}
\label{tab:df_benchmark}
\end{table}

 \end{APPENDICES}

\end{document}